\documentclass{amsart}

\usepackage{amssymb,graphicx,color,rotating}
\usepackage{pinlabel}

\usepackage[all]{xy}
\SelectTips{cm}{}

\usepackage[dvipsnames]{xcolor}

\usepackage{hyperref}
\hypersetup{colorlinks=true,citecolor=brown,linkcolor=brown,urlcolor=black,filecolor=black}

\numberwithin{equation}{section}
\allowdisplaybreaks

\theoremstyle{plain}
\newtheorem{theorem}{Theorem}[section]
\newtheorem{lemma}[theorem]{Lemma}
\newtheorem{corollary}[theorem]{Corollary}
\newtheorem{proposition}[theorem]{Proposition}

\theoremstyle{definition}
\newtheorem{definition}[theorem]{Definition}
\newtheorem{example}[theorem]{Example}

\theoremstyle{remark}
\newtheorem{remark}[theorem]{Remark}

\newcommand\nc{\newcommand}
\nc\rnc{\renewcommand}

\nc{\GMb}{\color[rgb]{.7,0,.7}}
\nc{\KHb}{\color[rgb]{0,0,.7}}
\nc{\KHe}{\normalcolor{}}
\nc{\GMe}{\normalcolor{}}
\nc\GMnote[1]{\marginpar{\GMb \tiny #1 \GMe}}
\nc\KHnote[1]{\marginpar{\KHb \tiny #1 \KHe}}
\nc\KHn[1]{\marginpar{\KHb \tiny #1 \KHe}}
\nc{\GMcut}[1]{\marginpar{\GMb \textbf{\footnotesize CUT}: \Tiny #1  \GMe}}
\nc{\KHcut}[1]{\marginpar{\KHb \textbf{\footnotesize CUT}: \Tiny #1  \KHe}}
\nc\KH[1]{{\KHb #1}}
\nc\GM[1]{{\GMb #1}}
\nc\blue[1]{{\color{blue}#1}}
\nc\darkblue[1]{{\textcolor[rgb]{0,0,.5}{#1}}}

\nc\id{\operatorname{id}}
\nc\ad{{\operatorname{ad}}}
\nc\Ad{{\operatorname{Ad}}}
\nc\ab{{\operatorname{ab}}}
\nc\End{\operatorname{End}}
\nc\Aut{\operatorname{Aut}}
\nc\Out{\operatorname{Out}}
\nc\IAut{\operatorname{IAut}}
\nc\opn{\operatorname}
\nc\Hom{\operatorname{Hom}}
\nc\inn{\operatorname{inn}}
\nc\Ker{\operatorname{Ker}}
\nc\word{\operatorname{word}}
\nc\col{\operatorname{col}}
\nc\Br{\mathrm{Br}}
\nc\Lie{\mathrm{Lie}}
\nc\Der{\mathrm{Der}}
\nc\rank{\opn{rank}}
\nc\Mat{\opn{Mat}}
\nc\GL{\opn{GL}}
\nc\GE{\opn{GE}}
\nc\Mag{\opn{Mag}}
\nc\Wh{\opn{Wh}}
\nc\proj{\mathrm{proj}}
\nc\gr{\mathrm{gr}}
\nc\grb{\gr_\bu}
\nc\plim{\varprojlim}
\nc\ilim{\varinjlim}

\nc\modN {{\mathbb N}}
\nc\modQ {{\mathbb Q}}
\nc\R{{\mathbb R}}
\nc\modZ {{\mathbb Z}}
\nc\Q{\mathbb{Q}}
\nc\Vect{\mathbf{Vect}}
\nc\Grp{\mathbf{Grp}}
\nc\NS{\mathbf{NS}}
\nc\eNs{\mathbf{eNs}}
\nc\egL{\mathbf{egL}}
\nc\Z{\modZ }
\nc\ZF{\modZ [F]}
\nc\ZH{\modZ [H]}

\nc\modR {{\mathcal R}}
\nc\modL {{\mathcal L}}
\nc\modC {{\mathcal C}}
\nc\modE {{\mathcal E}}
\nc\modV {{\mathcal V}}
\nc\modT {\mathcal{T}}
\nc\modH {\mathcal{H}}
\nc\modM {\mathcal{M}}
\nc\modI {\mathcal{I}}
\nc\modB {\mathcal{B}}
\nc\modD {\mathcal{D}}
\nc\bC{\bar{C}}
\nc\modG {\mathcal{G}}
\nc\tH{\tilde{\modH }}
\nc\modF {\mathcal{F}}
\nc\modP {\mathcal{P}}
\nc\modS {\mathcal{S}}
\nc\G{\modG}
\nc\T{\modT}
\rnc\H{\modH}

\nc\g{{\mathfrak g}}

\nc\xto[1]{{\overset{#1}{\longrightarrow}}}
\nc\yto[1]{{\underset{#1}{\longrightarrow}}}
\nc\xyto[2]{{\overset{#1}{\underset{#2}{\longrightarrow}}}}
\nc\simeqto{\overset{\simeq}{\rightarrow }}
\nc\trl{\triangleleft}
\nc\trll{\,\triangleleft\,}
\nc\trr{\triangleright}
\nc\trrr{\,\triangleright\,}
\nc\ct{\overset{\cong}{\to}}
\nc\onto\twoheadrightarrow
\nc\projto{\underset{\text{proj}}{\longrightarrow}}
\nc\ho{\mathbin{\hat\otimes}}
\nc\la{\langle}
\nc\ra{\rangle}
\nc\lala{\langle\!\langle}
\nc\rara{\rangle\!\rangle}
\nc\mt{\mapsto}
\nc\rt{\rtimes}
\nc\bu{\bullet} \nc\Lbu{L_\bu}

\nc\fig[1]{Figure~\ref{#1}}
\nc\FI[2]{\begin{figure}  \begin{center}\input{#1.pstex_t}\end{center} \caption{#2}  \label{#1}  \end{figure}}
\nc\FIGURE[2]{\begin{figure}  \boxed{\tt fig: #1}  \caption{#2}  \label{fig:#1}  \end{figure}}

\nc\ul{\underline}
\nc\bA{\bar{\A}}	\nc\bAk{\bA_k}
\nc\bK{\bar K}		\nc\bKb{\bK_\bu}	\nc\bKbu{\bK_\bu}
\nc\bG{\bar G}		\nc\bGb{\bG_\bu}
\nc\ti{\tilde}
\nc\wh{\widehat}
\nc\wt{\widetilde}
\nc\ol{\overline}
\nc\ch{\check}

\nc\np{\newpage}
\nc\xym{\xymatrix}
\nc\lto{\longrightarrow}
\nc\Alpha{\mathsf{A}}
\nc\A{\Alpha}
\nc\AAlpha{\mathbb{A}}
\rnc\AA{\AAlpha}
\nc\al{\alpha}
\nc\be{\beta}
\nc{\laitrap}{\textrm{\reflectbox{$\partial$}}}
\newcommand{\double}[2]{\left\{\!\!\left\{#1,#2\right\}\!\!\right\}}

\nc\ot{\otimes}

\begin{document}

\title[The Johnson--Morita theory for the handlebody group]{The Johnson--Morita theory \\ for the handlebody group}

\author{Kazuo Habiro}
\address{Graduate School of Mathematical Sciences, University of Tokyo, 
3-8-1 Komaba, Meguro-ku, Tokyo 153-8914, Japan}
\email{habiro@ms.u-tokyo.ac.jp}

\author{Gw\'ena\"el Massuyeau}
\address{Université Bourgogne Europe, CNRS, 
IMB (UMR 5584), 21000 Dijon, France}
\email{gwenael.massuyeau@ube.fr}

\date{\today}


\begin{abstract}
The Johnson--Morita theory is an algebraic approach 
to the mapping class group of a surface, 
in which one considers its action on the successive nilpotent quotients 
of the fundamental group of the surface. 
In this paper, 
we develop an analogue of this theory 
for the handlebody group, 
i.e$.$ the mapping class group of a 
$3$-dimensional handlebody.
Thus, we obtain a  filtration on the handlebody group,
 prove that its associated graded 
embeds  into a Lie algebra of ``special derivations'',
and give an explicit diagrammatic description 
of this graded Lie algebra 
in terms of ``oriented trees with beads''.
Our new diagrammatic method 
reveals part of the richness of the algebraic structure of the handlebody group, 
which lies mainly in the subgroup generated by Dehn twists along meridians: the so-called ``twist group''.
As an application, we obtain that each term 
of the associated graded of the lower central series 
of the twist group is infinitely generated.
\end{abstract}

\maketitle

\setcounter{tocdepth}{1}
\tableofcontents

\section{Introduction}

The \emph{Johnson--Morita theory} studies the mapping class group 
of a surface
by considering its action 
on the lower central series of the fundamental group of the surface.
In this paper, we introduce an analogue of this theory for the handlebody group.

\subsection{The Johnson--Morita theory for the mapping class group}
\label{subsec:usual_Johnson}

We first briefly outline the original Johnson--Morita theory.
Let $\Sigma$ be a compact, connected, oriented surface with one boundary component.
Let $\modM := \modM(\Sigma,\partial \Sigma)$ 
be the mapping class group of $\Sigma$ relative to $\partial \Sigma$.

By a classical result of Dehn and Nielsen, the canonical action of 
$\modM$ on $\pi:=\pi_1(\Sigma,\star)$,
with base point $\star \in \partial \Sigma$, is faithful.
Thus, $\modM$ embeds into $\Aut(\pi)$.
Let
$$
\pi=\Gamma_1 \pi \geq \Gamma_2 \pi \geq  \cdots \geq  
\Gamma_k \pi \geq  \cdots
$$
be the lower central series  of $\pi$.
Its associated graded Lie algebra is isomorphic to
 the free Lie algebra  $\Lie(H)$ on $H:= H_1(\Sigma;\Z)$.

The \emph{Johnson filtration} of $\mathcal M$
is the decreasing sequence of subgroups
$$
\modM=\modM_0 \geq \modM_1 
\geq  \cdots \geq   \modM_k \geq  \cdots
$$
defined by
\begin{eqnarray}
\label{eq:JJ} \modM_k &:=&
\ker \big( \modM \longrightarrow \Aut(\pi/\Gamma_{k+1} \pi)\big) \\
 \label{eq:JJJ}  &=& \big\{ f \in \modM\mid  f(x)x^{-1} \in \Gamma_{k+i} \pi 
 \text{ for all $x\in\Gamma_i\pi, i\geq 1$}\big\}.
\end{eqnarray}
Johnson studied its first few terms
(see \cite{Johnson_survey}).
The first term $\modI := \modM_1$, known 
as the \emph{Torelli group}, is the subgroup of $\modM$
acting trivially on $H$.
Morita \cite{Morita_abelian} studied the Johnson filtration systematically.
It has trivial intersection
\begin{equation} \label{eq:intersection}
\bigcap_{k\geq 1} \modM_k = \{1\},
\end{equation} 
and it is an \emph{N-series} 
or is \emph{strongly central}, i.e., 
\begin{equation} \label{eq:strongly_central}
\big[\modM_k,\modM_l\big] \subset \modM_{k+l} \quad \hbox{for all $k,l\geq 1$}.
\end{equation}
An important problem is to compute the associated graded Lie algebra
$$
\overline{\modM}_+ := \bigoplus_{k\geq 1} \modM_k/ \modM_{k+1}.
$$
The conjugation of $\modM$ on
$\modI$ induces an action
of the symplectic group $\hbox{Sp}(H) \simeq \modM/\modI$ on $\overline{\modM}_+$.
Here the symplectic form $\omega:H \times H \to \Z$ is the homology intersection form of~$\Sigma$.

By works of Johnson, Morita and others, 
the structure of $\overline{\modM}_k=\modM_k/\modM_{k+1}$ is well understood for $k=1,2$,
and so is its rationalization $\overline{\modM}_k\otimes \Q$ for some higher values of $k$ 
in relation with Hain's computation \cite{Hain_Torelli}
of the Malcev Lie algebra of $\modI$
(see \cite{Morita_survey}). 
The general procedure to determine 
the abelian group $\overline{\modM}_k$ for $k\ge1$ is as follows.
For any $f\in \modM_k$ define a map $\tau_k(f): \Lie(H) \to \Lie(H)$ 
by
\begin{equation} \label{eq:formula}
\tau_k(f)([x]_i)= \big[f(x) x^{-1}\big]_{i+k}
\quad(x\in \Gamma_i\pi, i\ge1).
\end{equation}
(Here, 
$[y]_j\in \Gamma_j\pi/\Gamma_{j+1}\pi=\Lie_j(H)$ denotes
the class  of $y\in \Gamma_j\pi$.)
The map $\tau_k(f)$ vanishes if and only if $f\in \modM_{k+1}$.
Furthermore, $\tau_k(f)$ is a derivation of degree $k$
and, since $f$ fixes $\zeta:=[\partial \Sigma] \in \pi$,
the map $\tau_k(f)$ vanishes on $[\zeta]_2\in\Lie_2(H)\simeq\Lambda^2H$,
which is dual to the intersection form $\omega$ on $H$.
Thus, for every $k\geq 1$,
we get an injective homomorphism
$$
\overline{\tau}_k: \overline{\modM}_k \longrightarrow 
\Der_k^\omega\big(\Lie(H)\big)
:=\big\{d\in\Der_k\big(\Lie(H)\big):
d([\zeta]_2)=0\big\},$$
which is called the \emph{$k$-th Johnson homomorphism}. 
All these homomorphisms $\tau_k$, for $k\ge1$, form
an injective $\hbox{Sp}(H)$-equivariant Lie algebra homomorphism
\begin{equation} \label{eq:tau_+}
\overline{\tau}_+:=(\overline{\tau}_k)_{k\geq 1}: \overline{\modM}_+ \longrightarrow \Der_+^\omega\big(\Lie(H)\big)
:=\bigoplus_{k\ge1}\Der_k^\omega\big(\Lie(H)\big)
\end{equation} 
of $\overline{\modM}_+$
into the Lie algebra of \emph{symplectic derivations} of positive degree.\\

\noindent
\begin{minipage}{0.7\textwidth}
\begin{example}
Consider the first Johnson homomorphism 
$$\tau_1:\modM_1=\modI\longrightarrow \Der_1^\omega\big(\Lie(H)\big)\simeq\Lambda^3H.$$
The group $\modI$ is generated by opposite Dehn twists $T_c T_d^{-1}$
along pairs of simple closed curves $c$ and $d$ cobounding a subsurface $S$ of $\Sigma$ \cite{Johnson_first}.
For such a pair $(c,d)$, we have 
\begin{equation} \label{eq:tau_1}
\tau_1\big(T_c T_d^{-1}\big) = \pm\,  \omega_S \wedge [c] \in \Lambda^3 H,
\end{equation}
where $\omega_S \in \Lambda^2 H$ is dual to the intersection form of $S$
(in the example on the right of genus $1$, $\omega_S = a \wedge b$).
This formula is crucial in the proof of the surjectivity of $\tau_1$ \cite{Johnson_abelian},
which implies that $\overline{\modM}_1 \simeq \Lambda^3H$.\\
\end{example}
\end{minipage}
\begin{minipage}{0.3\textwidth}
\labellist
\small\hair 2pt
 \pinlabel {\textcolor{blue}{$c$}} [t] at 42 285
 \pinlabel {\textcolor{blue}{$d$}} [t] at 180 286
 \pinlabel {$\vdots$} at 111 217
 \pinlabel {\textcolor{brown}{$b$}} [r] at 52 401
 \pinlabel {\textcolor{brown}{$a$}} [r] at 80 473
\endlabellist
\centering
\includegraphics[scale=0.3]{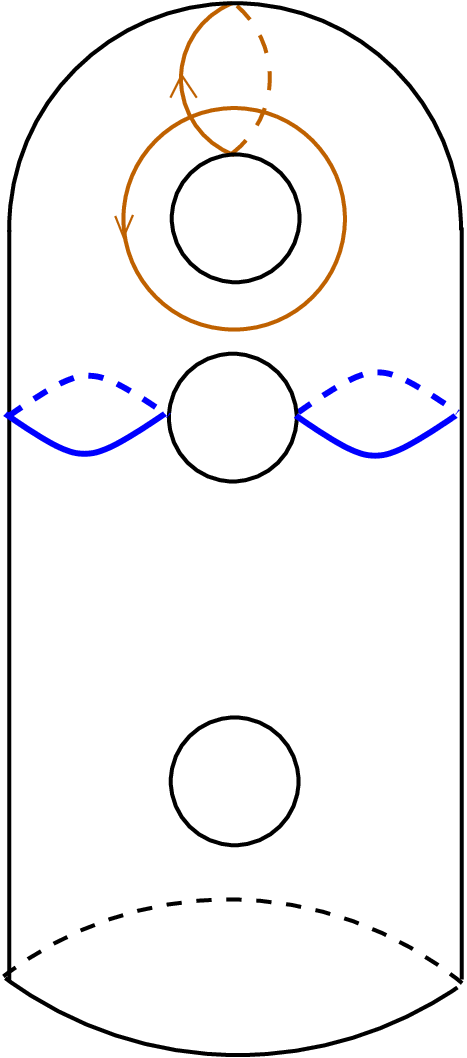}
\end{minipage}

The computation of $\overline{\modM}_2$ was carried out by Morita \cite{Morita_Casson}.
Yet, in degree $k\geq 2$, it is much easier to work with rational coefficients. 
Since $\Der_+^\omega\big(\Lie(H)\big)$ is torsion-free,
it embeds into  
$\Der_+^\omega\big(\Lie(H)\big)\otimes \Q =\Der_+^\omega\big(\Lie(H^\Q)\big)$,
the Lie $\Q$-algebra of positive-degree symplectic derivations
of the free Lie $\Q$-algebra $\Lie(H^\Q)$ on $H^\Q:=H_1(\Sigma;\Q)$.
The latter has a diagrammatic description, 
which is implicit in \cite{Kontsevich} 
and appears e.g$.$ in \cite{HP} and \cite{GL05}:
there is an isomorphism of graded Lie $\Q$-algebras
\begin{equation} \label{eq:Kontsevich}
\Der_+^\omega\big(\Lie(H^\Q)\big) \simeq \mathcal{D}(H^\Q)   , 
\end{equation}
where $\mathcal{D}(H^\Q)$ is the $\Q$-vector space 
generated by ``oriented trivalent trees'' 
with  leaves  colored by  $H^\Q$.
The Lie bracket of $\mathcal D(H^\Q)$ is defined by  gluing leaves to leaves 
using the pairing~$\omega$. For instance, the generators of $\mathcal{D}(H^\Q)$ are\\[0.1cm]
$$
\labellist
\small\hair 2pt
 \pinlabel {$\mathsf{Y}_{a,b,c}:=$} [r] at -20 114
  \pinlabel {$\mathsf{H}_{d,e,f,g}:=$} [r] at 810 114
 \pinlabel {$a$} [tr] at 2 25
 \pinlabel {$c$} [br] at 2 221
 \pinlabel {$b,$} [l] at 260 114
 \pinlabel {$e$} [tr] at 830 25
 \pinlabel {$d$} [br] at 830 223
  \pinlabel {$g$} [lb] at 1210 223
 \pinlabel {$f$} [tl] at 1210 30
  \pinlabel {(with $a,b,\dots,g \in H^\Q$)} [l] at 1410 5
\endlabellist
\centering
\includegraphics[scale=0.12]{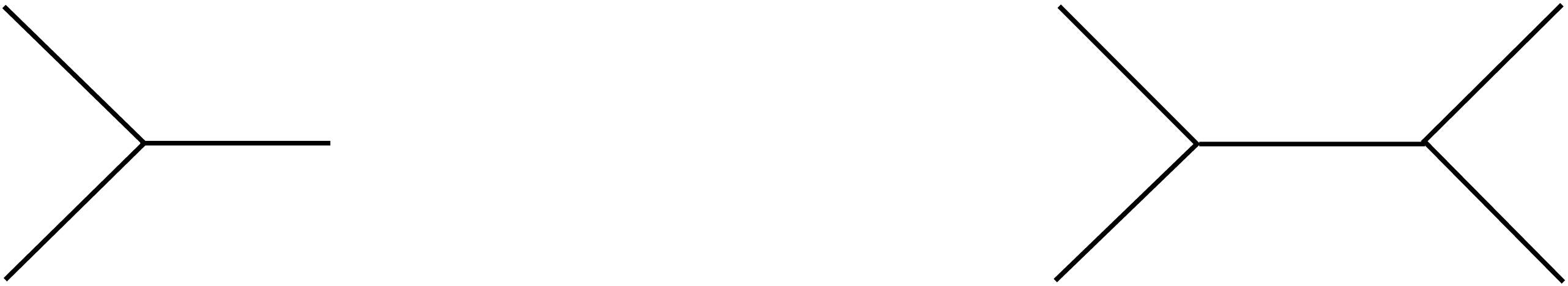} 
\qquad \qquad \qquad 
$$
in degrees $1$ and $2$, respectively,
and the Lie bracket in degree $1+1$ is given by
\begin{eqnarray}
\label{eq:trees_bracket} \quad \big[ \mathsf{Y}_{a,b,c}, \mathsf{Y}_{u,v,w} \big] &=&
\phantom{+}\omega(a,u)\, \mathsf{H}_{b,c,v,w} 
+ \omega(b,u)\, \mathsf{H}_{c,a,v,w} + \omega(c,u)\, \mathsf{H}_{a,b,v,w}\\
\notag && +\omega(a,v)\, \mathsf{H}_{b,c,w,u} 
+ \omega(b,v)\, \mathsf{H}_{c,a,w,u} + \omega(c,v)\, \mathsf{H}_{a,b,w,u}\\
\notag && +\omega(a,w)\, \mathsf{H}_{b,c,u,v} 
+ \omega(b,w)\, \mathsf{H}_{c,a,u,v} + \omega(c,w)\, \mathsf{H}_{a,b,u,v}.
\end{eqnarray}
This Lie bracket  in $\mathcal{D}(H^\Q)$ and
formula  \eqref{eq:tau_+} in degree $1$ allow
for explicit computations of $\tau_k$ in every degree $k>1$
since,  according to Hain \cite{Hain_Torelli}, the space
$(\Gamma_k \modI/ \Gamma_{k+1} \modI)\otimes \Q$
surjects  onto $\overline{\modM}_k\otimes \Q$.

To conclude this quick overview of the Johnson--Morita theory 
for the mapping class group, 
let us recall
that all the Johnson homomorphisms~$\tau_k$
for  $k\geq 1$ unify into a single map
\begin{equation}
\label{eq:varrho_theta}
\varrho^\theta: \modI \longrightarrow \widehat{\Der}_+^\omega\big(\Lie(H^\Q)\big),
\end{equation}
where the hat $\,\widehat{\quad}\,$ denotes the degree-completion.
The map $\varrho^\theta$ induces  the rationalization of \eqref{eq:tau_+} at the graded level
and we can regard $\varrho^\theta$ as an ``infinitesimal'' version 
of the canonical action of $\modI$ on $\pi$.
To define  $\varrho^\theta$  
we need a ``symplectic expansion''~$\theta$ of~$\pi$,
which manifests formality
of the free group~$\pi$.
Indeed $\theta$ identifies
the Malcev Lie algebra of $\pi$
with the degree-completed free Lie algebra on $H^\Q$. Although $\varrho^\theta$ heavily depends on the choice of $\theta$,
it enjoys several properties.
It is a group embedding
if the target is endowed with the Baker--Campbell--Hausdorff (BCH) product
associated to the Lie bracket and,
besides the property of determining all the Johnson homomorphisms,
the map $\varrho^\theta$ gives
(for an appropriate choice of $\theta$)
the tree-level of the representation of the Torelli group
that is induced by the universal finite-type invariant of $3$-manifolds~\cite{Mas12}.

\subsection{Johnson--Morita theory for extended N-series}

In this paper, we develop an analogue of the Johnson--Morita
theory for the \emph{handlebody group}, i.e$.$
the mapping class group of a $3$-dimensional handlebody $V$.

The possibility of such a theory has been mentioned in
 \cite[Ex$.$ 10.9]{HM18} as an instance of a general framework of extended N-series and extended graded Lie algebras, see Section \ref{sec:general_theory}.
An \emph{extended N-series} 
$K_*=(K_m)_{m\geq 0}$
is a descending filtration
$$
K_0 \geq K_1 \geq \cdots \geq K_m \geq \cdots
$$
of a group $K_0$ satisfying
$[K_m, K_n] \subset K_{m+n}$ for all $m,n \geq 0$.
An \emph{action} of a group $\modG$ on the extended N-series $K_*$ is an action of $\G$ on $K_0$ such that $g(K_m)=K_m$ for all $g\in \G$ and $m\ge0$.
Similarly to \eqref{eq:JJJ}, 
the \emph{Johnson filtration} 
$$\G=\G_0\ge \G_1\ge\G_2\ge\cdots$$
is defined by
$$
\modG_m := \bigcap_{j\geq 0} \modG_m^j ,
\quad \hbox{where } \modG_m^j := \big\{g \in \modG :  g(x)x^{-1} \in K_{m+j}
 \text{ for all }x \in K_j\big\}
$$
but, in contrast with \eqref{eq:JJ}, 
we do not necessarily have  $\modG_m = \modG_m^1$.
As in the case of the mapping class group, 
this generalized Johnson  filtration is an N-series 
and so has an associated graded $\overline{\modG}_+$.

To the extended N-series $K_*=(K_m)_{m\geq 0}$ is associated the 
\emph{extended graded Lie algebra} $\overline{K}_\bullet$,
which is the pair of the graded Lie algebra
$\overline{K}_+=\bigoplus_{m\geq 1} K_m/K_{m+1 }$ in positive degrees
and the group $\overline{K}_0 = K_0/K_1$ in degree $0$,
the latter acting on the former by graded Lie algebra automorphisms.
For $k\geq 1$,
a \emph{derivation} $d=(d_0,d_+)$ of  $\overline{K}_\bullet$ 
of degree $k$
is a pair $(d_0,d_+)$ of a degree $k$ derivation $d_+$ of the graded Lie algebra $\overline{K}_+$,
and of a $1$-cocycle
$d_0:\overline{K}_0  \to \overline{K}_k$ 
which measures
the defect of $\overline{K}_0$-equivariance of $d_+$.
Let $\Der_k(\overline{K}_\bullet)$ denote the $\Z$-module of degree $k$ derivations of~$\overline{K}_\bullet$.
Then, for every $k\geq 1$, the \emph{$k$-th Johnson homomorphism}
$$
\tau_k : \modG_k  \lto \Der_k(\overline{K}_\bullet)
$$
is defined similarly to \eqref{eq:formula}.
Specifically, for $f\in \modG_k$, the formula
$$\tau_k(f)([x]_i)= \big[f(x) x^{-1}\big]_{i+k}\quad (x\in K_i, i\geq 0),$$
defines a derivation $\tau_k^+(f)$ of $\overline{K}_+$ for $i>0$
and a $1$-cocycle 
$\tau_k^0(f):\overline{K}_0  \to \overline{K}_k$ for $i=0$.
Thus, the $\tau_k$ for $k\geq1$ form an embedding 
\begin{equation}\label{eq:tau_+_bis} 
\overline{\tau}_+: \overline{\modG}_+ 
\longrightarrow \Der_+ ( \overline{K}_\bu) 
\end{equation} 
of graded Lie algebras.
Under a certain formality assumption on the extended N-series $K_*$,
and similarly to \eqref{eq:varrho_theta},
we also obtain a group embedding
\begin{equation} \label{eq:varrho_theta_bis}
\varrho^\theta: \modG_1 \longrightarrow \widehat{\Der}_+(\overline K_\bu^\Q),
\end{equation}
which induces  the rationalization of \eqref{eq:tau_+_bis} at the graded level.

\subsection{Johnson--Morita theory for free pairs}

In Section \ref{sec:free_pairs},
we apply the above constructions  
to a \emph{free pair} $(\pi,\Alpha)$, 
a pair of a free group $\pi$ of finite rank
and a non-abelian normal subgroup $\Alpha \leq \pi$ 
with $F:=\pi/\Alpha$ free of finite rank.
A free pair $(\pi,\Alpha)$ yields an extended N-series $\Alpha_*$:
$$
\Alpha_0:=\pi \quad \hbox{and} \quad 
\Alpha_i:= \Gamma_i \Alpha \ \hbox{for $i\geq 1$}.
$$
The associated extended graded Lie algebra $\overline \Alpha_\bu$ of $\A_*$
consists of the group $\overline{\A}_0=F$ 
and the free Lie algebra $\overline{\A}_+=\Lie(\AAlpha)$ on
the abelianization $\AAlpha:= \Alpha_{\operatorname{ab}} = \Alpha/\Gamma_2 \Alpha$.
So, with a slight abuse of notation, we write
$$
\overline \Alpha_\bu =F \ltimes \Lie(\AAlpha),
$$
where the action of $F$ on $\AAlpha$
makes $\AAlpha$ 
into a free $\Z[F]$-module of finite rank. 

Let
$$\modG:= \Aut(\pi,\Alpha)=\big\{f\in\Aut(F): f(\A)=\A\big\},$$
which
acts on the extended N-series $\Alpha_*$.
Hence, there is a Johnson filtration $\modG_*$ 
and a family of Johnson homomorphism $(\tau_k)_{k\geq 1}$.
By \cite[\S 10.1]{HM18} we have
\begin{equation}\label{GkGk1}
\modG_k= \modG_k^1\quad \text{for }k\geq 1,
\end{equation}
see \eqref{GmGm1}.
Similarly, for $f \in \modG_k$, the derivation $\tau_k^+(f)$
determines the 1-cocycle $\tau_k^0(f)$, but the converse may not be true.
We also observe   that
the extended N-series $\Alpha_*$ is formal.
Hence we get in Theorem \ref{thm:varrho} an embedding \eqref{eq:varrho_theta_bis}
of $\modG_1$ into $\widehat{\Der}_+\big(F\!\ltimes\!\Lie(\AAlpha^\Q)\big)$.
 
In Section \ref{sec:first_quotients}, we 
construct a map
\begin{equation}\label{mag}
J^F: \mathcal G\longrightarrow \GL(p+q;\Z[F]),
\end{equation}
where $q=\rank F$ and $p= \rank\pi-\rank F$,
by reducing the coefficients of the Magnus representation of $\modG= \Aut(\pi,\Alpha)$ to $\Z[F]$.
The map $J^F$ factors through $\G/\G_2$
and is our main tool 
to determine the first terms of the Johnson filtration 
$$
   \modG=\modG_0 \geq \modG_1^0 \geq \modG_1 \geq \modG_2^0 \geq \cdots.
$$
In particular, $\modG_1^0/\modG_1$ is non-trivial, see Proposition \ref{r59}.

\subsection{Johnson--Morita theory for the handlebody group}

Section \ref{sec:Johnson_handlebody} gets to the heart of the subject,
namely the \emph{handlebody group} $\modH$, which is the mapping class group  of the handlebody $V$ 
relative to a disk  $D \subset \partial V$.
The comparison between the Johnson--Morita theory of the mapping class group 
and our approach for the handlebody group 
is sketched in Table \ref{tab:comparison}.

\begin{table}[]
    \centering \small 
\begin{tabular}{|c|c|c|} \hline 
    Group  &  mapping class group $\modM $  &  handlebody group $\modH$ \\ \hline
    Action  &  on the group $\pi$ & on  the pair $(\pi, \Alpha)$\\ \hline
    Filtration & $(\modM_k)_{k\geq 0}$ &  $(\modH_k)_{k\geq 0}$ \\ \hline
    0-th graded quotient & ${\modM_0}/{\modM_1}\simeq \hbox{Sp}(H)$ &
    ${\modH_0}/{\modH_1} \simeq \hbox{Aut}(F)$ \\ \hline
    1-st subgroup & $\modM_1=\modI$, the Torelli group 
    & $\modH_1=\modT$, the twist group \\ \hline
    Johnson homomorphisms & $\tau_k:\modM_k \to \Der_k^\omega\big(\Lie(H)\big) $ 
    &  $\tau_k:\modH_k \to \Der_k^\zeta\big( F\!\ltimes\!\Lie(\AAlpha) \big) $ \\ \hline
    ``Infinitesimal'' action & 
    ${\varrho^\theta: \modI \to\widehat{\Der}_+^\omega\big(\Lie(H^\Q)\big)}$
    & $\varrho^\theta: \modT  \to 
    \widehat{\Der}_+^\zeta\big(F\!\ltimes\!\Lie(\AAlpha^\Q)\big)$
    \\ \hline
    Diagram space & 
     $\mathcal{D}(H^\Q)$   &  $\modD(\AAlpha^\Q, \Q[F])$ \\ \hline
\end{tabular}\\[0.3cm]
    \caption{Comparison between the Johnson--Morita theories for the mapping class group
    and the handlebody group.}
    \label{tab:comparison}
\end{table}

Let $\Sigma:= \partial V \setminus \hbox{int}(D)$,
which is a compact, oriented surface with 
$\partial\Sigma\cong S^1$ and with genus $g\geq 1$.
The inclusion $\iota: \Sigma \to V$
induces an embedding of $\modH$ into  $\modM=\modM(\Sigma,\partial \Sigma)$,
which allows to view the former  group as a subgroup of the latter.
Let $F:=\pi_1(V,\star)$, 
and consider the free pair $(\pi,\Alpha)$ where
$$
\pi:=\pi_1(\Sigma,\star) \quad \hbox{and} 
\quad \Alpha:=\ker (\iota_*:\pi \lto F ).
$$
In the following, 
we apply the previous constructions 
to the automorphism group
$\modG=\Aut(\pi,\Alpha)$ of this free pair.
Indeed, according to Griffiths \cite{Griffiths} 
(see Theorem~\ref{thm:Griffiths}), we have 
\begin{equation} \label{eq:Griffiths}
\modH=\modM\cap\modG 
\quad \hbox{(as subgroups of  $\Aut(\pi)$)};
\end{equation}
in other words, a mapping class $g$ of 
$\Sigma \hbox{ rel }  \partial \Sigma$  
extends to a mapping class of $V$ rel $D$ 
if and only if $g$ fixes $\A$ setwise.
Thus,  the filtrations
$(\modG_k)_{k\geq 0}$ and $(\modG_k^0)_{k\geq 0}$ 
of $\modG$ restrict to filtrations
$(\modH_k)_{k\geq 0}$ and $(\modH_k^0)_{k\geq 0}$  of $\modH$, respectively.

%

By \eqref{GkGk1}, 
we have $\modH_k=\modH_k^1\subset \modH_k^0$ for $k\geq 1$.  
Using the identity $f(\zeta)=\zeta\in \pi$ 
for $f\in \modH$, we obtain the following,
which was announced in  \cite[Ex$.$ 10.9]{HM18}.

\begin{theorem}[see Theorem \ref{thm:=}]
    \label{main1}
    We have $\modH_k=\modH_k^0$ for $k\ge1$.
\end{theorem}

\noindent
Thus, we may redefine the Johnson filtration of $\H$ as follows.

\begin{corollary}
    We have an extended N-series
    $\H=\H_0\geq \H_1 \geq \H_2\geq \cdots$
    with trivial intersection, satisfying for any $k\geq 1$
    $$
        \H_k
        =\ker\big(\H\lto \Aut(\pi/\Gamma_k\A)\big)
        =\ker\big(\H\lto \Aut(\A/\Gamma_{k+1}\A)\big).
    $$
\end{corollary}

The first subgroup of the Johnson filtration
of the handlebody group
$$\T:=\H_1=\ker(\H\lto\Aut(F))$$
is known as the \emph{twist group} or the \emph{Luft subgroup} of $\H$.
Indeed, Luft \cite{Luft} (see Theorem \ref{thm:Luft}) showed that $\T$ is generated by Dehn twists along \emph{meridians} of $V$,
i.e$.$
simple closed curves in $\Sigma$ 
bounding properly embedded disks in $V$.
In the comparison between the handlebody
group and the mapping class group (Table~\ref{tab:comparison}),
the twist group is an analogue of the Torelli group.
The latter is known to be 
residually  torsion-free nilpotent \cite{Hain_Torelli,BP}. 
Similarly,  as a direct application of the Johnson filtration
for the handlebody group, we obtain the following.

\begin{theorem}[see Corollary \ref{cor:residual}]
The group $\modT$ is residually 
torsion-free nilpotent.
\end{theorem}
\noindent

Section \ref{sec:first_quotient_handlebody} 
considers the first few terms
of the Johnson filtration $(\modH_k)_{k\geq 1}$.
Since the canonical homomorphism $\H\to\Aut(F)$ is surjective, 
we have an isomorphism
$$\H/\T=\H_0/\H_1\simeq \Aut(F).$$
Furthermore, 
a homomorphism $\Mag:\T\to\Mat(g\times g;\Z[F])$ 
extracted from \eqref{mag} embeds
$\modH_1/\modH_2$ into a space of hermitian matrices,
see Proposition \ref{prop::hermitian}.

In Section \ref{sec:Johnson_homomorphisms_handlebody},
we consider the Johnson homomorphisms for the handlebody group.
While $\tau_1$ is equivalent to 
the above representation $\Mag: \modT \to \hbox{Mat}(g\times g;\Z[F])$,
we need a more general approach to study the Johnson homomorphisms 
in arbitrary degrees.
They form a morphism of  graded Lie algebras
\begin{equation} \label{eq:tau_x}
\overline{\tau}_+:
\overline{\modH}_+ \longrightarrow
\Der_+^\zeta\big(F\!\ltimes\!\Lie(\AAlpha)\big)    ,
\end{equation}
whose source is the associated graded of the Johnson filtration,
and whose target consists of positive-degree derivations of 
$\overline{\Alpha}_\bu =F \ltimes \Lie(\AAlpha)$ 
that vanish on the boundary element $[\zeta]_1 \in \overline{\Alpha}_1=\AAlpha$.
We call $\Der_+^\zeta\big(F\!\ltimes\!\Lie(\AAlpha)\big)$ the \emph{Lie algebra of special derivations}, which
 is the analogue of the Lie algebra of symplectic derivations
in the handlebody case.
We also refine formality of the free pair $(\pi,\Alpha)$
with the boundary condition (see Lemma~\ref{lem:special}):
this leads to \emph{special expansions} of $(\pi,\Alpha)$, 
which should be compared to symplectic expansions of $\pi$.
Thus, we obtain the following analogue of \eqref{eq:varrho_theta} 
for the handlebody group.

\begin{theorem}[see Theorem \ref{thm:varrho_HH}]
Let  $\theta$ be a special expansion of $(\pi,\Alpha)$.
There is an embedding\\[-0.6cm]
\begin{equation}\label{eq:varrho_theta_ter}
\varrho^\theta: \modT  \lto 
\widehat{\Der}_+^\zeta\big(F\!\ltimes\!\Lie(\AAlpha^\Q)\big)    
\end{equation}  
of the twist group into the degree-completed Lie algebra of special derivations 
(equip\-ped with the BCH product),
which induces \eqref{eq:tau_x} 
on the associated graded (with rational coefficients)
\end{theorem}

In Section \ref{sec:special_derivations}, 
we consider the canonical maps 
$\Der_+(F\!\ltimes\!\Lie(\AAlpha)) \to Z^1(F,\Lie(\AAlpha))$
and $\Der_+(F\!\ltimes\!\Lie(\AAlpha)) \to \Hom(\AAlpha,\Lie(\AAlpha))$
which associate to any derivation $d=(d_0,d_+)$ the corresponding $1$-cocycle $d_0$
and the restriction $d_+\vert_{\AAlpha}$ of the corresponding derivation, respectively. 
Let 
$$
D^0_+ \subset Z^1(F,\Lie(\AAlpha)) \quad \hbox{and}  \quad
D^1_+ \subset \Hom(\AAlpha,\Lie(\AAlpha))
$$
be the subgroups defined by the boundary condition $d([\zeta]_1)=0$, which is
satisfied by any $d\in \Der_+^\zeta(F\!\ltimes\!\Lie(\AAlpha))$.
It turns out that the corresponding homomorphisms 
 $\Der_+^\zeta(F\!\ltimes\!\Lie(\AAlpha)) \to D^0_+$
and $\Der_+^\zeta(F\!\ltimes\!\Lie(\AAlpha)) \to D^1_+$ are isomorphisms.
Thus, we obtain two  descriptions $D^0_+$ and $D^1_+$ 
of the Lie algebra $\Der_+^\zeta(F\!\ltimes\!\Lie(\AAlpha))$, 
and we are free to work with either the $1$-cocycles 
$\tau_k^0(f)\in D^0_k$ 
or the homomorphisms $\tau_k^1(f) \in D^1_k$
for $f\in \modH_k$.

We give in  Section \ref {sec:diagrams} a diagrammatic description 
of the graded Lie algebra $\Der_+^\zeta(F\!\ltimes\!\Lie(\AAlpha))$ 
with rational coefficients. 
Specifically, we consider  
a $\Q$-vector space $\modD(\AAlpha^\Q, \Q[F])$
that is generated by oriented trivalent trees
with leaves colored by $\AAlpha^\Q:= \AAlpha \otimes \Q$
and edges colored by $\Q[F]$. For example, here is a generator in degree $3$:
$$
\labellist
\small\hair 2pt
\pinlabel {${a}$} [r] at 2 150
 \pinlabel {${d}$} [r] at 2 5
\pinlabel {${b}$} [l] at 293 149
\pinlabel {${c}$} [l] at 294 2
\pinlabel {$x$} [b] at 53 111
\pinlabel {$y$} [b] at 153 85
\pinlabel {$z$} [tr] at 255 48
\pinlabel {\quad (with $a,b,c,d\in \AAlpha^\Q$ and $x,y,z\in \Q[F]$)} [l] at 340 2
\endlabellist
\centering
\includegraphics[scale=0.24]{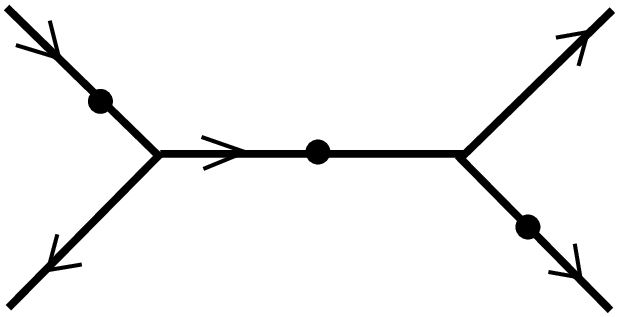} \qquad \qquad \qquad \qquad 
$$
In comparison with the space $\modD(H^\Q)$ which serves for the mapping class group of~$\Sigma$,
the definition of $\modD(\AAlpha^\Q, \Q[F])$ involves the cocommutative Hopf algebra $\Q[F]$
and the $\Q[F]$-module $\AAlpha^\Q$.
We obtain the following analogue of \eqref{eq:Kontsevich} for $\modH$.

\begin{theorem}[See Theorem \ref{th:tree_bracket}]
There is  a graded Lie $\Q$-algebra isomorphism
$$
{\Der}_+^\zeta\big(F\!\ltimes\!\Lie(\AAlpha^\Q)\big)    
\simeq \modD\big(\AAlpha^\Q, \Q[F]\big),
$$    
where the Lie bracket of trees
is defined by ``grafting leaves-to-beads'' and ``branching leaves-to-leaves'' using
intersection operations 
$\Theta: \Z[F] \times \AAlpha \to \Z[F] \otimes \Z[F]$
and $\Psi: \AAlpha \times \AAlpha \to \Z[F] \otimes \AAlpha$, respectively.
\end{theorem}

\noindent 
Thus, the Lie bracket  in $\modD(\AAlpha^\Q, \Q[F])$
is an analogue of \eqref{eq:trees_bracket}, 
where the operations $\Theta$ and $\Psi$ play the role of the pairing $\omega$.
These operations, 
 presented in Appendix~\ref{sec:intersection}, are derived 
from the ``homotopy intersection form''
$\eta$ of $\Sigma$ \cite{Turaev}.
The properties of $\Theta$ and $\Psi$ necessary for the proof of Theorem \ref{th:tree_bracket}
are derived from the axioms of a ``quasi-Poisson double bracket'', which is produced from $\eta$ \cite{MT14}.

Section \ref{sec:formulas_examples} is devoted to explicit computations 
of the Johnson homomorphisms, starting with the degree $1$ case:

\noindent
\begin{minipage}{0.75\textwidth}
\begin{example}
Consider the first Johnson homomorphism $\tau_1^0: \modH_1 \to D_1^0$,  defined  on the twist group~$\modH_1= \modT$.
We can identify the target $ D^0_1$
with $\mathrm{Sym}^2(\AAlpha)_{\Z[F]}$,
the $\Z[F]$-coinvariants
of the symmetric tensors $\mathrm{Sym}^2(\AAlpha)$. 
For any properly embedded disk $U\subset V$,
we have 
\begin{equation}  \label{eq:tau_1^0}
\tau_1^0\big(T_{\partial U}\big) = - [u] \otimes [u] \in D_1^0    ,
\end{equation}
where $u\in \Alpha$ is  obtained from the closed curve $\partial U$
by orienting it and basing it at $\star$ in an arbitrary way
(see Proposition \ref{prop:tau_1(disk_twist)}).

\end{example}
\end{minipage}
\begin{minipage}{0.23\textwidth}
\labellist
\small\hair 2pt
 \pinlabel {\textcolor{blue}{$U$}} [t] at 42 285
 \pinlabel {$\vdots$} at 111 217
\endlabellist
\centering
\includegraphics[scale=0.23]{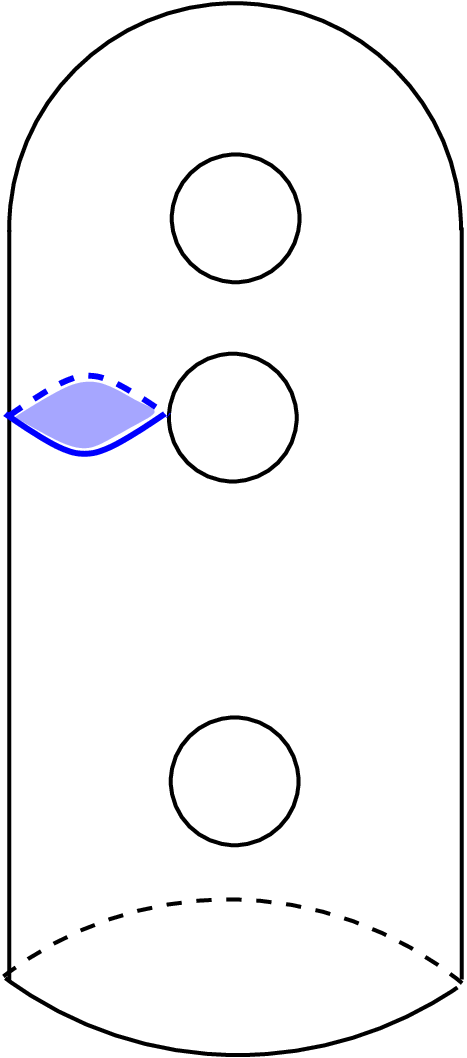}
\end{minipage}

Using the diagrammatic formula for the Lie bracket in
${\Der}_+^\zeta\big(F\!\ltimes\!\Lie(\AAlpha^\Q)\big) \simeq \modD(\AAlpha^\Q, \Q[F])$,
one can use \eqref{eq:tau_1^0} to compute $\tau_k$ on $\Gamma_k \modT \subset \modH_k$.
This is the key to prove the following:

\begin{theorem}[see Theorem \ref{th:infiniteness}]
Let $g\geq 3$.
There exists a subgroup $\modL$ of $\modT$
which is free of countably-infinite rank and 
such that $\overline{\modL}_+$
(the associated graded of $\modL$ with respect to its lower central series)
embeds both into $\overline{\modT}_+$ 
(the associated graded of $\modT$ with respect to its lower central series)
and $\overline{\modH}_+$ 
(the associated graded of $\modH$ with respect to the Johnson filtration).
\end{theorem}

\noindent
It follows that, for every $k\geq 1$, 
the quotient groups $\Gamma_k \modT/\Gamma_{k+1} \modT$
and $\modH_k/\modH_{k+1}$ are not finitely generated.
For $k=1$, we recover McCullough's result \cite{McCullough}
that the (abelianization of the) twist group $\modT$ is not finitely generated 
(see Theorem~\ref{th:McCullough} for a more precise statement).
Furthermore, as new results, we obtain that each term
of the Johnson filtration (resp$.$ of the lower central series)
of $\modT$ is also not finitely generated.
This is in sharp contrast with what is known 
for the Torelli group \cite{Johnson_finiteness} 
and its corresponding filtrations \cite{CEP}.

Next, as a strong generalization of  \eqref{eq:tau_1^0},
we compute the representation \eqref{eq:varrho_theta_ter} on 
a Dehn twist $T_{\partial U}$ along an arbitrary meridian $\partial U$,
and for any special expansion $\theta$ of the free pair $(\pi,\Alpha)$ (Theorem \ref{th:KK_analogue}).
This is an analogue
of a formula of Kawazumi and Kuno \cite{KK14}, who computed \eqref{eq:varrho_theta} for an arbitrary Dehn twist.

Finally, we consider embeddings of 
the pure braid group $PB_g$ into the twist group $\modT$.
Theorem \ref{th:Milnor_to_Johnson} relates the lower central series of $PB_g$ 
to the Johnson filtration of $\modH$ and, consequently, 
Milnor invariants to Johnson homomorphisms. 
(An analogous result for the surface mapping class groups was given by
Gervais and Habegger \cite{GH02}.)
This is another
evidence of
the richness of the Johnson--Morita theory 
for $\modH$.\\

\noindent
\textbf{Acknowledgments.} 
The authors would like to thank J.-B. Meilhan 
for his helpful comments on the first stages of this manuscript, 
J. Darné for pointing out a missing hypothesis
in \cite[Theorem 10.2]{HM18} (see the footnote of page \pageref{page:Darne}),
and Mai Katada for her helpful comments.
The second author is also grateful to R. Hain 
for explaining him the main constructions of \cite{Hain_handlebody}.

The work of the first author was partly funded by JSPS KAKENHI Grant Number 18H01119 and 22K03311.
The work of the second author was partly funded by the project “AlMaRe” (ANR-19-CE40- 0001-01);
the IMB receives support from the EIPHI Graduate School (ANR-17-EURE-0002).\\

\noindent
\textbf{Conventions.} Unless otherwise stated, all group actions are \emph{left} actions
and modules (over any ring) are \emph{left} modules.
If a group $G$ acts on an abelian group $A$, then the action of $g\in G$
on $a\in A$ is denoted by ${}^ga$ or $g\cdot a$, or even $g a$
if there is no risk of confusion.
If not specified, the ground ring for linear algebra is $\Z$.

If $S$ is a subset of a group $G$, let $\la S \ra$ denote the subgroup of $G$ generated by~$S$, and $\lala S \rara$ the subgroup normally generated by $S$.
For all $x,y \in G$, we set ${}^xy=xyx^{-1}$, $y^x=x^{-1}yx$
and  $[x,y]=xyx^{-1}y^{-1}$. For any two subgroups $K,H$ of $G$, let $[K,H]$ denote the subgroup of $G$ generated by the commutators
$[k,h]$ for $k\in K,h\in H$.

For any two $\Z$-modules $U$ and $V$, an element 
$z\in U\otimes V$ is sometimes 
denoted by $z^\ell \otimes z^r$ to suggest its  expansion 
$z=\sum_i u_i \otimes v_i$ in terms of finitely many 
elements $u_i\in U$ and $v_i \in V$.

\section{Johnson homomorphisms for an extended N-series}
\label{sec:general_theory}

In this section, we summarize parts of the theory of generalized Johnson homomorphisms 
as developed in \cite{HM18} for extended N-series.

\subsection{Extended N-series and extended graded Lie algebras} \label{subsec:eN-series}

An \emph{extended N-series} $K_*=(K_m)_{m\ge 0}$  is a  descending  series of subgroups
\begin{gather*}
K_0\ge K_1\ge  \cdots \ge K_k \geq \cdots
\end{gather*}
 such that $ [K_m,K_n]  \le  K_{m+n}$ for all $m,n\ge 0$.
In particular, $K_m$ is normal in $K_0$ for any $m\geq 0$.
An \emph{automorphism} of $K_*$ is an automorphism $f$ of $K_0$ with
$f(K_m)=K_m$ for all $m\ge0$.  Let $\Aut(K_*)$ denote the group of all automorphisms of $K_*$.

An \emph{extended graded Lie algebra}
$L_\bu=(L_m)_{m\ge 0}$ consists of a group $L_0$,
 a graded Lie algebra $L_+  = (L_m)_{m\ge 1}$ and
 an action {$(z,x)\mapsto {}^z x$} of $L_0$ on $L_+$ by graded Lie algebra automorphisms.
Any {extended N-series} $K_*$ has an \emph{associated} extended graded Lie algebra $\overline{K}_\bullet$
defined by $\overline{K}_m=K_m/K_{m+1}$ for all $m\geq 0$, where the Lie bracket in $\overline{K}_\bullet$ is induced
by the commutator operation $(x,y)\mapsto [x,y]$,
and $\overline{K}_0$ acts on $\overline{K}_+$ by conjugation $(x,y) \mapsto {}^xy$.

We will sometimes need the rational version $\overline{K}_\bullet^\Q$ defined by $\overline K_0^\Q=K_0/K_1$
and the graded Lie $\Q$-algebra 
$\overline{K}_+^\Q= \big((K_m/K_{m+1})\otimes \Q\big)_{m\geq 1}$.

A \emph{morphism} $f_\bullet=(f_m)_{m\ge 0}: L_\bu\rightarrow L'_\bu$ of extended graded Lie algebras
consists of  a group homomorphism $f_0: L_0\rightarrow L'_0$,
and  a graded Lie algebra homomorphism  $f_+  = (f_m)_{m\ge 1}: L_+\rightarrow L'_+$
which is equivariant over $f_0$:
$$
f_m({}^xy)={}^{f_0(x)} (f_m(y) ) \quad \hbox{for all $x\in L_0$, $y\in L_m$, $m\geq 1$.}
$$

A \emph{derivation} of \emph{degree} $m\ge1$ of an extended graded Lie algebra $L_\bullet$ 
 is a family  $d=(d_i)_{i\ge 0}$ of maps
  $d_i: L_i\rightarrow L_{m+i}$ satisfying the following  conditions.
\begin{enumerate}
\item $d_+ = (d_i)_{i\ge 1}$ is a derivation of the graded Lie
  algebra $L_+$, i.e. 
$$
  d_{i+j}([a,b])=[d_i(a),b]+[a,d_j(b)] \quad \hbox{for all $a\in L_i$, $b\in L_j$, $i,j\ge 1$.}
$$
\item $d_0: L_0\rightarrow L_m$ is a $1$-cocycle, i.e. 
$$
  d_0(ab)=d_0(a)+{}^a(d_0(b)) \quad \hbox{for all $a,b\in L_0$}.
$$
\item $d_0$ controls the defect of $L_0$-equivariance of $d_+$, i.e.
$$
    d_i({}^ab)={}^a(d_i(b)) + [d_0(a),{}^ab] \quad \hbox{for  all $a\in L_0$, $b\in L_i$, $i\ge 1$}.
$$
\end{enumerate}

Let $\Der_m(L_\bu)$ be the set  of derivations of $L_\bullet$ of degree $m$. Then $\Der_+(L_\bu) := (\Der_m(L_\bu))_{m\ge 1}$ is a graded Lie algebra whose Lie bracket
generalizes the usual Lie bracket for derivations of $L_+$  \cite[Theorem 5.2]{HM18}.
Furthermore, $\Der_0(L_\bu) := \Aut(L_\bu)$ acts on $\Der_+(L_\bu)$ by conjugation 
and we get an extended graded Lie algebra  $\Der_\bu(L_\bu)$ \cite[Theorem 5.3]{HM18}.
The necessary definitions will appear 
in Section \ref{sec:free_pairs}.

Let $K_*$ be an extended N-series. An \emph{action} of a group $\modG$ on $K_*$
is an action of $\modG$ on $K_0$ such that $g(K_m)=K_{m}$ for all $g\in \modG,m\geq 0$. In other words, it is a homomorphism $\modG\to\Aut(K_*)$.
In this case, $K_*$ induces an extended N-series $\modG_*$, called the \emph{Johnson filtration} and defined by 
\begin{equation} \label{eq:J_filtration}
\modG_m := \bigcap_{j\geq 0} \modG_m^j ,
\quad \hbox{where } \ 
\modG_m^j := \big\{g \in \modG \mid   g(x)x^{-1} \in K_{m+j} \text{ for all }x \in K_j\big\}
\end{equation}
for $m\geq 0$. Note that $\modG_0=\modG$.
There is an injective extended graded Lie algebra morphism 
\begin{equation}   \label{eq:Johnson_global}
\overline{\tau}_\bu: \overline{\modG}_\bu \longrightarrow \Der_\bu ( \overline{K}_\bu),
\end{equation}
called the \emph{Johnson morphism}. 
For the definition of~$\bar \tau_\bu$,
see \cite[\S 6]{HM18} and also Section \ref{sec:free_pairs}.

\begin{remark}
The study of extended N-series includes that of \emph{N-series}, which were considered in full generality by Darn\'e, too~\cite{Darne}.
(Indeed, an N-series is an extended N-series $K_*$ 
for which the acting group $K_0/K_1$ is trivial.)
\end{remark}

\subsection{Formal extended N-series} \label{subsec:formal_eN-series}

For now, we work over $\Q$. If $K_*$ is formal in some sense, then the graded Lie algebra homomorphism $\bar \tau_+$ in \eqref{eq:Johnson_global} 
is the associated graded of an ``infinitesimal'' action of $\modG$ on $K_*$.
We review below such a  situation,
and refer the reader to \cite{SW19,SW20}
for a general introduction to formality of groups.

The extended N-series $K_*$ induces a filtration $J_*^\Q(K_*)$ of the algebra $\Q[K_0]$,
where $J_m^\Q(K_*)$ is the ideal generated by the elements  of the form  ${(x_1-1)\cdots (x_p-1)}$
for all  $x_1 \in K_{m_1}, \dots, x_p \in K_{m_p}$, $m_1+\dots +m_p \geq m$, $m_1,\ldots ,m_p\ge 1,\ p\ge 1$.
Then the Hopf algebra  structure of $\Q[K_0]$ is compatible with the filtration $J_*^\Q(K_*)$. Therefore we have a graded Hopf $\Q$-algebra 
$$
\gr_\bullet\big( J_*^\Q(K_*)\big) 
= \bigoplus_{i\geq 0} {J_i^\Q(K_*)}/{J_{i+1}^\Q(K_*)}
$$ 
and a complete Hopf $\Q$-algebra 
$$
\widehat{\Q [K_*]} = \plim_{k}\  \Q[K_0]\, / J_k^\Q(K_*).
$$
The extended N-series $K_*$ is said to be \emph{formal}
if $\widehat{\Q [K_*]}$ is isomorphic to 
the degree-completion of $\gr_\bullet\big( J_*^\Q(K_*)\big)$ through an isomorphism whose associated graded is the identity.

We can characterize the formality as follows.
First of all,
we have the following  generalization for extended N-series \cite[Theorem 11.2]{HM18}  of a classical result of Quillen for the lower central series of groups \cite{Quillen}:
\begin{equation} \label{eq:Quillen}
\gr_\bullet\big( J_*^\Q(K_*)\big) 
\simeq U(\overline K_\bu^\Q).    
\end{equation}
Here  $ U(\overline K_\bu^\Q)$ 
is the \emph{universal enveloping algebra} of the extended graded Lie $\Q$-algebra~$\overline K_\bu^\Q$.
As a $\Q$-coalgebra, $U(\overline K_\bu^\Q)$ is the tensor product 
$U(\overline K_+^\Q)\otimes_\Q\Q[K_0/K_1]$ of the usual
universal enveloping algebra of the Lie $\Q$-algebra $\overline K_+^\Q$
and the group algebra of $K_0/K_1$. The multiplication is defined by 
\begin{equation} \label{eq:mult_law}
(v\otimes y) \cdot (v' \otimes y') = v\,{}^y\!v' \otimes yy'
\quad \hbox{for all $v,v'\in U(\overline K_+^\Q)$
and $y,y' \in K_0/K_1$}.    
\end{equation}
We regard $U(\overline K_+^\Q)$ and $\Q[K_0/K_1]$
as subalgebras of $U(\overline K_\bu^\Q)$ in the obvious way,
 and any element $ v\otimes y\in  U(\overline K_+^\Q) $ is written as a product $v \cdot y$,
see  \cite[\S 11]{HM18} for details.
The extended N-series $K_*$ is formal if and only if there is a monoid {homomorphism}
$$
\theta: K_0 \longrightarrow \hat U(\overline K_\bu^\Q)
$$
which takes values in the degree-completion $\hat U(\overline K_\bu^\Q)$ of $U(\overline K_\bu^\Q)$ and  
maps any $x\in K_i$ ($i \geq 0$)  to a group-like element of the form
\begin{equation} \label{theta}
\theta(x) = \left\{\begin{array}{ll}
1+ [x]_{i} + (\deg >i) & \hbox{if } i>0,\\
 {[x]_0} + (\deg>0) & \hbox{if } i=0,
\end{array}\right.
\end{equation}
where $[x]_i$ denotes
the class of $x$ modulo $K_{i+1}$.

In such a case, $\theta$ is called an \emph{expansion} of the extended N-series $K_*$ and it allows for the following constructions.
Consider the homomorphism 
$$
\rho^\theta : \modG \longrightarrow 
\Aut\big( \hat U(\overline K_\bu^\Q)\big),
\quad g \longmapsto \hat \theta \circ \widehat{\Q[g]} \circ \hat \theta^{-1},
$$
where $\widehat{\Q[g]}$ denotes the automorphism of $\widehat{\Q [K_*]}$ 
induced by the action of $g$ on $K_0$
and the isomorphism $\hat \theta: \widehat{\Q [K_*]} \to \hat U(\overline K_\bu^\Q)$ 
is the extension of $\theta$.
Note that $\rho^\theta $ takes values in 
the group $\operatorname{IAut}\big( \hat U(\overline K_\bu^\Q)\big)$
of complete Hopf algebra automorphisms of 
$\hat U(\overline K_\bu^\Q)$
that induce the identity on the associated graded.
Let also $$\Der_+\big( \hat U(\overline K_\bu^\Q)\big)$$ 
be the space of derivations of the  algebra $\hat U(\overline K_\bu^\Q)$
that map every  $x\in \overline K_0^\Q=K_0/K_1$ 
to $\widehat{\overline K}_+^\Q\, x$
and that maps $\widehat{\overline K}_{\geq m}^\Q$ to  
$\widehat{\overline K}_{\geq m+1}^\Q$ for every $m\geq 1$.
According to the following  lemma, 
which is implicit in \cite[\S 12]{HM18},
there is a one-to-one correspondence between
positive-degree derivations 
of the extended graded Lie $\Q$-algebra $\overline K_\bu^\Q$ 
and derivations of its universal enveloping algebra 
of the previous type.

\begin{lemma} \label{lem:canonical_isos}
There are canonical $\Q$-linear isomorphisms
$$
\xymatrix{
\operatorname{IAut}\big( \hat U(\overline K_\bu^\Q)\big)
\ar[r]_-{\simeq}^-\log & 
\Der_+\big( \hat U(\overline K_\bu^\Q)\big)
\ar[r]_-\simeq^-{\operatorname{res}} & \widehat{\Der}_+(\overline K_\bu^\Q).
}
$$
Here $\,\log$ is  the formal logarithm series, 
i.e$.$ it maps any automorphism $a$ to 
$$
\log(a) := \sum_{n\geq 1} \frac{(-1)^{n+1}}{n!}  (a -\id)^n,
$$
and $\mathrm{res}$ is the restriction map, which maps any derivation $d$ to
$$
\operatorname{res}(d):= (d_0,d_+)
\ \hbox{ where } \ d_0(x) = d( x) x^{-1},\, x\in \overline{K}_0
\ \hbox{ and } \ d_+(u) =d(u),\, u \in \overline K_+^\Q.
$$
\end{lemma}

\begin{proof}
The map $\log : \operatorname{IAut}\big( \hat U(\overline K_\bu^\Q)\big) 
\to \Der_+\big( \hat U(\overline K_\bu^\Q)\big)$ is proved to be well-defined
by following the first four paragraphs of the proof of \cite[Lemma 12.5]{HM18}.

To construct the inverse to $\mathrm{log}$, observe that any 
$d \in \Der_+\big( \hat U(\overline K_\bu^\Q)\big)$ increases degrees, 
hence the formal power series 
$$
\exp(d) := \sum_{n\geq 1} \frac{d^n}{n!}
 $$
 converges and induces the identity on the associated graded. 
 Since $d$ is an algebra derivation, $\exp(d)$ is an algebra automorphism.
 Besides, $d$ is a coderivation, i.e$.$ we have the identity 
 $\Delta d= (d\hat \otimes \id + \id \hat \otimes d) \Delta$, 
 as can be checked on the elements of $\overline K_\bu^\Q$ 
 (which generate the algebra $U(\overline K_\bu^\Q)$). Therefore $\exp(d)$ is a coalgebra map.
 Thus, we obtain a  map  
 $\exp: \Der_+\big( \hat U(\overline K_\bu^\Q)\big) 
 \to \operatorname{IAut}\big( \hat U(\overline K_\bu^\Q)\big)  $,
 which is inverse to $\log$.

 The map $\operatorname{res} : \Der_+\big( \hat U(\overline K_\bu^\Q)\big)
 \to \widehat{\Der}_+(\overline K_\bu^\Q)$ is proved to be well-defined
by following the last paragraph of the proof of \cite[Lemma 12.5]{HM18}.
Since $\overline K_\bu^\Q$  generate the algebra $U(\overline K_\bu^\Q)$, this map is injective.

To show that $\operatorname{res}$ is also surjective, 
let $(d_0,d_+) \in \widehat{\Der}_+(\overline K_\bu^\Q)$.
The Lie algebra derivation $d_+$ of $\overline K_+^\Q$
extends to a unique algebra derivation $\tilde{d}_+$ of its universal enveloping algebra.
Besides, let $\tilde{d}_0:\Q[K_0/ K_1] \to \hat U(\overline K_\bu^\Q)$ 
be the $\Q$-linear map defined by $\tilde d_0(x)=d_0(x)x$ for any $x\in K_0/K_1$.
Then, there is a unique $\Q$-linear map 
$\tilde d: U(\overline K_\bu^\Q) \to \hat U(\overline K_\bu^\Q)$
defined by 
$\tilde{d}(vy)  = \tilde{d_+}(v)\,y + v\, \tilde{d}_0(y)$ 
for any $v\in U(\overline K_+^\Q)$ 
and $y \in \Q[K_0/K_1]$.
Since $\tilde d$ increases degrees, it extends to 
$\tilde d: \hat U(\overline K_\bu^\Q) \to \hat U(\overline K_\bu^\Q)  $ by continuity.
It remains to check that $\tilde d$ is an algebra derivation.
For any $v,v'\in U(\overline K_+^\Q) $ and $y,y' \in K_0/K_1$, we have
\begin{eqnarray*}
\tilde d\big((vy)(v'y')\big) & = & \tilde d\big((v\,{}^y\!v')(yy')\big) \\
&=& \tilde d_+\big(v\,{}^y\!v'\big)\, (yy')+ (v\,{}^y\!v')\, \tilde d_0(yy') \\
&=& 
 \big( \tilde d_+(v)\, {}^y\!v' + v\, \tilde d_+({}^y\!v')
 + v\,{}^y\!v'\,  d_0(y)+ v\,{}^y\!v'\, {}^y \!d_0(y')\big)yy'
\end{eqnarray*}
and, since the defect of $\Q[K_0/K_1]$-equivariance of $d_+$
(and, so, $\tilde d_+$) is controlled by the 1-cocycle $d_0$, we get
\begin{eqnarray*}
 \tilde d\big((vy)(v'y')\big) 
&=&  \big( \tilde d_+(v)\, {}^y\!v' + v\, {}^y\!\tilde d_+(v')
+ v\, d_0(y)\, {}^y\!v' + v\,{}^y\!v'\, {}^y \!d_0(y') \big)\, yy'.
\end{eqnarray*}
On the other hand, we have
\begin{eqnarray*}
(vy)\, \tilde d(v'y'\big) &=& (vy)\, \big( \tilde{d_+}(v')\,y' 
+  v'\, \tilde{d}_0(y')\big) 
 \ = \ v\,y\,\tilde{d_+}(v')\, y' +v\,y\,v'd_0(y')y'
\end{eqnarray*}
and
\begin{eqnarray*}
\tilde d(vy)\, (v'y') &=& 
\big(\tilde{d_+}(v)\,y + v\, \tilde{d}_0(y) \big)\, (v'y')
\ = \ \tilde{d_+}(v)\,y\,v'\,y'+ v\,  d_0(y)\, y\, v'\, y'.
\end{eqnarray*}
We conclude that 
$\tilde d\big((vy)(v'y')\big) 
= (vy)\, \tilde d(v'y'\big) + \tilde d(vy)\, (v'y')$.
\end{proof}

The previous lemma shows that, for any $m\geq 1$ and $g\in \modG_m$, 
the automorphism $\rho^\theta(g)$ induces an element 
{$\varrho^\theta(g) \in \widehat{\Der}_+(\overline K_\bu^\Q)$} of  degree $\geq m$.
The complete Lie algebra $\widehat{\Der}_+(\overline K_\bu^\Q)$ can be regarded as a filtered group 
whose multiplication law is defined 
by the BCH formula.
Then we obtain a filtered group homomorphism
\begin{equation} \label{eq:varrho}
\varrho^\theta: \modG_1 \longrightarrow \widehat{\Der}_+(\overline K_\bu^\Q)
\end{equation}
inducing the rational version $\bar\tau_+^\Q$ of $\bar \tau_+$ on the associated graded,
see \cite[Theorem 12.6]{HM18}.
We view \eqref{eq:varrho} as an  infinitesimal version
of the action of the group~$\modG_1$ 
on the extended N-series $K_*$.

\section{Johnson homomorphisms for a free pair}
\label{sec:free_pairs}

We continue the review of generalized Johnson homomorphisms 
by focusing on the automorphism group of a free pair, as in \cite[\S 10.1]{HM18}.
We also establish the formality in this case.

\subsection{Free pairs and their automorphism groups} \label{subsec_free_pairs}

A \emph{free pair} is a pair $(\pi,\Alpha)$ of a free group $\pi$ and
a  non-abelian normal  subgroup $\Alpha$ with $F:=\pi/\Alpha$ free. 
Let $\varpi:\pi \to F$ denote the canonical projection.
In the sequel, for simplicity, 
we assume that the free groups $\pi$ and $F$ have finite rank.

\begin{remark}
By \cite[Theorem 6.4]{FJ50}, $\pi$ is the free product $A*B$ of subgroups $A,B\le\pi$ such that
  $A\subset\A$ and $\varpi$ maps $B$ isomorphically onto~$F$.
Thus, there is a basis of $\pi$
\begin{equation} \label{eq:free_basis}
   \{\al_i\}_{i\in I}\sqcup\{\be_j\}_{j\in J} 
\end{equation}
(with $I$ and $J$ finite)
such that $A=\la\al_i \, \vert \, i \ra$, 
$B=\la\be_j \, \vert \, j  \ra$ 
and $\A=\lala\al_i \, \vert \, i \rara$.
Although the constructions in this section are independent
of the choices of $A,B$ and their bases,
we sometimes use these bases.
\end{remark}

An \emph{automorphism} of $(\pi,\Alpha)$ is an automorphism $g$ of $\pi$ with $g(\Alpha) = \Alpha$.
Let $\Aut(\pi,\Alpha)$ denote the group of automorphisms of $(\pi,\Alpha)$.
Following \cite[\S 10.1]{HM18}, let us review how the theory of generalized Johnson homomorphisms applies to  $\Aut(\pi,\Alpha)$.
Let
$$
\Alpha = \Gamma_1 \Alpha \ge \Gamma_2 \Alpha \ge  \Gamma_3 \Alpha \ge  \cdots
$$
be the \emph{lower central series} of $\Alpha$, defined inductively by $\Gamma_1 \Alpha:= \Alpha$ 
and $\Gamma_{i+1} \Alpha:= [\Gamma_i\Alpha,\Alpha]$ ($i\ge1$). 
Setting $\Alpha_0:= \pi$ and $\Alpha_m:= \Gamma_m \Alpha$ for $m\geq 1$, we get an extended N-series
$$
\Alpha_* := (\Alpha_i)_{i\geq 0},
$$
 and consequently an extended graded Lie algebra
$$
\overline{\Alpha}_\bullet := (\Alpha_m/\Alpha_{m+1})_{m\geq 0}.
$$

\begin{lemma} \label{lem:freeness}
The group  $\AAlpha := \overline{\Alpha}_1 = \Alpha/[\Alpha,\Alpha]$ is free as a $\Z[F]$-module,
where the action of $F$ on $\AAlpha$ is induced by the conjugation of $\pi$ on $\Alpha$.
\end{lemma}

\begin{proof}
Consider a basis of $\pi$ of type \eqref{eq:free_basis},
and let $[\alpha_i]\in \Alpha/[\Alpha,\Alpha]$ denote
the class of $\alpha_i \in \Alpha$. 
Then the $\Z[F]$-module $\AAlpha$ is free on the subset~$\big\{
[\alpha_i] \big\}_{i\in I}$.
\end{proof}

Since $\Alpha$ is a free group, 
its associated graded $\overline{\Alpha}_+$ is the free Lie algebra $\Lie(\AA)$ on $\AA$.
For any $x\in \pi$ and $m\geq 0$,
let $[x]_m$ (or sometimes $[x]$) denote the class of $x$ modulo $\Alpha_{m+1}$.
Note that if $x\in \Alpha_m$, then  $[x]_m \in \overline{\Alpha}_{m}$.

The group $\modG:= \Aut(\pi,\Alpha)$ acts on the extended N-series $\Alpha_*$. 
Let  $\modG_* = (\modG_i)_{i\geq 0}$ be the corresponding {Johnson filtration}, as defined by \eqref{eq:J_filtration}.
It follows from \cite[Proposition 10.1]{HM18} that
$$
\modG_m= \modG_m^0 \cap \modG_m^1 \quad \text{for all $m\geq 0$},
$$
where 
$$
\modG_m^0=\ker\big(\modG \lto \Aut(\pi/\Alpha_m)\big)  
\quad \hbox{and} \quad 
\modG_m^1=\ker\big(\modG \lto \Aut(\Alpha/\Alpha_{m+1})\big).
$$\label{page:Darne}
Furthermore, it follows from \cite[Theorem 10.2]{HM18}\footnote{The result \cite[Theorem 10.2]{HM18} applies to an extended N-series $K_*=(K_m)_{m\geq 0}$ such that $K_m=\Gamma_m K_1$ and $K_1$ is a
non-abelian free group. Its proof is by induction on $m\geq 0$, which is achieved thanks to \cite[Lemma 10.3]{HM18}. But, 
this lemma does not apply for $m=1$.
This gap in \cite[Theorem 10.2]{HM18} is fixed by assuming that  $\overline{K}_1$ is torsion-free as a $\Z[\overline{K}_0]$-module.} 
and Lemma \ref{lem:freeness} that
\begin{equation}\label{GmGm1}
\modG_m = \modG_m^1 \quad \text{for all $m\geq 0$}.
\end{equation}
 
Hence we have two mutually nested filtrations $\modG_*$ and~$\modG^0_*$:
\begin{equation} \label{nested}
\modG=\modG^0_0=\modG_0\ge \modG^0_1\ge \modG_1\ge \cdots \ge \modG_{m-1} \ge \modG^0_m\ge \modG_m\ge \cdots.
\end{equation}
Since $\Alpha$ is free, we have $\bigcap_{i}\Gamma_i\A=\{1\}$. 
Hence
\begin{equation} \label{eq:completeness}
\bigcap_{m\geq 0} \modG_m^0= \bigcap_{m\geq 0} \modG_m= \{1\}   . 
\end{equation} 

\subsection{Truncations of the Johnson morphism for a free pair} 

Since the graded Lie algebra $\overline{\Alpha}_+=\Lie(\AAlpha)$ 
is free on its degree $1$ part $\AAlpha$,
any derivation of the extended graded Lie algebra $\overline{\Alpha}_\bu$ is equivalent to its truncations 
to the degree~$0$ and~$1$ parts of $\overline{\Alpha}_\bu$. Specifically, there is an isomorphism
\begin{equation} \label{eq:DD}
\Der_m(\overline{\Alpha}_\bu)  \longrightarrow D_m(\overline{\Alpha}_\bu),\quad (d_i)_{i\ge 0} \longmapsto (d_0,d_1),
\end{equation}
where
\begin{eqnarray}
D_0(\overline{\Alpha}_\bu)
\notag &= & \big\{(d_0,d_1)\in \Aut(F)\times \Aut(\AAlpha)\\
\label{eq:equiv0} & & \qquad \big\vert \  d_1({}^fa)={}^{d_0(f)}(d_1(a))\hbox{ for $f\in F$, $a\in \AAlpha$}\big\}
\end{eqnarray}
and, for every $m\geq 1$, 
\begin{eqnarray}
\notag    D_m(\overline{\Alpha}_\bu)&= &     \big\{(d_0,d_1)\in Z^1(F,\overline{\Alpha}_m)\times \Hom(\AAlpha,\overline{\Alpha}_{m+1})\\
\label{eq:equiv1}    && \qquad  \big\vert \ d_1({}^fa)=[d_0(f),{}^fa]+{}^f(d_1(a))\  \hbox{for } f\in F,a\in \AAlpha    \big\}.
\end{eqnarray}
According to \cite[Proposition 7.4]{HM18},
the extended graded Lie algebra structure on $\Der_\bu(\overline{\Alpha}_\bu)$
corresponds through~\eqref{eq:DD} to the following  extended graded Lie algebra structure on  $D_\bu(\overline{\Alpha}_\bu)$:

\begin{itemize}
  \item {the} Lie bracket $[d,e]\in D_{m+n}(\overline{\Alpha}_\bu)$ of $d=(d_0,d_1)\in D_m(\overline{\Alpha}_\bu)$ 
  and $e=(e_0,e_1)\in D_n(\overline{\Alpha}_\bu)$
  with  $m,n\ge 1$ is defined by
    \begin{eqnarray}
     \label{eq:b0} &&[d,e]_0(f)  =  d_n(e_0(f))-e_m(d_0(f))-[d_0(f),e_0(f)] \quad  \hbox{for } f\in F,\\
      \label{eq:b1} &&{[d,e]_1(a)}  = d_{n+1}(e_1(a))-e_{m+1}(d_1(a))\quad\quad
      \quad \quad \quad \quad   \hbox{for }  a\in \AAlpha,
    \end{eqnarray}
    where $d_+=(d_i)_{i\geq 1}$ and $e_+=(e_j)_{j\geq1}$
    {are the derivations of $\overline{\Alpha}_+$ extending $d_1$ and $e_1$, respectively;}
  \item {the} action  ${}^e d\in D_m(\overline{\Alpha}_\bu)$ of $e=(e_0,e_1)\in D_0(\overline{\Alpha}_\bu)$ on $d=(d_0,d_1)\in D_m(\overline{\Alpha}_\bu)$
  with $m\ge 1$ is defined by
    \begin{eqnarray}
	\label{eq:e0} ({}^ed)_0(f)&=&e_m d_0 e_0^{-1}(f)\quad \hbox{for } f\in F,\\
	\label{eq:e1} ({}^e d)_1(a)&=&e_{m+1}d_1e _1^{-1}(a)\quad  \hbox{for } a\in \AAlpha,
    \end{eqnarray}
    where $e_+=(e_i)_{i\geq 1}$ is the automorphism of ${\overline{\Alpha}_+}$ {extending}~$e_1$.
\end{itemize}

We now apply the truncation isomorphism \eqref{eq:DD} to the  Johnson morphism~$\overline{\tau}_\bu$
mentioned in  \eqref{eq:Johnson_global}.
For every $m\geq 0$, 
let $\tau_m^0$ and $\tau_m^1$ denote the two components of 
the image of  $\bar \tau_m$ in $\Der_m(\overline{\Alpha}_\bu)$ by the map \eqref{eq:DD}.
Hence, for $m=0$, we get two homomorphisms
\begin{equation}\label{eq:tau_0_0_1}
\tau _0^{{0}} : {\modG_0}  \longrightarrow  \Aut(F) \quad \hbox{and} \quad
 \tau _0^{{1}}: {\modG_0}  \longrightarrow  \Aut({\AAlpha})
\end{equation}
giving the canonical actions of $\modG$ on $F$ and $\AAlpha$, respectively.
Besides, for any $m\geq 1$,  we  obtain two homomorphisms
\begin{equation}\label{eq:tau_m_0_1}
\tau _m^{{0}} : {\modG_m}  \longrightarrow  Z^1\big(F, \overline{\Alpha}_m \big)
\quad \hbox{and} \quad  \tau _m^{{1}} : {\modG_m}  \longrightarrow  \Hom \big({\AAlpha}, \overline{\Alpha}_{m+1} \big),
\end{equation}
which maps any $g\in \modG_m$ to 
$$
\big([x]_0 \longmapsto [g(x)x^{-1}]_{m}\big) \quad \hbox{and} \quad \big([a]_1 \mapsto [g(a)a^{-1}]_{m+1}\big),
$$
respectively.
Note that, for all $m\geq 0$, we have
\begin{equation} \label{eq:ker_tau_m_0_1}
\ker \tau _m^{0} =  {\modG^0_{m+1}} \quad  \hbox{and} \quad \ker \tau _m^{1} = {\modG^1_{m+1}}  = {\modG_{m+1}}.    
\end{equation}

In the sequel, the homomorphisms $\tau^0_i$ and 
$\tau^1_i$ (for $i\geq 0$) 
are called the \emph{Johnson homomorphisms}
for the automorphism group of 
the  free pair $(\pi,\Alpha)$.

\begin{remark} \label{rem:00}
The two sequences  of homomorphisms, 
$(\tau^0_i)_{i\geq 0}$ and $(\tau^1_i)_{i\geq 0}$,
are defined on the successive terms of the filtration $\modG_*$, rather than on $\modG^0_*$. Nonetheless,
\begin{enumerate} 
\item if  restricted to $\modG^0_{m+1} \subset \modG_m$, the homomorphism $\tau^1_{m}$ takes values in 
$\Aut_{\Z[F]}(\AAlpha)$ \big(resp$.$ in $\Hom_{\Z[F]}  \big({\AAlpha}, \overline{\Alpha}_{m+1} \big) $\big)
for $m=0$ (resp$.$ for $m\geq 1$),
\item and the homomorphism $\tau_m^0$ extends on $\modG_m^0\supset \modG_m$ to a homomorphism (resp$.$ to a $1$-cocycle)  for  $m\geq 2$ (resp$.$ for $m=1$):
$$
\xymatrix{
\modG^0_{m} \ar@{-->}[rr]^-{\ti\tau_m^0}  &&  Z^1 \big(\pi, \overline{\Alpha}_m \big)  \\
{\modG_{m}} \ar[rr]^-{\tau_m^0}  \ar@{^{(}->}[u]&&  Z^1 \big(F, \overline{\Alpha}_m  \big). \ar[u]
}
$$
\end{enumerate}
See \cite[Proposition 10.5]{HM18} and \cite[Proposition 10.6]{HM18}.
\end{remark}

\begin{remark}
When $\Alpha=\pi$, 
we have $\modG^0_{m+1} =\modG_m$ and $\tau^0_m$ is trivial for all ${m\geq 0}$.
Then, we get the usual theory of Johnson homomorphisms for the automorphism group $\Aut(\pi)$ of the free group $\pi$.
See \cite{Satoh} for a survey.
\end{remark}

For the reader's convenience, we now write down the various properties of the 
Johnson homomorphisms   $(\tau^0_i)_{i\geq 0}$ 
and $(\tau^1_i)_{i\geq 0}$ that are directly inherited 
from the properties of the Johnson morphism $\overline{\tau}_\bu$:
\begin{itemize}
\item For any $g\in \modG_0$,
the two components $\tau^0_0(g),\tau^1_0(g)$ of $\overline{\tau}_0(g)$ are related  by 
 \begin{equation}
  \label{eq:equiv}
\tau^1_0(g)({}^xa) = {}^{\tau^0_0(g)(x)}\big(\tau^1_0(g)(a)\big)
\quad \hbox{for all $x\in F$, $a \in \AAlpha$}
\end{equation}
and, for any $g\in \modG_m$ and $m\geq 1$, the two components $\tau^0_m(g),\tau^1_m(g)$  of $\overline{\tau}_m(g)$ are related by
\begin{equation}
\label{eq:non-equiv}
\tau^1_m(g)({}^xa) = {}^x\big(\tau^1_m(g)(a )\big)+ \big[\tau^0_m(g)(x),{}^xa\big]
\quad \hbox{for all $x\in F$, $a \in \AAlpha$}.
\end{equation}
(Identity \eqref{eq:equiv} corresponds 
to the equivariance of  the automorphism  $\overline{\tau}_0(g)$ of the extended graded Lie algebra  $\overline{\Alpha}_\bu$,
and identity \eqref{eq:non-equiv} is the defect of $F$-equivariance of the degree $m$ derivation $\overline{\tau}_m(g)$ of the extended graded Lie algebra  $\overline{\Alpha}_\bu$.)
\item For any $g\in \modG_m\ (m\geq 1)$ and $h \in \modG_n\ (n\geq 1)$,  we can compute $t_0:=\tau^0_{m+n}([g,h])$ 
as well as $t_1:=\tau^1_{m+n}([g,h])$  from  $d_i:=\tau^i_{m}(g),i\in\{0,1\}$ and $e_j:=\tau^j_{n}(h), j\in \{0,1\}$ by 
\begin{eqnarray}
\label{eq:bracket0}
t_0(x)&=&d_n(e_0(x))-e_m(d_0(x)) - [d_0(x),e_0(x)] \quad \hbox{for all } x\in F,\\
t_1(a) &=& d_{n+1}(e_1(a))- e_{m+1}(d_1(a)) \quad \hbox{for all } a\in \AAlpha,
\label{eq:bracket1}
\end{eqnarray}
where $d_+=(d_i)_{i\geq 1}$ and  $e_+=(e_i)_{i\geq 1}$ 
denote the derivations of $\Lie({\AAlpha})= \overline{\Alpha}_+$
that extend $d_1$ and $e_1$, respectively. 
(These identities correspond to the fact that $\overline{\tau}_+$ is a homomorphism of graded Lie algebras.)
\item For all $g\in \modG$ and  $h\in \modG_m \ ( m\geq 1)$, we have
\begin{eqnarray}
\label{eq:equiv0_bis}
  \tau_m^0({}^g h) &=& \Lie_m(\tau_0^1(g)) \circ \tau_m^0( h) \circ \tau_0^0(g)^{-1}, \\
  \label{eq:equiv1_bis}
\tau_m^1({}^g h) &=& \Lie_{m+1}(\tau_0^1(g)) \circ \tau_m^1(h) \circ \tau_0^1(g)^{-1},    
\end{eqnarray} 
where $\Lie(\tau_0^1(g))$ is the automorphism of $\Lie(\AAlpha)= \overline{\Alpha}_+$
induced by $\tau_0^1(g) \in \Aut(\AAlpha)$.
(These identities  correspond to the equivariance of  $\overline{\tau}_+$ over $\overline{\tau}_0$.)
\end{itemize}

As a first application of the Johnson homomorphisms 
for the automorphism group
$\modG= \Aut(\pi, \Alpha)$
of the free pair $(\pi, \Alpha)$,
we mention the following property of the group 
$
\modG_1=\ker \big(\modG \to \Aut(\AAlpha) \big).
$

\begin{theorem} \label{th:residual}
The group $\modG_1$
is residually torsion-free nilpotent.
\end{theorem}

\begin{proof}
Recall that the \emph{rational lower central series} 
$$
G = \Gamma^\Q_1 G \geq
\Gamma^\Q_2 G \geq \Gamma^\Q_3 G \geq \cdots
$$ 
of a group $G$
is defined by $\Gamma^\Q_k G
:=\big\{g \in G : 
\exists n\geq 1,\, f^n \in \Gamma_k G\big\}$,
and that the residual torsion-free nilpotency of $G$
is equivalent to the triviality of 
the intersection of $(\Gamma_k^\Q G)_{k\geq 1}$.

Since the Johnson filtration of $\modG$ is an extended N-series,
it restricts to an N-series on $\modG_1$. 
Therefore, we have $\Gamma_k \modG_1 \subset \modG_k$
for every $k\geq 1$. Furthermore, 
since $\modG_{m+1}= \modG_{m+1}^1 =\ker \tau_m^1$ and 
since the target of $\tau_m^1$ is torsion-free as an abelian group
(for $\overline{\Alpha}_{m+1}=\Lie_{m+1}(\AAlpha)$ to be so),
we obtain by an induction on $m\in \{1,\dots,k-1\}$
that $\Gamma^\Q_k \modG_1 \subset \modG_{m+1}$.
Thus,  $\Gamma^\Q_k \modT  \subset \modG_k$
and we deduce from \eqref{nested} that
$
\bigcap_{k\geq 1} \Gamma^\Q_k \modT
$
is trivial.
\end{proof}

\subsection{Formality of free pairs} 

We shall prove that the extended N-series $\Alpha_*$ is formal 
in the sense of \S \ref{subsec:formal_eN-series}. For this,
we work over $\Q$.

Since $\overline{\Alpha}_+^\Q$ is the free Lie $\Q$-algebra on $\AAlpha^\Q:= \AAlpha \otimes \Q$,
the universal enveloping algebra  of the extended graded Lie algebra $\overline{\Alpha}_\bu^\Q$ 
is 
$$
U\big(\overline{\Alpha}_\bu^\Q\big) 
= U\big({\overline{\Alpha}_+^\Q}\big) \otimes_\Q \Q[F] 
= T(\AAlpha^\Q) \otimes_\Q \Q[F] ,
$$
where $T(\AAlpha^\Q)$ is the tensor algebra of $\AAlpha^\Q$.
(The multiplication in $U\big(\overline{\Alpha}_\bu^\Q\big)$
is given by \eqref{eq:mult_law}).
So, for degree-completions, we have
$$
\widehat{U}(\overline{\Alpha}_\bu^\Q) = T(\AAlpha^\Q)\, \hat \otimes_\Q\, \Q[F]  
:= \prod_{m\geq 0}\big( (\AAlpha^\Q)^{\otimes m} \otimes_\Q \Q[F] \big),
$$
which strictly contains 
$
\widehat T(\AAlpha^\Q)  \otimes_\Q \Q[F]  
= \big( \prod_{m\geq 0} (\AAlpha^\Q)^{\otimes m} \big) \otimes_\Q \Q[F].
$

\begin{definition} \label{def:expansion_free_pair}
An \emph{expansion} of  the free pair $(\pi,\Alpha)$ is a monoid homomorphism
$$
\theta: \pi \longrightarrow \widehat T(\AAlpha^\Q) \otimes_\Q \Q[F] 
$$
for which there is a map $\ell^\theta$ from  $\pi$ to the degree-completion of $\Lie(\AAlpha^\Q)$ such that
\begin{eqnarray} 
\label{eq:ell}
\theta(x) &=& 
\exp\big(\ell^\theta(x)\big) \otimes \varpi(x) \quad \hbox{for all $x\in\pi$},\\
\label{eq:ell_a}
\ell^\theta(a) &=& [a]_1 + (\deg \geq 2) \quad \hbox{for any $a\in \Alpha$}.
\end{eqnarray}
\end{definition}

From \cite[Lemma 11.1]{HM18}, which identifies 
the group-like part of $\widehat{U}(\overline{\Alpha}_\bu^\Q)$, it follows that 
an expansion of the free pair $(\pi,\Alpha)$ 
in the sense of Definition \ref{def:expansion_free_pair} is the same as
an expansion of the extended N-series~$\Alpha_*$ 
in the sense of \S \ref{subsec:formal_eN-series}.

\begin{lemma} \label{lem:expansion}
There exists an expansion $\theta$ of $(\pi, \Alpha)$.
In particular, the extended N-series $\Alpha_*$  is formal.
\end{lemma}

\begin{proof} 
Let $\{\al_i\}_{i\in I}\sqcup\{\be_j\}_{j\in J} $
be a basis of $\pi$ of type \eqref{eq:free_basis}.
We set $a_i=[\alpha_i]_1 \in \AAlpha$ for every $i\in I$,
and $x_j=[\beta_j]_0 \in F$ for every $j\in J$.
Let $\theta: \pi \to \widehat T(\AAlpha^\Q) \otimes \Q[F] $ be the  monoid homomorphism such that
$$
\theta(\alpha_i) = \exp(a_i) \otimes 1 = \sum_{k\geq 0} \frac{a_i^k}{k!} \otimes 1
\quad \hbox{and} \quad
\theta(\beta_j)= 1 \otimes x_j.
$$
For any $\ell,\ell'\in \Lie(\AAlpha^\Q)$ and $f,f'\in F$, we have
\begin{eqnarray*}
\big(\exp(\ell) \otimes f\big) \cdot \big(\exp(\ell') \otimes f'\big) 
&=& \big(\exp(\ell)  \exp({}^f\! \ell')\big) \otimes (ff')\\
&=& \exp\Big( \ell + {}^f\! \ell' + \frac{1}{2}[\ell, {}^f\! \ell'] + \cdots \Big) \otimes (ff').
\end{eqnarray*} 
Since $\pi$ is generated by $\alpha \sqcup \beta$, 
the BCH formula  implies that
there is a unique map $\ell^\theta$  satisfying \eqref{eq:ell}. 
Furthermore, we have
\begin{equation}  \label{eq:ppr}
\ell^\theta(xx') =  \ell^\theta(x) + {}^{\varpi(x)}\! \big(\ell^\theta(x')\big) 
+ \frac{1}{2}\big[\ell^\theta(x), {}^{\varpi(x)}\! \big( \ell^\theta(x')\big)\big] + 
\cdots
\end{equation}
for all $x,x'\in \pi$, 
where the terms not shown are Lie commutators of $\ell^\theta(x)$ and ${}^{\varpi(x)}\! \big( \ell^\theta(x')\big)$ of higher length.

Let $\widetilde \Alpha$ be the subset of $\Alpha$ consisting of the elements $a$ satisfying \eqref{eq:ell_a}.
Clearly $\widetilde \Alpha$ contains $\alpha$. By \eqref{eq:ppr},
$\widetilde \Alpha$ is stable under multiplication, hence under conjugation by elements of $\alpha$.
Besides, $\widetilde \Alpha$ is stable under conjugation by elements of $\beta$ since we have,
for any $a\in \widetilde{\Alpha}$ and $j\in J$,
\begin{eqnarray*}
\theta(\beta_j a \beta_j^{-1}) &=&  
(1 \otimes \varpi(\beta_j)) \cdot \big(\exp(\ell^\theta(a)) \otimes 1\big)  \cdot (1 \otimes \varpi(\beta_j)^{-1}) \\
&=& \exp\big( {}^{\varpi(\beta_j)}(\ell^\theta(a))  \big)  \otimes 1 \\
&=& \exp\big( {}^{\varpi(\beta_j)}([a]_1) + (\deg \geq 2) \big)  \otimes 1  
\end{eqnarray*}
so that $\beta_j a \beta_j^{-1}$ belongs to $ \widetilde{\Alpha}$. We conclude that $\widetilde \Alpha=\Alpha$.
\end{proof}

Recall that $\mathcal{G}_1$ is the subgroup of $\mathcal{G}$ acting trivially on the abelianization $\AAlpha$ of~$\Alpha$. 
The complete Lie algebra 
$\widehat D_+\big(\overline{\Alpha}^\Q_\bu\big)$ is viewed as a group
with multiplication given by the BCH formula.

\begin{theorem} \label{thm:varrho}
Let $\theta$ be an expansion of the free pair $(\pi,\Alpha)$. 
There is a group homomorphism
\begin{equation} \label{eq:varrho_free_pair}
\varrho^\theta: \mathcal{G}_1 \lto \widehat D_+\big(\overline{\Alpha}^\Q_\bu\big)
\end{equation}
whose restriction to $\mathcal{G}_m$ $(m\geq 1)$
starts in degree $m$ 
with $(\tau^0_m,\tau^1_m)$.
Furthermore, $\varrho^\theta$ is injective.
\end{theorem}

\begin{proof}
The first statement is a specialization of \eqref{eq:varrho} to the 
extended N-series $\Alpha_*$ for the free pair  $(\pi,\Alpha)$. Indeed,
 since $\overline{\Alpha}_+=\Lie(\AAlpha)$ is a torsion-free abelian group, 
there is no loss of information in considering the rational version  
$\overline{\tau}_+^\Q$ instead of $\overline{\tau}_+$ \cite[Remark 12.8]{HM18}.
Furthermore, we use here the ``truncation'' isomorphism
$\Der_+ (\overline{\Alpha}^\Q_\bu)  
\simeq D_+(\overline{\Alpha}^\Q_\bu)$.

The second statement is a direct consequence of \eqref{eq:completeness}.
\end{proof}

\section{First terms of the Johnson filtration for a free pair}
\label{sec:first_quotients}

Let $p,q\ge0$ be integers with either $p\ge2$ or $p,q\ge1$.
We assume that $\pi$ is a free group with basis 
$\alpha \sqcup \beta =\{\alpha_1,\dots, \alpha_p\}\sqcup\{\beta_1,\dots,\beta_q\}$  and that 
$$
\Alpha := \langle\!\langle \alpha_1,\dots, \alpha_p \rangle\!\rangle.
$$ 
Set $F:= \pi/\Alpha$ and let $\varpi:\pi \to F$ be the canonical projection.
Then $F$ is a free group with basis $(x_1,\dots, x_q)$, 
where $x_j:=\varpi(\beta_j)$ for each $j\in\{1,\dots,q\}$.
Hence $(\pi,\Alpha)$ is a free pair.
In this section, we compute the first few terms
of the Johnson filtration for  $\Aut(\pi,\Alpha)$.

\subsection{Magnus representations for a free pair}
\label{subsec:Magnus}

Let $\Z[\pi]$ be the group algebra of $\pi$. Let $\varepsilon: \Z[\pi] \to \Z$ be the augmentation, 
and $\overline{\ \cdot \ }: \Z[\pi] \to \Z[\pi]$ the linear map defined by $\overline{u}=u^{-1}$ for all $u\in \pi$.
Consider the (left) Fox derivatives with respect to the basis $(\alpha,\beta)$ of $\pi$, which are the linear maps 
$$
\frac{\partial \ }{\partial \alpha_i}: \Z[\pi] \longrightarrow \Z[\pi] \quad  (i=1,\dots,p)
\quad \hbox{and} \quad 
\frac{\partial \ }{\partial \beta_j}: \Z[\pi] \longrightarrow \Z[\pi] \quad (j=1,\dots,q)
$$
defined by
\begin{equation} \label{eq:Fox}
w-\varepsilon(w) = \sum_{i=1}^p \frac{\partial w}{\partial \alpha_i} (\alpha_i-1)
+ \sum_{j=1}^q \frac{\partial w}{\partial \beta_j} (\beta_j-1) \quad \hbox{for all $w\in \Z[\pi]$}.
\end{equation}

We associate to every $g\in \Aut(\pi)$ its \emph{free Jacobian matrix} with respect to the basis $\alpha \sqcup\beta$ of $\pi$:
$$
J(g) := 
{\scriptsize
 \begin{pmatrix} 
\overline{\frac{\partial g(\alpha_1)}{\partial \alpha_1} }& \cdots& \overline{\frac{\partial g(\alpha_p)}{\partial \alpha_{1}} } & \overline{\frac{\partial g(\beta_1)}{\partial \alpha_1} } & \cdots& \overline{\frac{\partial g(\beta_q)}{\partial \alpha_{1}}} \\
\vdots & &\vdots & \vdots & & \vdots\\
\overline{\frac{\partial g(\alpha_1)}{\partial \alpha_p}} & \cdots& \overline{\frac{\partial g(\alpha_p)}{\partial \alpha_{p}} } &\overline{ \frac{\partial g(\beta_1)}{\partial \alpha_p}} & \cdots&\overline{\frac{\partial g(\beta_q)}{\partial \alpha_{p}} } \\ 
&&&&&\\
\overline{\frac{\partial g(\alpha_1)}{\partial \beta_1} }& \cdots& \overline{\frac{\partial g(\alpha_p)}{\partial \beta_{1}} } & \overline{\frac{\partial g(\beta_1)}{\partial \beta_1} } & \cdots& \overline{\frac{\partial g(\beta_q)}{\partial \beta_{1}}} \\
\vdots & &\vdots & \vdots & & \vdots\\
\overline{\frac{\partial g(\alpha_1)}{\partial \beta_q}} & \cdots& \overline{\frac{\partial g(\alpha_p)}{\partial \beta_{q}} } &\overline{ \frac{\partial g(\beta_1)}{\partial \beta_q}} & \cdots&\overline{\frac{\partial g(\beta_q)}{\partial \beta_{q}} }  
\end{pmatrix}} \in \GL({p+q};\Z[\pi]).
$$

In the sequel, we use the notations of \S \ref{subsec_free_pairs}. In particular, set $\modG=\Aut(\pi,\Alpha)$.
Let $J^F: \modG \to \GL({p+q};\Z[F])$ be the map obtained by restricting $J:\Aut(\pi) \to \GL({p+q};\Z[\pi])$ to $\modG$  and by applying $\varpi$ to each matrix entry. 
The next proposition identifies the four blocks 
(namely $p\times p$, $p\times q$, $q\times p$ and $q \times q$)
of the corresponding  matrices.

\begin{proposition} \label{prop:blocks}
\begin{enumerate}
\item The lower-left  block of $J^F$ is trivial.
\item The lower-right block of $J^F$ defines a crossed homomorphism
$$
\Mag^0_0: \modG \longrightarrow \GL(q;\Z[F]), \quad
g \longmapsto \Bigg(\varpi\Big(\overline{\frac{\partial g(\beta_j)}{\partial \beta_i} }\Big)\Bigg)_{i,j}
$$
which is equivalent to $\tau_0^0: \modG \to \Aut(F)$
(in the sense that $\Mag_0^0$ factors through an injective crossed homomorphism $\Aut(F)\to \GL(q;\ZF)$).
\item The upper-left block of $J^F$ defines a crossed homomorphism
$$
\Mag^1_0: \modG \longrightarrow \GL(p;\Z[F]), \quad
g \longmapsto \Bigg(\varpi\Big(\overline{\frac{\partial g(\alpha_j)}{\partial \alpha_i} }\Big)\Bigg)_{i,j}
$$
which, with the knowledge of $\Mag^0_0$, 
is equivalent to $\tau_0^1: \modG \to \Aut(\AAlpha)$.
\item The upper-right block of $J^F$ restricts to a homomorphism
$$
\Mag^0_1: \modG_1 
\longrightarrow \Mat(p\times q;\Z[F]), \quad
g \longmapsto \Bigg(\varpi\Big(\overline{\frac{\partial g(\beta_j)}{\partial \alpha_i} }\Big) \Bigg)_{i,j}
$$
which is equivalent to $\tau^0_1: \modG_1 \to Z^1(F, \AAlpha)$
(in the sense that $\Mag^0_1$ factors through an injective homomorphism $Z^1(F,\AA)\to\Mat(p\times q;\ZF)$).
\end{enumerate}
\end{proposition}

The next lemma, which follows from elementary properties
of Fox derivatives,
is needed for the proof of Proposition \ref{prop:blocks}.

\begin{lemma} \label{lem:kappa}
The map $\kappa: \AAlpha \to \Z[F]^p$ defined by
\begin{equation}   \label{eq:kappa}
\kappa([a]) = \Big(\varpi\big(\frac{\partial a}{\partial \alpha_1} \big), \dots , \varpi\big(\frac{\partial a}{\partial \alpha_p} \big)\Big) 
\end{equation}
is an isomorphism of $\Z[F]$-modules, where $\Z[F]$ acts on $\Z[F]^p$ 
by left multiplication.
\end{lemma}

\begin{proof}
For any $a_1,a_2 \in \Alpha$, and for all $i\in \{1,\dots,g\}$, we have
$$
\varpi\big(\frac{\partial (a_1 a_2)}{\partial \alpha_i} \big) = 
\varpi\big(\frac{\partial a_1 }{\partial \alpha_i} \big)  + \varpi\big(a_1\frac{\partial  a_2}{\partial \alpha_i} \big) 
= \varpi\big(\frac{\partial a_1 }{\partial \alpha_i} \big)  + \varpi\big(\frac{\partial  a_2}{\partial \alpha_i} \big).
$$
This shows that the right side of \eqref{eq:kappa} defines a homomorphism $\Alpha \to \Z[F]^p$ and, 
since the group $\Z[F]^p$ is abelian, $\kappa$ is a well-defined homomorphism.

Let $i\in \{1,\dots,p\}$, $a\in \Alpha$ and $w\in \pi$. We have
\begin{eqnarray*}
\varpi\big(\frac{\partial (w a w^{-1})}{\partial \alpha_i} \big)
&=& \varpi\big(\frac{\partial w }{\partial \alpha_i} \big) + \varpi\big(w\frac{\partial a }{\partial \alpha_i} \big) 
+ \varpi\big(w a \frac{\partial  w^{-1}}{\partial \alpha_i} \big) 
  \ = \ \varpi(w) \varpi\big(\frac{\partial a} {\partial \alpha_i} \big), 
\end{eqnarray*}
which shows that the map $\kappa$ is $\Z[F]$-linear.

By Lemma \ref{lem:freeness}, the $\Z[F]$-module $\AAlpha$ is free on the classes of $\alpha_1,\dots,\alpha_p$.
This basis of  $\AAlpha$ 
is sent by $\kappa$ to the canonical basis of $\Z[F]^p$.
Hence $\kappa$ is an  isomorphism.
\end{proof}

\begin{proof}[Proof of Proposition \ref{prop:blocks}]
It is well known that $J$ is a crossed homomorphism, i.e.,
$$
J(gg') = J(g)\cdot g\big(J(g')\big) \quad \hbox{for all $g,g' \in \Aut(\pi)$},
$$
see e.g$.$ \cite[\S 3]{Birman}. 
Hence, we get
\begin{equation} \label{eq:product}
J^F(gg') = J^F(g)\cdot g^F\big(J^F(g')\big) \quad \hbox{for all $g,g' \in \modG$},
\end{equation}
where we set $g^F:= \tau_0^0(g)\in \Aut(F)$  for all $g\in \modG$.

Let $g\in \modG$. Then $g(\alpha_i)\in \Alpha$ for all $i\in \{1,\dots,p\}$.
Thus, statement (1) follows since, for any $a \in \Alpha$ 
and for all $j \in \{1,\dots,q\}$, we have
\begin{equation} \label{eq:a}
\varpi\Big({\frac{\partial a}{\partial \beta_j}} \Big)=0.
\end{equation}
To prove \eqref{eq:a}, observe that,
 for any $a',a'' \in \Alpha$, we have
$$
\varpi\Big({\frac{\partial a' a''}{\partial \beta_j}} \Big)
= \varpi\Big({\frac{\partial a'}{\partial \beta_j}} + a' {\frac{\partial a''}{\partial \beta_j}}  \Big)=
 \varpi \Big({\frac{\partial a'}{\partial \beta_j}}\Big) + \varpi \Big({\frac{\partial a''}{\partial \beta_j}}  \Big);
$$
thus, it is enough to check \eqref{eq:a} 
for $a= w \alpha_k w^{-1}$, where $k \in \{1,\dots,p\}$ and $w\in \pi$:
\begin{eqnarray*}
\varpi\Big({\frac{\partial w \alpha_k w^{-1}}{\partial \beta_j}} \Big)
&= & \varpi\Big({\frac{\partial w }{\partial \beta_j}} \Big) + \varpi(w) \, \varpi\Big({\frac{\partial  \alpha_k w^{-1}}{\partial \beta_j}} \Big) \\
&= & \varpi\Big({\frac{\partial w }{\partial \beta_j}} \Big) + \varpi(w) \, \varpi\Big({\frac{\partial  w^{-1}}{\partial \beta_j}} \Big) \ = \ 
\varpi\Big({\frac{\partial w w^{-1}}{\partial \beta_j}} \Big) \ = \ 0.
\end{eqnarray*}

We deduce from statement (1) and from \eqref{eq:product} that the two diagonal blocks of $J^F$
define crossed homomorphisms with values in  general linear groups.

Let $g\in \modG$. Then  $g^F=\tau_0^0(g)\in \Aut(F)$ is determined by its free Jacobian matrix 
with respect to the basis $\{x_1,\dots,x_q\}$.
Thus, statement (2) will follow since the latter 
is equal to the lower-right block of $J^F$.  Indeed, for all  $r\in \{1,\dots, q\}$, we have
$$
g(\beta_r)-1=
\sum_i \frac{\partial g(\beta_r)}{\partial \alpha_i} (\alpha_i-1) + \sum_j \frac{\partial g(\beta_r)}{\partial \beta_j} (\beta_j-1).
$$
Therefore, by applying $\varpi$, we get 
$$
g^F(x_r)-1 = \sum_j \varpi\Big(\frac{\partial g(\beta_r)}{\partial \beta_j}\Big)(x_j-1)
$$
which implies that 
$$
\varpi\Big(\frac{\partial g(\beta_r)}{\partial \beta_j}\Big) = \frac{\partial g^F(x_r)}{\partial x_j} \
\hbox{for all $j\in\{1,\dots,q\}$.}
$$

Let $g\in \modG$.
Since the pair $(\tau^0_0(g),\tau^1_0(g))$ satisfies \eqref{eq:equiv}, 
the automorphism $\tau^1_0(g)$ is determined by its values on the basis 
$\{[\alpha_1],\dots, [\alpha_p]\}$ and by $\tau^0_0(g)$. The converse is also true since $\AAlpha$ is torsion-free as a $\Z[F]$-module.
Besides, Lemma~\ref{lem:kappa} implies that the upper-left block of $J^F(g)$
contains as much information as $[g(\alpha_1)], \dots, [g(\alpha_p)] \in \AAlpha$.
Thus we have proved statement (3).

We conclude with the proof of statement (4). First observe that the diagonal blocks of $J^F$ on $\modG_1$ are identity blocks
since $\modG_1= \ker \tau^1_0$. Thus, the upper-right block of $J^F: \modG_1 \to \GL({p+q};\Z[F])$ is a homomorphism.
Let $g\in \modG_1$. The $1$-cocycle $\tau^0_1(g)$ is determined by its values on the basis $\{x_1,\dots,x_q\}$ of $F$.
For each $j\in \{1,\dots,q\}$, the value of $\tau^0_1(g)$ on $x_j$ is 
$[g(\beta_j)\beta_j^{-1}] \in \AAlpha$ and, for all $i\in \{1,\dots,p\}$,
$$
\varpi\Big(\frac{\partial (g(\beta_j)\beta_j^{-1})}{\partial \alpha_i} \Big)
= \varpi\Big(\frac{\partial g(\beta_j)}{\partial \alpha_i} \Big).
$$
Hence, according to Lemma \ref{lem:kappa}, 
the value of $\tau^0_1(g)$ on $x_j$ is encoded by (the conjugate of)
the $j$-th column of the upper-right block of $J^F(g)$.
\end{proof}

\subsection{The quotient $\G/\G_1^0$}

Let us consider 
the first four terms of the filtration \eqref{nested}:
\begin{equation} \label{eq:start_filtration}
   \modG=\modG_0 \geq \modG_1^0 \geq \modG_1 \geq \modG_2^0 
\end{equation}
The homomorphism $\tau_0^0: \modG \to \Aut(F)$
is split surjective 
since any automorphism of $\la \beta_1,\dots, \beta_q \ra \simeq F$ extends
to an automorphism of $(\pi,\Alpha)$ that fixes each of $\alpha_1,\dots, \alpha_p$.
Hence $\tau_0^0$ induces an isomorphism $\modG/\modG_1^0\simeq \Aut(F)$.
Also, we can identify
$\modG/\modG_1^0$ with a subset of $\GL(q;\Z[F])$ 
via the crossed homomorphism~$\Mag_0^0$. 

In the rest of this section, 
we study the next two successive quotients in \eqref{eq:start_filtration}.

\subsection{The quotient $\G^0_1/\G_1$}

Let us now consider the quotient $\modG_1^0/\modG_1$.
It embeds into $\Aut_{\Z[F]}(\AAlpha)$ via $\tau_0^1$
(see Remark \ref{rem:00}).
Equivalently, $\modG_1^0/\modG_1$ embeds into $\GL(p;\Z[F])$ via $\Mag_0^1$. 
Consider the homomorphism$$r:=\Mag^1_0|_{\G^0_1}:\G^0_1\to\GL(p;\ZF).$$
We have 
$
\ker(r) = \ker(\tau_0^1) \cap \G^0_1 = \G_1^1 \cap \G^0_1 =\G_1.
$
Thus, we have an exact sequence
  \begin{gather}
    \label{e47}
    1\longrightarrow\G_1\longrightarrow\G^0_1\overset{r}{\longrightarrow}\GL(p;\ZF).
  \end{gather}
Now, it remains to identify the image of $r$
(or, equivalently, the image of $\tau_0^1$ in 
$\Aut_{\Z[F]}(\AAlpha) \simeq \GL(p;\ZF)$).
We give below a partial answer.

Let $R$ be a (associative, unital) ring, and let $U(R)$ be its group of units.
The {\em general elementary subgroup} $\GE(p;R)$ of $\GL(p;R)$ 
is generated by  invertible diagonal matrices and elementary matrices, i.e., it is  generated by
\begin{itemize}
\item $d_i(u):=I_p+(u-1)E_{ii}$ for $i\in \{1,\ldots ,p\}$ 
and $u\in U(R)$,
\item $e_{ij}(w):=I_p+wE_{ij}$ for $i,j\in \{1,\ldots ,p\}$, $i\neq j$, and
  $w\in R$.
\end{itemize}
Following Cohn \cite{Cohn}, a ring $R$ is said to be 
\emph{generalized euclidean} if $\GE(p;R)=\GL(p;R)$ for every $p\geq 2$.
Of course, generalized euclidean rings include ordinary euclidean rings
(as a consequence of the Gauss algorithm).
We are interested here in the ring $R:=\Z[F]$
for a free group $F$ of rank $q$, but it does not seem
to be known whether it is generalized euclidean or not
(even in the case $q=1$ \cite{Guyot}).

\begin{remark}
The group algebra $\Q[F]$ with coefficients 
in $\Q$ is generalized euclidean \cite[Theorem 3.4]{Cohn}, 
but  free associative rings  (including the $1$-variable polynomial ring $\Z[X]$) are not \cite[end of \S 8]{Cohn}.  
\end{remark}

The following gives a partial information on $r(\G^0_1)\simeq \G^0_1/\G_1$.

\begin{proposition}
\label{r59}
We have \
$\GE(p;\ZF)\le r(\G^0_1)\le\GL(p;\ZF).$
\end{proposition}

In the rest of this section,  we use the following notations
for endomorphisms of the free group
$\pi=\la\al_1,\dots,\al_p,\be_1,\dots,\be_q\ra$.  For distinct
generators $u_1,\dots,u_r\in\{\al_i\}_i \sqcup \{\be_j\}_j$ and elements
$v_1,\dots,v_r\in\pi$, let $(u_1\mapsto v_1,\dots,u_r\mapsto v_r)$
denote the endomorphism of $\pi$ sending $u_i$ to $y_i$ for
$i=1,\dots,r$ and each of the other generators to itself.

\begin{proof}[Proof of Proposition \ref{r59}]
Recall that  $U(\ZF)=\pm F$ 
by a classical result of Higman~\cite{Higman}.
Hence $\GE(p;\ZF)$ is
generated by the matrices  $d_i(\epsilon x)$ with $\epsilon=\pm 1$ and $x\in F$ 
and the matrices $e_{ij}(x)$ with $x\in F$. All these matrices belong to the image of $r$:
\begin{itemize}
    \item  we have $r\big(\ti d_i(\epsilon x)\big)=d_i(\epsilon x)$, where $\ti
  d_i(\epsilon x):=\big(\al_i\mt{}^{x^{-1}}\al_i^\epsilon\big)\in\G^0_1$
  (here and below, we are identifying $x\in F$ with a lift in $B\le\pi$);
  \item we have $r(\ti e_{ij}(x))=e_{ij}(x)$, where $\ti
  e_{ij}(x):=\big(\al_j\mt ({}^{x^{-1}}\al_i)\, \al_j\big)\in\G^0_1$. 
\end{itemize}
\end{proof}

\begin{remark}
  \label{r2}
For any group $G$, the general elementary groups $\GE(p;\Z[G])$ 
are used in the definition of the \emph{Whitehead group}:
\begin{gather*}
  \hbox{Wh}(G)=\ilim_p\GL(p;\Z[G])/\ilim_p\GE(p;\Z[G]).
\end{gather*}
According to \cite{Stallings}, the Whitehead group $\hbox{Wh}(F)$ of the free group $F=F_q$ 
of rank~$q$ is trivial. 
Hence the surjectivity of $r=r_{p,q}$ holds stably in $p$, i.e., for any fixed
$q\ge 0$, the inductive limit
\begin{gather*}
  \ilim_p r_{p,q}:\ilim_p (\G_{p,q})^0_1\rightarrow \ilim_p \GL(p;\Z [F_q]),
\end{gather*}
is surjective.
\end{remark}

\subsection{The quotient $\modG_1/\modG^0_2$}

We now identify 
$\modG_1/\modG^0_2$.

\begin{proposition}
  \label{r73}
  We have the short exact sequence
  \begin{equation} \label{eq:tau^0_1}
    1\longrightarrow\G^0_2\longrightarrow
    \G_1\overset{\tau^0_1}{\longrightarrow }Z^1(F,\AAlpha)\longrightarrow 1
  \end{equation}
  or, equivalently, we have the short exact sequence
   \begin{equation} \label{eq:Mag^0_1}
    1\longrightarrow\G^0_2\longrightarrow
    \G_1\overset{\Mag_1^0}{\longrightarrow } \Mat(p\times q;\Z[F]) \longrightarrow 1.
  \end{equation}
  Thus we have
  $ \ 
  \G_1/\G^0_2\simeq Z^1(F,\AAlpha)\simeq\Mat(p\times q;\Z[F]).
  $
\end{proposition}

\begin{proof}
The equivalence of \eqref{eq:tau^0_1} and 
\eqref{eq:Mag^0_1} follows directly from the
sequence of isomorphisms
\begin{equation} \label{eq:sequence_iso}
Z^1(F,\AAlpha) \simeq \AAlpha^q \simeq (\Z[F]^p)^q = 
\Mat(p\times q;\Z[F])    
\end{equation}
through which $\tau^0_1$ corresponds to $\Mag^0_1$.
(See the proof of Proposition \ref{prop:blocks}.(4).)
Hence it suffices to prove that  $\tau^0_1$ is surjective.

  For $i\in\{1,\dots,q\}$, $s\in\{1,\dots,p\}$, $b\in \la \beta_1,\dots,\beta_q \ra$, set
  \begin{gather*}
    \varphi_{i,s,b}:=(\al_s\mt{}^b\al_s)\circ(\be_i\mt\al_s\be_i)\circ(\al_s\mt{}^{b}\al_s)^{-1}.
  \end{gather*}
  Since $(\be_i\mt \al_s \be_i)\in\G_1$ and $(\al_s\mt
  {}^b\al_s)\in\G$, we have $\varphi_{i,s,b}\in\G_1$.
  Observe that $\varphi_{i,s,b}(\be_i)={}^b\al_s\be_i$, and
  $\varphi_{i,s,b}(\be_j)=\be_j$ for $j\neq i$.
  Hence, we have
  \begin{gather*}
    \tau^0_1(\varphi_{i,s,b})(x_j)=
    \left[({}^b\al_s)^{\delta _{j,i}}\right]
    \in \AA.
  \end{gather*}
  Since the $\tau^0_1(\varphi_{i,s,b})$ form a basis of the free abelian
  group $Z^1(F,\AAlpha)$, it follows that $\tau^0_1$ is surjective.
\end{proof}

\section{The Johnson filtration for the handlebody group}
\label{sec:Johnson_handlebody}

We now review some basic facts about the handlebody group,
and we start the study of the Johnson filtration in this situation.

\subsection{The handlebody group $\modH$}

Let $V$ be a handlebody of genus $g$.
Let $D$ be a disk in $\partial V$, and
set $\Sigma := \partial V \setminus \operatorname{int}(D)$. 
Let $\modH$ be the mapping class group of $V$ rel $D$, 
which is called the  \emph{handlebody group}.
Let $\modM$ 
be the mapping class group of $\Sigma$ rel $\partial \Sigma$. 
By restricting self-homeomorphisms of $V$ to $\Sigma$, we obtain  an injective homomorphism
$\modH  \to \modM$ (see e.g$.$ \cite[\S 3]{Hensel}).
Thus we  regard $\modH$ as a subgroup of~$\modM$.

Choose a base point $\star \in \partial D$, and define the \emph{twist group}
$$
\modT := \ker \big(\modH \longrightarrow \Aut( \pi_1(V,\star))\big)
$$ 
as the subgroup of $\modH$ acting trivially on
the fundamental group of $V$.
A \emph{disk-twist} is the isotopy class of self-homeomorphisms of $V$
defined by twisting $V$ along a properly embedded disk $U$ in $V$
(see e.g$.$ \cite[\S 5]{Hensel}).
The corresponding element of $\modH \subset \modM$ is the Dehn twist $T_{\partial U}$
along the simple closed curve $\partial U$,
which is a \emph{meridian} of $V$.
Thus, disk-twists are also called ``meridional Dehn twists'' in the literature.

In the following, we derive two results on the handlebody group $\modH$ of $V$ rel $D$ from results of Luft \cite{Luft} and Griffiths \cite{Griffiths} on the handlebody group of $V$ rel~$\star$.

\begin{theorem}[Luft] \label{thm:Luft}
The subgroup $\modT$ of $\modH$ is generated by disk-twists.
In other words, the subgroup $\modT$ of $\modM$ is generated by Dehn twists along meridians
(i.e$.$ simple closed curves of $\Sigma$  null-homotopic in $V$).
\end{theorem}

\begin{proof}
This is derived from Luft's theorem  \cite{Luft}.
Let $\modT_\circ$ be the subgroup of $\modH$ generated by disk-twists.
Since any disk-twist acts trivially on $\pi_1(V,\star)$, we have $\modT_\circ \subset \modT$. It remains to prove the converse inclusion.

Let $\hat\modM$ and $\hat\modH$ be the mapping class groups of $\partial V$ and $V$, respectively, rel $\star$.
Let $\hat\modT$ denote the twist group for $V$ rel $\star$, which is defined 
as the kernel of the canonical map $\hat\modH \to \Aut( \pi_1(V,\star))$.
By \cite[Cor. 2.2]{Luft}, $\hat\modT$
coincides
with the subgroup $\hat\modT_\circ$  of $\hat\modH$ generated by disk-twists. 
There is a short exact sequence
\begin{equation} \label{eq:simple_ses}
 1\to   \Z {\longrightarrow }\modM 
 \overset{p}{\longrightarrow} \hat\modM\rightarrow 1,
\end{equation}
where $1\in \Z$ is mapped to the Dehn twist $T_\zeta$ 
along the boundary curve $\zeta=\partial \Sigma$,
and $p$ is the canonical map.
Let $f\in  \modT$. Clearly $p(f)\in \hat\modT$ so that  $p(f)\in \hat\modT_\circ$.
Thus, for some $h\in \modT_\circ$ we have $p(h)=p(f)$.
Hence $f=h\, T_\zeta^{n}$ for some $n\in \Z$.
Since $T_\zeta$ is a disk-twist, we have $f\in  \modT_\circ$.
\end{proof}

We regard $\modH$ and $\modT$
as subgroups of the automorphism group $\Aut(\pi)$ of the free group $\pi:= \pi_1(\Sigma,\star)$,
via the Dehn--Nielsen representation $\modM \to \Aut(\pi)$.
Set $F:= \pi_1(V,\star)$ and
let $\varpi: \pi \to F$ be induced 
by the inclusion $\iota: \Sigma \hookrightarrow V$. 
Then, for $\Alpha := \ker \varpi$, we have the free pair $(\pi, \Alpha)$ and we can consider the subgroup $\modG:= \Aut(\pi,\Alpha)$ of $\Aut(\pi)$.
Note that $\A$ is isomorphic to the relative homotopy group $\pi_2(V,\Sigma)$.

\begin{theorem}[Griffiths] \label{thm:Griffiths}
We have $\modH = \modM \cap \modG$. 
\end{theorem}

\begin{proof}
This is derived from Griffiths' theorem \cite{Griffiths}.
Here we use the notations in the proof of Theorem~\ref{thm:Luft}.
The inclusion $\modH \subset \modM \cap \modG$ follows
easily from the functoriality of the fundamental group.

To prove the converse inclusion, let $f\in \modM \cap \modG$
and $\hat f = p(f)\in \hat\modM$.
As an automorphism of $\pi_1(\partial V,\star)$, 
$\hat f$ preserves the kernel of 
$\iota_*: \pi_1(\partial V,\star) \to \pi_1(V,\star)$.
Thus we have $\hat f\in\hat\modH$ by \cite[Cor. 10.2]{Griffiths}. 
Hence for some $h\in \modH$ we have $p(h)=\hat f$.
By the short exact sequence \eqref{eq:simple_ses}, we have 
$f=h\, T_\zeta^n$ for some $n \in \Z$, hence $f\in \modH$.
\end{proof}

\subsection{The Johnson filtration of $\modH$}

By Theorem \ref{thm:Griffiths} we can apply to $\modH$ the constructions and results  of the previous sections for $\modG=\Aut(\pi,\Alpha)$.

In particular, 
by restricting the filtrations \eqref{nested} of $\modG$ to~$\modH$, 
and setting $\modH_m=\modH\cap\G_m$ and $\modH^0_m=\modH\cap\G^0_m$ 
for $m\geq 0$,
we obtain two nested filtrations
\begin{equation} \label{eq:HH}
\modH=\modH^0_0=\modH_0\ge \modH^0_1\ge \modH_1\ge \cdots \ge \modH_{m-1} \ge \modH^0_m\ge \modH_m\ge \cdots  
\end{equation}
such that
\begin{equation} \label{eq:completeness_HH}
\bigcap_{m\geq 0} \modH_m^0
= \bigcap_{m\geq 0} \modH_m= \{1\}   . 
\end{equation}
It turns out that these  two filtrations coincide.

\begin{theorem} \label{thm:=}
For each $m\geq 0$, we have $\modH^0_m= \modH_m$.
\end{theorem}

To prove Theorem \ref{thm:=}, we need the identification of $\modM$ 
with the subgroup of $\Aut(\pi)$ fixing the homotopy class
$$
\zeta := [\partial \Sigma]
$$
of the boundary curve. 
We write $\zeta$ in  an explicit basis of the free group~$\pi$ as follows.
Let $g\geq 1$ be the genus of $\Sigma$, and fix a system 
\begin{equation} \label{eq:(a,b)}
    (\alpha,\beta)=(\alpha_1,\dots,\alpha_g,\beta_1,\dots,\beta_g)
\end{equation} 
of meridians and parallels on the surface $\Sigma$:   \label{page:meridians_parallels}
$$
\labellist \scriptsize \hair 0pt 
\pinlabel {\large $\star$} at 567 3
\pinlabel {$\alpha_1$} [l] at 310 105
\pinlabel {$\alpha_g$} [l] at 568 65
\pinlabel {$\beta_1$} [b] at 216 80
\pinlabel {$\beta_g$} [b] at 472 69
\endlabellist
\includegraphics[scale=0.4]{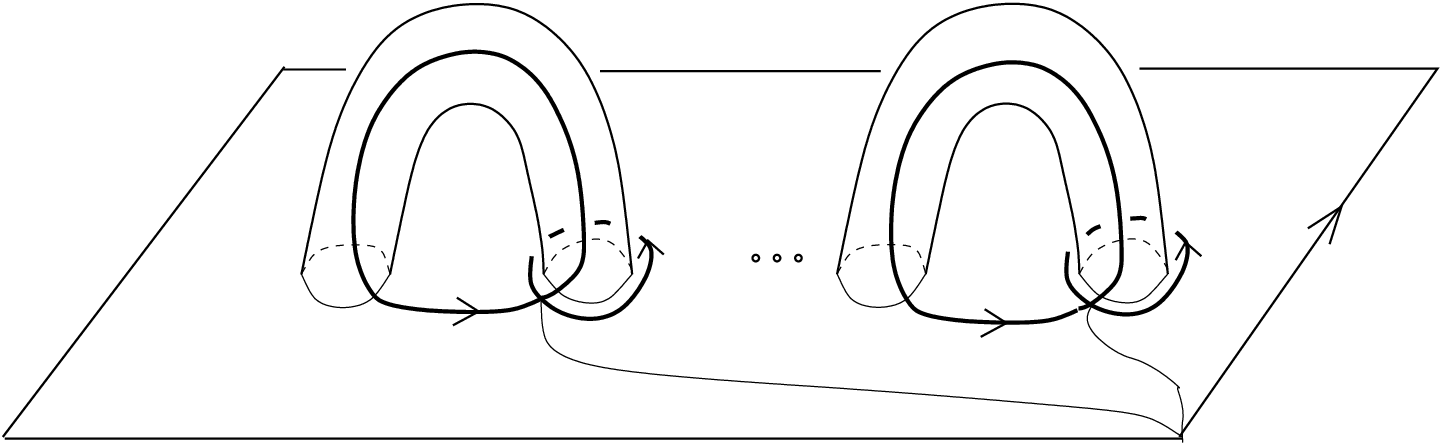}
$$
Here the curves $\alpha_1,\dots,\alpha_g$ bound pairwise-disjoint embedded disks in $V$.
The  basis of $\pi$  defined by $(\alpha,\beta)$ and the connecting arcs to $\star$ shown above
is still denoted by $(\alpha, \beta)$. Then we have
\begin{equation} \label{eq:zeta}
\zeta^{-1}= \prod_{i=1}^g \big[\beta_i^{-1},\alpha_i\big].    
\end{equation}

\begin{proof}[Proof of Theorem \ref{thm:=}]
  \def\equ{\underset{\A_{m+1}}{\equiv}}
It suffices to prove $\modH^0_m\subset \modH^1_m$.
  Thus, for $f\in \modH^0_m=\modH\cap\G^0_m$, 
  we need to check that
  \begin{equation}\label{eq:f(a)a^{-1}}
      f(a)a^{-1}\in\A_{m+1} \quad \hbox{for all $a\in \Alpha$}.
  \end{equation}
  Since $f \in \G^0_m$ and $\Alpha=\lala \alpha_1,\dots,\alpha_g\rara$, it suffices to verify \eqref{eq:f(a)a^{-1}}
  for $a=\alpha_i$  with $i=1,\ldots ,g$.
  Setting
  \begin{gather*}
    u_i := f(\al_i) \al_i^{-1}\in \A_m,\quad
    v_i := f(\be_i) \be_i^{-1}\in \A_m,
    \quad (i=1,\ldots ,g),
  \end{gather*}
  we have
  \begin{gather*}
    \begin{split}
      &f([\al_i,\be_i^{-1}])=[f(\al_i),f(\be_i)^{-1}]
      =[u_i\al_i,\be_i^{-1}v_i^{-1}]
      =[u_i\al_i,\be_i^{-1}] \cdot[u_i\al_i,v_i^{-1}]^{\be_i}\\
      &\equ [u_i\al_i,\be_i^{-1}] 
      = {}^{u_i}[\al_i,\be_i^{-1}] \cdot [u_i,\be_i^{-1}] 
      \equ [\al_i,\be_i^{-1}] \cdot [u_i,\be_i^{-1}]. 
    \end{split}
  \end{gather*}
Hence, by \eqref{eq:zeta}, we get
\begin{gather*}
    \begin{split}
      &f(\zeta^{-1})=f\left(\prod_{i=1}^g[\be_i^{-1},\al_i]\right)
      \equ\prod_{i=1}^g [\be_i^{-1},u_i] \cdot  [\be_i^{-1},\al_i]  \\
      &\quad \equ\left(\prod_{i=1}^g [\be_i^{-1},\al_i] \right) 
      \cdot\left(\prod_{i=1}^g  [\be_i^{-1},u_i] \right)
      =  \zeta^{-1} \cdot \prod_{i=1}^g [\be_i^{-1},u_i].
    \end{split}
  \end{gather*}
It follows that
 \begin{gather*}
    1\equ\prod_{i=1}^g[\be_i^{-1},u_i]
    =\prod_{i=1}^g   u_i{}^{\be_i} \cdot u_i^{-1}.
  \end{gather*}
Thus, we obtain
  \begin{gather*}
    \sum_{i=1}^g \big(x_i^{-1}-1\big)[u_i]_m=0 \in \ol\A_m,
  \end{gather*}
where $x_i=[\beta_i]\in F=\pi/\Alpha$
and we use the $\Z[F]$-module structure of $\ol\A_m$.
Since the free group $F$ is left-orderable, 
we can apply Lemma \ref{lem:free_Lie} below to deduce that 
the $\Z[F]$-module $\ol\A_m= \Lie_m(\AAlpha)$ is free.
 Then, it follows from  Lemma~\ref{r54} below 
 that $[u_i]_m=0\in \ol\A_m$ for all $i$, i.e.,
  $u_i\in \A_{m+1}$ for all $i$.
\end{proof}

The rest of this section is devoted to some lemmas used in the proof of Theorem~\ref{thm:=}.
We shall need the following definitions.

The free magma $\Mag(S)$ on a set $S$ consists of non-associative words in the alphabet $S$.
Let $|w|$ denote the length of such a word $w$.
For instance, if $x,y\in S$, then $(x,((x,y),y)) \in \Mag(S)$
and $\vert (x,((x,y),y))  \vert=4$.
A \emph{Hall set} on the alphabet~$S$ \cite[\S 4.1]{Reutenauer}
is a totally-ordered subset $H \subset \Mag(S)$ containing $S$ and satisfying the following two conditions:
\begin{itemize}
\item[$(H_1)$] for any word $h=(h',h'')$ in $H \setminus S$, we have $h''\in H$ and $h<h''$;
\item[$(H_2)$] for any word $h=(h',h'')$ in $\Mag(S) \setminus S$, 
we have $h\in H$ if and only if
$$
h',h''\in H \hbox{ and } h'<h'' 
\hbox{ and either $h'\in S$ or $h'=(u,v)$ with $v \geq h''$.}   
$$
\end{itemize}
Note that, if a group $G$ acts on a set $S$,
then it also acts on $\Mag(S)$ by acting on every letter of any word.
For instance, if $x,y \in S$ and $g\in G$, then 
$g\cdot (x,(y,x))= (g\cdot x,(g\cdot y,g \cdot x))$.

\begin{lemma} \label{lem:equiv-Hall}
Let $G$ be a group acting on a set $S$.
Let $\leq$ be a total order on $S$ such that $s\leq t$ implies $g \cdot s\leq g \cdot t$ for all $s,t\in S$ and $g\in G$.
Then, there exists a $G$-stable Hall set $H$ on the alphabet $S$ such that the order $\leq$ of $H$ extends that of~$S$.
\end{lemma}

\begin{proof}
We shall prove the lemma by refining the usual argument for the existence
of Hall sets (see e.g$.$ \cite[Prop. 4.1]{Reutenauer}).

Let $\Mag(\bullet)$ denote the free magma on one element $\bu$. 
The unique map $S \to \{ \bullet\}$ induces a magma homomorphism 
$p:\Mag(S) \to \Mag(\bullet)$, which records the parenthesization
of non-associative words in $S$.
For any $u\in \Mag(\bullet)$, 
the fiber $p^{-1}(u)$ can be canonically identified
with $S^{\vert u \vert}$ (for instance, given $u=(\bullet,(\bullet,\bullet))$,
we identify the word $(s_1,(s_2,s_3))$ with the triplet $(s_1,s_2,s_3)$
for any $s_1,s_2,s_3\in S$).
Choose a total order of $\Mag(\bullet)$ such that 
\begin{equation} \label{eq:length_order}
    \hbox{for all $u,u' \in\Mag(\bullet)$, $\vert u \vert < \vert u' \vert$
    implies $u>u'$}.
\end{equation}
It lifts via $p$
to a unique total order of $\Mag(S)$ that restricts to the lexicographic order 
of every fiber $p^{-1}(u)\simeq S^{\vert u \vert}$.
Since the order of $S$ is compatible with the $G$-action, so is this total order of $\Mag(S)$.

Then, the condition $(H_2)$ gives an inductive rule to construct a Hall set $H$
starting with $S$ in length $1$ (the order of $H$ is then the restriction of the total order of $\Mag(S)$). 
Then $(H_1)$ follows from \eqref{eq:length_order}.

It remains to verify that the Hall set $H$ is $G$-stable.
 Let $h \in H$ and $g\in G$.
One verifies $g\cdot h\in H$ by induction on $\vert h\vert$.
If $\vert h\vert =1$, then (obviously) $h\in S$ so that $g\cdot h\in S \subset H$.
If $\vert h\vert >1$, then $h=(h',h'')$ satisfies $(H_2)$ and, using the compatibility of the $G$-action with the total order $\leq$ in $\Mag(S)$, it is straightforward to verify that $g\cdot h=(g\cdot h',g\cdot h'')$ also satisfies $(H_2)$. 
\end{proof}

\begin{lemma} \label{lem:free_Lie}
Let $(G,\leq)$ be a left-ordered group
and $M$ a free $\Z[G]$-module.
Then, $\Lie(M)$ is free as a $\Z[G]$-module.
\end{lemma}
\begin{proof}
Let $X$ be a $\Z[G]$-basis of $M$, 
and $\leq$ a total order on $X$.
Then  
$$
G\cdot X =\big\{ g\cdot x \, \big\vert \, g\in G, x\in X\big\}
$$
is a $\Z$-basis of $M$.
Identifying $G\cdot X$ with $X\times G$ in the obvious way,
we can transport the lexicographic order of the latter to the former.
Thus, $G\cdot X$ is a totally ordered $G$-set whose order
is compatible with the $G$-action.
Hence, by Lemma \ref{lem:equiv-Hall}, there exists a $G$-stable
Hall set $H \subset \Mag(G\cdot X)$. 
Let $\beta:\Mag(G\cdot X) \to \Lie(M)$ be the bracketing map.
By a property of a Hall set
(see e.g$.$ \cite[Theorem~4.9]{Reutenauer}), 
$\beta(H)$ is a $\Z$-basis of $\Lie(M)$.
Let $E$ be the subset of $H$ such that $E$ retains, in each $G$-orbit contained in $H$,  
the unique non-associative word 
whose leftmost letter is in $X \subset G \cdot X$.

Let us verify that $\beta(E)$ freely generates the $\Z[G]$-module $\Lie(M)$. 
Since any element of $H$ can be transformed to a (unique) element of $E$
by the $G$-action, and since $\beta(H)$ generates $\Lie(M)$ as an abelian group,
the subset $\beta(E)$ generates $\Lie(M)$ as a $\Z[G]$-module. 
Assume now a $\Z[G]$-linear relation between some elements of $\beta(E)$.
Since the $\Z[G]$-module $\Lie(M)$ is graded, we can assume that 
this linear relation occurs in the homogeneous part $\Lie_m(M)$ 
of degree $m$ for some $m\geq 1$:
$$
\sum_{e\in E'_m} z_e \cdot \beta(e) \in \Lie_m(M),
$$
where the sum is taken over a subset $E'_m$ of 
$E$ consisting of finitely many  words of length $m$, and 
$z_e\in \Z[G]$ for all $e\in E'_m$. Decomposing 
$z_e=\sum_{g\in G} z_e(g)\cdot g$, we obtain a $\Z$-linear relation:
$$
\sum_{e\in E'_m,g\in G} z_e(g) \cdot \beta(g\cdot e) \in \Lie_m(M)
$$
Since $\beta(H)$ is $\Z$-free, we conclude that each $z_e(g)\in \Z$ is trivial, 
so that each $z_e\in \Z[G]$ is trivial. 
\end{proof}

\begin{lemma}
  \label{r54}
  Let $F$ be a free group of rank $n$ with basis $\{x_1,\dots, x_n\}$,
 and let $M$ be a free $\Z[F]$-module.  If $m_1,\ldots ,m_n\in M$ satisfy
  \begin{gather}
    \label{e40}
    \sum_{i=1}^n (1-x_i) \cdot m_i=0,
  \end{gather}
  then $m_1=\dots =m_n=0$.
\end{lemma}

\begin{proof}
  We can reduce the lemma to the case where $M=\Z[F]$, by
  considering the coefficients with respect to a $\Z[F]$-basis of $M$.  
  This case follows from the well-known fact that 
  the augmentation ideal $I_F$ of $\Z[F]$ is free on $x_1-1,\dots,x_n-1$ as a $\Z[F]$-module.
  (See e.g$.$ \cite[I.(4.4)]{Brown}.)
\end{proof}

\section{First terms of the Johnson filtration for the  handlebody group} 
\label{sec:first_quotient_handlebody}

In this section, we consider the first few terms of the Johnson filtration \eqref{eq:HH} of the handlebody group:
$$
\modH  = \modH_0 \ \ge \  \modH_1 \ \ge \  \modH_2.
$$

Recall that the twist group $\modT=\modH_1^0$
is the subgroup  of $\modH$ acting trivially on $F\simeq\pi/\Alpha$
or, equivalently (by Theorem \ref{thm:=}),
$\modT=\modH_1$ is the subgroup of $\modH$ 
acting trivially on the abelianization $\AAlpha$ of $\Alpha$.
As a first application of the Johnson filtration
of $\modH$, we deduce the following property of $\modT$
directly from Theorem \ref{th:residual}.

\begin{corollary} \label{cor:residual}
The group $\modT$ is residually
torsion-free nilpotent.
\end{corollary}

The group $\modH_0/\modH_1$ embeds both
into $\Aut(F)$ and $\Aut(\AAlpha)$.
Since any automorphism of $F$ can be realized 
by an element of $\modH$ 
(see \cite[Cor.~2.1]{Luft}, 
and  also \cite{Griffiths,Zieschang}),
we have 
\begin{gather}\label{modH0modH1}
\modH_0/\modH_1 \simeq \Aut(F).
\end{gather}

The image of the 
embedding $\modH_0/\modH_1\hookrightarrow\Aut(\AAlpha)$ admits the following characterization.
An element $\rho\in\Aut(\AAlpha)$
is \emph{quasi-equivariant} if there is a $\overrightarrow{\rho}\in\Aut(F)$ 
such that  
$$
\rho(f\cdot a) = \overrightarrow{\rho}(f) \cdot \rho(a)
\quad \hbox{for all $f\in F$ and $a\in \AAlpha$,}
$$
where the action of $F$ on $\AAlpha$
is as in Lemma \ref{lem:freeness}.
Clearly, 
 quasi-equivariant elements
 form a subgroup $\Aut_{\operatorname{qe}}(\AAlpha)$ of $\Aut(\AAlpha)$. 
Since $\overrightarrow{\rho}$ is uniquely determined by~$\rho$,  
there is a well-defined homomorphism
$$\overrightarrow{(-)}:\Aut_{\mathrm{qe}}(\AAlpha)\to\Aut{(F)}.$$
Let $\Aut_{\operatorname{qe}}^\zeta(\AAlpha)$ denote
the subgroup of  $\Aut_{\operatorname{qe}}(\AAlpha)$
fixing $[\zeta]=[\zeta]_1 \in \AAlpha=\Alpha_1/\Alpha_2$.
Having fixed  a system of curves $(\alpha,\beta)$ as in \eqref{eq:(a,b)}, this element writes
\begin{equation} \label{eq:zeta_ab}
[\zeta] = \sum_{i=1}^g (1-x_i^{-1})\cdot a_i     ,
\end{equation}
where $a_i=[\alpha_i]\in \AAlpha$ and $x_i=\varpi(\beta_i)\in F$.
In the sequel, for any matrix $M=(m_{ij})_{i,j}\in \Mat(g\times g;\Z[F])$, let $M^\dagger = (\overline{m_{ji}})_{i,j}$ denote its conjugate transpose.

\begin{proposition} \label{prop:Mag_Mag}
The homomorphism $\overrightarrow{(-)}$ on $\Aut_{\mathrm{qe}}(\mathbb A)$ induces an isomorphism
\begin{gather*}
\overrightarrow{(-)}:
\Aut_{\operatorname{qe}}^\zeta(\AAlpha) 
\underset\simeq\longrightarrow
 \Aut(F),
\end{gather*}
which fits into the following commutative diagram
$$
\xymatrix{
& \mathcal{H} \ar@{->>}[rd]  \ar@{->>}[ld] \ar@/_6pc/[ddl]_-{\Mag_0^1}  
\ar@/^6pc/[ddr]^-{\Mag^0_0}  & \\
\Aut_{\operatorname{qe}}^\zeta(\AAlpha) \ar@{+->}[d] 
 \ar[rr]_-\simeq^-{\overrightarrow{(-)}} & & \Aut(F) \ar@{+->}[d]_-J \\
\GL(g;\Z[F]) \ar[rr]_-\simeq^-{(-)^{\dagger,-1}} & &\GL(g;\Z[F])
}
$$
where the left (resp$.$ right) diagonal arrow gives the canonical action
on $\AAlpha$ (resp$.$~$F$),
the left vertical arrow maps any $\rho \in \Aut_{\operatorname{qe}}^\zeta(\AAlpha)$
to the  conjugate of its matrix $\big(a_i^*(\rho(a_j))\big)_{i,j}$
in the basis $(a_1,\dots,a_g)$ of $\AAlpha$
and the right vertical arrow $J$ maps any $r\in \Aut(F)$
to its free Jacobian matrix 
$J(r)=\Big(\overline{\frac{\partial r(x_j)}{\partial x_i} }\Big)_{i,j}$
in the basis $(x_1,\dots,x_g)$.  
\end{proposition}

\begin{proof}
The commutativity of the right upper triangle is the definition of $\Mag^0_0$,
that of the left upper triangle follows from Lemma \ref{lem:kappa}
and that of the central triangle is clear.
Hence, we only need to prove the commutativity of the bottom square:
\begin{equation} \label{eq:b_square}
\Big(\overline{\frac{\partial \overrightarrow{\rho}(x_j)}{\partial x_i}}\Big)_{i,j}^{-1}
= \big({a_j^*(\rho(a_i))}\big)_{i,j},  \quad 
\hbox{for all $\rho\in\Aut_{\operatorname{qe}}^\zeta(\AAlpha)$.}
\end{equation}
Note that 
$$
\rho([\zeta]) \stackrel{\eqref{eq:zeta_ab}}{=}
-\sum_i \big(\overrightarrow{\rho}(x_i^{-1})-1\big) \cdot \rho(a_i)
= - \sum_{i,l} (x_l^{-1}-1) 
\overline{\frac{\partial \overrightarrow{\rho}(x_i)}{\partial x_l}} \cdot \rho(a_i)
$$
which, using Lemma \ref{r54} and $\rho([\zeta])=[\zeta]$, implies  that
\begin{equation}\label{equ}
a_l = \sum_i  \overline{\frac{\partial\overrightarrow{\rho}(x_i)}{\partial x_l}}
\cdot \rho(a_i) =  
\sum_{i,k}  \overline{\frac{\partial\overrightarrow{\rho}(x_i)}{\partial x_l}}
a_k^*\big(\rho(a_i) \big) \cdot a_k.
\end{equation}
Since $(a_1,\dots,a_g)$ is a $\Z[F]$-basis of $\AAlpha$,
we obtain \eqref{eq:b_square}.

The surjectivity of $\modH \to \Aut(F)$ implies that of $\overrightarrow{(-)}$.
Injectivity of $\overrightarrow{(-)}$ follows from \eqref{equ}.
The other injectivity and surjectivity  in the diagram easily follow. 
\end{proof}{}

Next, we would like to identify  
$\modH_1/\modH_2$.
By Proposition \ref{prop:blocks},
one can embed it into the abelian group $\Mat(g\times g;\Z[F])$ 
via the representation
$$
\Mag:=\Mag_1^0: \modT \lto \Mat(g\times g;\Z[F]), \
f \longmapsto 
\Bigg(\varpi\Big(\overline{\frac{\partial f(\beta_j)}{\partial \alpha_i} }\Big) \Bigg)_{i,j}.
$$
The next proposition is a first step towards
the determination of the image of $\Mag$,
and the more difficult problem of computing the abelianization of $\modT$. 

\begin{proposition} \label{prop::hermitian}
The Magnus representation takes values in hermitian matrices:
for all $f \in \modT$, we have $\Mag(f)= \Mag(f)^\dagger$.
Besides, for all $f \in \modT$ and $h\in \modH$, we have
\begin{gather}\label{eq:hermitian}
\Mag(hfh^{-1}) = \Mag_0^0(h)\cdot h^F\big(\Mag(f)\big)\cdot  \Mag_0^0(h)^\dagger,
\end{gather}
where $h^F\in \Aut(F)$ is 
the action of $h$ on $F=\pi_1(V,\star)$.
\end{proposition}

\begin{proof}
Set $n=2g$ and $(z_1,\dots,z_{n})
=(\alpha_1,\dots,\alpha_g,\beta_1,\dots,\beta_g)$.
Let $f\in \mathcal{M}$.
We start by recalling  that the condition for $f\in \Aut(\pi)$
to preserve  the homotopy intersection form $\eta$ 
(see \S \ref{subsec:hif_surface}) implies
some strong conditions on the free Jacobian matrix 
$$
J(f)=\left(\overline{\frac{\partial f(z_j)}{\partial  z_i }}\right)_{i,j} 
\in \GL(2g;\Z[\pi]).
$$
(These appear under equivalent forms e.g$.$ in \cite{Morita_abelian} and \cite{Perron}.)

Recall that the (left) Fox derivatives 
$\frac{\partial \ }{\partial z_i}$ are uniquely defined by \eqref{eq:Fox}
and, similarly, we  define some (right) Fox derivatives
$\frac{\laitrap \ }{\laitrap z_i}$ by the identity
$$
w-\varepsilon(w) = \sum_{i=1}^{n} (z_i-1) \frac{\laitrap w}{\laitrap z_i}
\quad \hbox{for all $w\in \Z[\pi]$}.
$$
The properties \eqref{eq:Fox_left}--\eqref{eq:Fox_right}
of $\eta$ to be a Fox pairing imply that
\begin{equation} \label{eq:eta_z}
\eta(x,y)= \sum_{i,j} \frac{\partial x}{\partial z_i} \eta(z_i,z_j) 
\frac{\laitrap y}{\laitrap z_j} \quad \hbox{for all $x,y\in \Z[\pi]$}.
\end{equation}
Using \eqref{eq:eta_z}, 
the identity $f \circ \eta = \eta \circ (f \times f)$
translates into the matrix identity:
$$
f(E) = \left( \frac{\partial f(z_i)}{\partial z_j}\right)_{i,j} 
\cdot E \cdot \left(\frac{\laitrap f(z_j)}{\laitrap  z_i }\right)_{i,j},
$$
where  $E$ denotes the matrix of $\eta$ in the basis $(z_1,\dots,z_{n})$.
It easily follows from the definitions of left and right Fox derivatives
that $\frac{\laitrap w}{\laitrap z_i} = z_i^{-1} 
\overline{\frac{\partial w}{\partial z_i}} w$ for any $w\in \pi$. Hence we obtain
\begin{equation} \label{eq:6-term_0}
f(E) = J(f)^\dagger
\cdot E \cdot \operatorname{diag}(z_1^{-1},\dots,z_{n}^{-1}) \cdot
J(f) \cdot \operatorname{diag}(f(z_1),\dots,f(z_{n})).
\end{equation}

Assume now that $f\in \modT$. 
By applying $\varpi: \Z[\pi] \to \Z[F]$, the previous
matrix identity simplifies to 
\begin{equation} \label{eq:6-term}
\varpi(f(E))=\varpi(E) = J^F(f)^\dagger \cdot \varpi(E) \cdot (I \oplus D^{-1}) \cdot
J^F(f)
\cdot (I \oplus D),
\end{equation}
where $I=I_g$ is the identity matrix and
$D=\operatorname{diag}(x_1,\dots,x_g)$ is the diagonal matrix with entries $x_i$.
It also follows from Proposition  \ref{prop:blocks} that 
\begin{equation} \label{eq:J^F(f)}
J^F(f)= \left(\varpi\Big(\overline{ \frac{\partial f(z_j)}{\partial  z_i }}\Big)\right)_{i,j} 
= 
\left(\begin{array}{c|c}   I  & \Mag(f)   \\ \hline
0  & I  \end{array}\right)
\end{equation}
and from the computation of $E$ in \S \ref{subsec:hif_surface} that
$$
\varpi(E) = \left(\begin{array}{c|c} 0 & D \\ \hline 
-I & D-I \end{array}\right).  
$$
Then the identity $\Mag(f)=\Mag(f)^\dagger$ easily follows 
from \eqref{eq:6-term}.

Let also $h\in \modH$. Since $J^F$ is a crossed homomorphism, we get
$$
J^F(hfh^{-1}) \stackrel{\eqref{eq:product}}{=} J^F(h) \cdot h^F\big(J^F(f)\big) \cdot  h^F\big(J^F(h^{-1})\big),
$$
where $h^F\in \Aut(F)$ is 
the action of $h$ on $F=\pi_1(V,\star)$.
Then, Propositions~\ref{prop:blocks} 
and~\ref{prop:Mag_Mag} give
$$
J^F(h) = \left(\begin{array}{c|c}  \Mag_0^0(h) & \Mag^0_1(h) \\ \hline 
0 & \Mag^0_0(h)^{\dagger,-1} \end{array}\right).  
$$
It is now straightforward to deduce \eqref{eq:hermitian} from \eqref{eq:J^F(f)} and the two identities above.
\end{proof}

\begin{remark}
By applying \eqref{eq:6-term_0} to an $f\in \modH$, 
we get an alternative justification for the identity
$
\Mag^1_0(f) \cdot \Mag^0_0(f)^\dagger =I \in \GL(g;\Z[F]),
$
which is contained in Proposition \ref{prop:Mag_Mag}.
\end{remark}

\begin{example}
Let $T:=T_{\partial \Sigma}$ be the Dehn twist along the boundary curve. Then we have
\begin{equation} \label{eq:Magnus_boundary}
\varpi\Big(\frac{\partial T(\beta_i)}{\partial \alpha_j}\Big)
= \varpi\Big(\frac{\partial \,{}^\zeta\beta_i}{\partial \alpha_j}\Big)
= (1-x_i)\, \varpi\Big(\frac{\partial \zeta}{\partial \alpha_j}  \Big) 
\stackrel{\eqref{eq:zeta_ab}}{=}
(1-x_i)\, (1-x_j^{-1}).
\end{equation}
Hence we have $\Mag(T) = \big( (1-x_i )(1-x_j^{-1})\big)_{i,j}$.
\end{example}

\section{Johnson homomorphisms for the handlebody group}
\label{sec:Johnson_homomorphisms_handlebody}

Restricting  \eqref{eq:tau_0_0_1}
to $\modH=\modH_0$ yields homomorphisms
$$
\tau _0^{{0}} : {\modH}  \longrightarrow  \Aut(F), \quad \tau _0^{{1}}: {\modH}  \longrightarrow  \Aut\big({\AAlpha}\big),
$$
which satisfy \eqref{eq:equiv},
and whose images are described  in Proposition \ref{prop:Mag_Mag}.
Furthermore, restricting  \eqref{eq:tau_m_0_1} yields homomorphisms
$$
\tau _m^{{0}} : {\modH_m}  \longrightarrow  Z^1\big(F, \overline{\Alpha}_m \big),
\quad  \tau _m^{{1}} : 
{\modH_m}  \longrightarrow  \Hom \big({\AAlpha}, \overline{\Alpha}_{m+1} \big)\quad (m\ge1),
$$
which satisfy \eqref{eq:non-equiv}. 
These  maps
satisfy \eqref{eq:bracket0}--\eqref{eq:equiv1_bis}.
By \eqref{eq:ker_tau_m_0_1} we have
\begin{equation} \label{eq:ker_tau_m_0_1_HH}
\ker \tau _m^{0} =  {\modH_{m+1}} \quad  \hbox{and} \quad \ker \tau _m^{1} =  {\modH_{m+1}}.    
\end{equation}

\begin{example} \label{ex:0_1}
By Proposition \ref{prop:blocks}, 
in degree $m=1$, $\tau_1^0: \modT \to Z^1(F,\AAlpha)$
and  $\tau_1^1: \modT \to \Hom(\AAlpha, \Lambda^2\AAlpha)$
are equivalent to 
$\Mag=\Mag_1^0: \modT \to \Mat(g\times g; \Z[F])$,
whose image is partly described in 
Proposition \ref{prop::hermitian}.
\end{example}

Equation \eqref{eq:ker_tau_m_0_1_HH}  
and Theorem \ref{thm:=} immediately imply the following.

\begin{corollary} \label{cor:equivalence}
The two families of Johnson homomorphisms
$(\tau_m^{{0}})_m$ and $(\tau_m^{{1}})_m$ are equivalent to each other
for the handlebody group.
\end{corollary}

For some purposes (for instance, to have another viewpoint 
on Corollary \ref{cor:equivalence}), 
it will be also convenient to use the
extended graded Lie algebra morphism
$$
\overline{\tau}_\bu:  \overline{\modH}_\bu \lto 
\Der_\bu(\overline{\Alpha}_\bu) \simeq D_\bu(\overline{\Alpha}_\bu),
$$
which encompasses the two families $(\tau_m^{{0}})_m$ and $(\tau_m^{{1}})_m$
into a single map.
It follows from \eqref{eq:ker_tau_m_0_1_HH} that $\overline{\tau}_\bu$ is injective. 
Clearly, $\overline{\tau}_\bu$ takes values in the extended graded Lie subalgebra 
$$
\Der_\bu^\zeta (\overline{\Alpha}_\bu)
\simeq D_\bu^\zeta (\overline{\Alpha}_\bu)
$$
whose degree $0$ part consists of the extended graded Lie algebra automorphisms of $\overline{\Alpha}_\bu$
that fix $[\zeta]_1\in \overline{\Alpha}_1$ and whose positive-degree part consists of extended graded Lie algebra derivations of $\overline{\Alpha}_\bu$
that vanish on $[\zeta]$.
We call it the
\emph{extended graded Lie algebra of special derivations} of $\overline{\Alpha}_\bu$.

In the sequel, we will mainly focus on the twist group, 
for which it is enough to consider the morphism of graded Lie algebras
\begin{equation} \label{eq:gr_gr}
\overline{\tau}_+:  \overline{\modH}_+ \lto 
\Der_+^\zeta(\overline{\Alpha}_\bu) \simeq D_+^\zeta(\overline{\Alpha}_\bu).
\end{equation}
Recall that $\ol{\tau}_m$ is equivariant in the sense that
$$
\overline{\tau}_m([h t h^{-1}]_m) = h_*  \circ \overline{\tau}_m([t]_m) \circ h_*^{-1}    
$$
for any $m\geq 1$, $t\in \modH_m$, $h\in \modH$, where 
$h_*\in \Aut(\overline{\Alpha}_\bu)$ is the unique automorphism of extended graded Lie algebras
given by $\tau_0^0(h)\in \Aut(F)$ in degree $0$
and by $\tau_0^1(h)\in \Aut^\zeta_{\operatorname{qe}}(\AAlpha)$ in degree $1$.

The graded Lie algebra 
$\Der_+^\zeta(\overline{\Alpha}_\bu) \simeq D_+^\zeta(\overline{\Alpha}_\bu)$
is a rich algebraic object, whose study is postponed to \S \ref{sec:special_derivations}.
In the meantime, we explain 
how the embedding \eqref{eq:gr_gr}
lifts to an embedding of $\modT=\modH_1$ 
into the degree-completion of $D_+^\zeta(\overline{\Alpha}_\bu)$. 
For this, we need the following refinement of Lemma \ref{lem:expansion}.

\begin{lemma}\label{lem:special}
There exists an expansion $\theta$ of the free pair $(\pi,\Alpha)$ such that
the associated map 
$\ell^\theta: \pi \to \widehat{\Lie}(\AAlpha^\Q)$ satisfies
\begin{equation} \label{eq:symplectic}
\ell^\theta(\zeta) = [\zeta]_1 \in \AAlpha.
\end{equation}
\end{lemma}

An expansion $\theta$ of $(\pi,\Alpha)$ 
that satisfies the additional condition~\eqref{eq:symplectic}, 
i.e., $\ell^\theta(\zeta)$ is concentrated in degree $1$,
is called \emph{special}.
This is an analogue 
of the notion of symplectic expansion of $\pi$ 
that was considered
in relation with the classical study of Johnson homomorphisms
for the Torelli group \cite{Mas12}.
(See Lemma \ref{lem:special_to_symplectic} below, in this connection.)

\begin{proof}[Proof of Lemma \ref{lem:special}]
We start by recalling a related notion of special expansion
for a free group with a given basis \cite{AT,Mas18}.
Let $D$ be a finitely-generated free group, 
with basis $(\delta_1,\dots,\delta_n)$.
Let $\mathbb{D}^\Q := D_{\operatorname{ab}} \otimes \Q$
be its rational abelianization, and set $d_i:=[\delta_i] \in \mathbb{D}^\Q$.
A \emph{special expansion} of $D$ 
(relative to the basis $(\delta_1,\dots,\delta_n)$)
is a monoid homomorphism $\theta_0:D \to \widehat{T}(\mathbb{D}^\Q)$ 
with the following two properties:
\begin{itemize}
\item for all $i\in \{1,\dots, n\}$,
there exists a primitive element $v_i \in \widehat{T}(\mathbb{D}^\Q)$
such that $\theta_0(\delta_i) = \exp(v_i) \exp(d_i) \exp(-v_i)$;
\item we have $\theta_0(\delta_n \cdots \delta_2 \delta_1)=
\exp(d_1+ d_2+ \dots+d_n)$.
\end{itemize}
Such a map $\theta_0$ can be constructed 
by  successive finite-degree approximations of $v_1,\dots,v_n$.
For instance, up to degree one, we can take
\begin{equation} \label{eq:degree_one}
v_i = r_i d_i + \frac{1}{2} \sum_{i<j} d_j +(\deg \geq 2),
\end{equation}
where $r_i \in \Q$ can be chosen arbitrarily.

Let $(\alpha,\beta)$ be a basis of $\pi$ of type \eqref{eq:(a,b)}. Recall that the curves $\alpha_1,\dots,\alpha_g$ are meridians of $V$.
Let $\Sigma'\subset\Sigma$ be the disk with $2g$ holes that is obtained from $\Sigma$ by removing the $g$ handles. Then ${D}:=\pi_1(\Sigma',\star)$ is free on the loops 
$\alpha_1', \alpha_1, \dots, \alpha'_g,\alpha_g$ shown below: 
\begin{equation} \label{eq:loops}
\labellist
\scriptsize\hair 2pt
 \pinlabel {$\alpha_1$} [l] at 280 80
\pinlabel {$\alpha'_1$} [r] at 110 80
\pinlabel {$\alpha_g'$} [r] at 425 80
\pinlabel {$\alpha_g$} [l] at 590 80
\pinlabel {\large $\star$} at 566 4
\endlabellist
\centering
\includegraphics[scale=0.42]{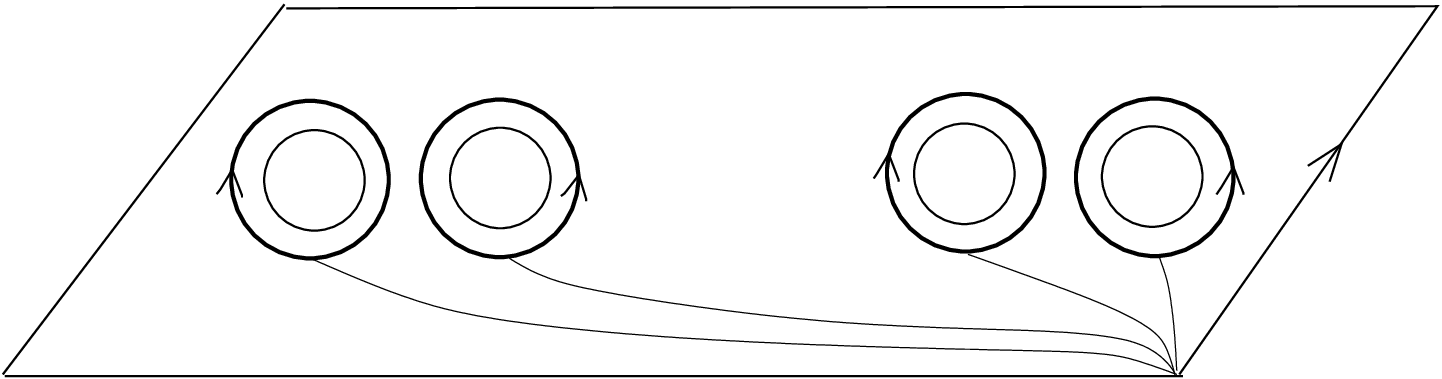}
\end{equation}
Let also ${B}$ be the free group on $\beta_1, \dots,\beta_g$. Then
\begin{equation} \label{eq:pi}
\pi \simeq  \frac{{D} \ast {B}}{{R}},
\end{equation}
where $R$ denotes the subgroup of the free product $D \ast B$ normally generated by the elements $\alpha'_i \beta_i^{-1} \alpha_i^{-1} \beta_i$ for all $i\in \{1,\dots,g\}$.

Choose now  a special expansion $\theta_0: D \to \widehat{T}(\mathbb{D}^\Q)$ 
of the free group $D$ 
(relative to the basis 
$\big((\alpha_1')^{-1}, \alpha_1, \dots, (\alpha'_g)^{-1},\alpha_g\big)$):
thus, there exist primitive elements 
 $u_1,u'_1,\dots,u_g,u'_g \in  \widehat{T}(\mathbb{D}^\Q)$ such that
$$
\theta_0(\alpha_i) = \exp(u_i) \, \exp(a_i) \, \exp(-u_i),
\quad 
\theta_0((\alpha'_i)^{-1}) = \exp(u'_i) \, \exp(-a'_i) \, \exp(-u'_i),
$$
where $a'_i:=[\alpha'_i]\in \mathbb{D}^\Q$ 
and $a_i:=[\alpha_i] \in \mathbb{D}^\Q$, and the following  condition is satisfied:
\begin{equation}   \label{eq:special}
\theta_0\big(\alpha_g (\alpha'_g)^{-1} \cdots \alpha_1 (\alpha'_1)^{-1} \big) 
= \exp\big( (a_1-a'_1) + \cdots + (a_g-a'_g)\big)
\end{equation}
The inclusion $\Sigma'\hookrightarrow\Sigma$ 
induces a homomorphism $D \to \sf{A}$
at the level of the fundamental groups, 
which further induces an (injective) algebra homomorphism 
$q:  \widehat{T}(\mathbb{D}^\Q) \to  \widehat{T}(\AAlpha^\Q)$.
Thus, we can define a multiplicative map 
$\theta_1: D  \to \widehat{T}(\AAlpha^\Q) \otimes_\Q \Q[F]$ by
\begin{equation}
 \label{eq:theta1}
\theta_1(\alpha_i):= q\theta_0(\alpha_i)\otimes 1  
\quad \hbox{and} \quad \theta_1(\alpha'_i):= q \theta_0 (\alpha'_i) \otimes 1.
\end{equation}
We also define a multiplicative map 
$\theta_2: B \to \widehat{T}(\AAlpha^\Q) \otimes_\Q \Q[F]$ by
\begin{equation}
 \label{eq:theta2}
\theta_2(\beta_i) := \exp\big( q (u_i)\big)\, \exp \big({}^{{x_i}}q (-u'_i) \big) \otimes x_i,
\end{equation}
where $x_i= \varpi(\beta_i)$.
Thus we obtain a multiplicative map
\begin{equation}
    \theta_1 \ast \theta_2: D \ast B \longrightarrow 
\widehat{T}(\AAlpha^\Q) \otimes_\Q \Q[F].
\end{equation}
For all $i\in\{1,\dots,g\}$, we have
\begin{eqnarray*}
&& \big(\theta_1 \ast \theta_2\big)\big(\alpha'_i \beta_i^{-1} \alpha_i^{-1} \beta_i\big) \\
&=&\theta_1(\alpha'_i) \cdot  \theta_2(\beta_i)^{-1} 
\cdot \theta_1(\alpha_i)^{-1} \cdot \theta_2(\beta_i) \\
&=& \big(   \exp q (u'_i) \,   \exp q(a'_i) \, \exp q(-u'_i) \otimes 1 \big) 
\cdot \big( \exp q (u'_i) \,  
\exp\big(  q (-u_i)^{{x_i}} \big)\otimes x_i^{-1} \big) \\
 && \cdot \big(   \exp q (u_i) \,   \exp q(-a_i) \, \exp q(-u_i) \otimes 1  \big) 
 \cdot \big( \exp q (u_i)\, \exp \big({}^{{x_i}} q (-u'_i)\big) \otimes x_i \big)  \\
&=& \big(   \exp q (u'_i) \,   \exp q(a'_i) \, 
\exp q \big((-u_i)^{x_i}  \big) 
\otimes  x_i^{-1}  \big)  \\
&&  \cdot \big(   \exp q (u_i) \,   \exp q(-a_i)  
\exp q \big( {}^{{x_i}} (-u'_i) \big) \otimes x_i   \big)  \\
&=& \exp q (u'_i) \,   \exp q(a'_i) \,    \exp q(-a_i)^{{x_i}}  
\exp q(-u'_i)  \otimes  1   \ = \ 1 \otimes 1,
\end{eqnarray*}
where the last identity follows 
from $q(a'_i)= q(a_i)^{x_i} \in \AAlpha^\Q$.
It follows from~\eqref{eq:pi} that $\theta_1 \ast \theta_2$ 
induces a multiplicative map
$$
\theta: \pi \longrightarrow \widehat{T}(\AAlpha^\Q) \otimes_\Q \Q[F].
$$
That $\theta$ satisfies conditions \eqref{eq:ell} and 
\eqref{eq:ell_a} in Definition \ref{def:expansion_free_pair}
follows directly from the definitions of $\theta_1$ and $\theta_2$. 
It remains to verify condition \eqref{eq:symplectic}:
\begin{eqnarray*}
 \theta(\zeta) \!\! \!\! & \stackrel{\eqref{eq:zeta}}{=}&
\theta\big([\alpha_g, \beta_g^{-1}] \cdots  [\alpha_1, \beta_1^{-1}] \big) \\
&=& \theta\big(\alpha_g (\alpha'_g)^{-1} 
\cdots \alpha_1 (\alpha'_1)^{-1} \big) \\
&=& q \theta_0 \big(\alpha_g (\alpha'_g)^{-1} 
\cdots \alpha_1 (\alpha'_1)^{-1} \big) \otimes 1 \\
& \stackrel{\eqref{eq:special}}{=} & q \exp\big( (a_1-a'_1) + \cdots 
+ (a_g-a'_g)\big) \otimes 1 \\
& = &  \exp\big( (q(a_1)-q(a'_1)) + \cdots + (q(a_g)-q(a'_g))\big) \otimes 1 \\
&=& \exp \big( (q(a_1) -  q(a_1)^{{x_1} } ) + \cdots 
+  (q(a_g) -  q(a_g)^{{x_g} }) \big) \otimes 1 
 \, \stackrel{\eqref{eq:zeta_ab}}{=} \, \exp([\zeta]_1) \otimes 1.
\end{eqnarray*}
\end{proof}

Recall that the complete Lie algebra 
$\widehat D_+^\zeta \big(\overline{\Alpha}^\Q_\bu\big)$ can be  viewed as a group
whose multiplication is given by the BCH formula.
The following provides an infinitesimal
version of the action of $\modT$ on the free pair $(\pi,\Alpha)$.

\begin{theorem} \label{thm:varrho_HH}
Let $\theta$ be a special  expansion of the free pair $(\pi,\Alpha)$. 
There is a homomorphism
\begin{equation} \label{eq:varrho_free_pair_HH}
\varrho^\theta: \modT \lto 
\widehat  D_+^\zeta\big(\overline{\Alpha}^\Q_\bu\big)
\end{equation}
whose restriction to $\modH_m$ $(m\geq 1)$ starts in degree $m$ 
with $(\tau^0_m,\tau^1_m)$ .
Furthermore, $\varrho^\theta$ is injective.
\end{theorem}

\begin{proof}
This is obtained by restricting the homomorphism $\varrho^\theta$ of 
Theorem \ref{thm:varrho} to the twist group $\modT=\modH_1 \subset \modG_1$.
Observe that, here, the condition for a $g\in \modT$
to fix $\zeta \in \pi$ translates into the property for the derivation 
$\log(\rho^\theta(g))$ to vanish on $\theta(\zeta)$
or, equivalently, on  $\log \theta(\zeta)=\ell^\theta(\zeta)$.
\end{proof}

\section{The Lie algebra of special derivations for the handlebody group} 
\label{sec:special_derivations}

In this section, we give two  equivalent descriptions of the Lie algebra 
of special derivations~$\Der_+^\zeta(\overline{\Alpha}_\bu) \simeq
D_+^\zeta(\overline{\Alpha}_\bu)$.

\subsection{The graded Lie algebra $D^0_+$}

Let $I_F$ be the augmentation ideal of $\Z[F]$, viewed as a left $\Z[F]$-module.
So far, we have regarded $\AAlpha$ as a left $\Z[F]$-module
using the left conjugation of $\pi$ on $\Alpha$,
but we can also regard it as a right $\Z[F]$-module $\AAlpha^r$ using the right conjugation.
Given a left (resp$.$ right) $\Z[F]$-module~$M$, 
let $M^*=\Hom_{\ZF}(M,\ZF)$ denote the module of $\Z[F]$-linear forms on $M$,
 which  is a right (resp$.$ left) $\Z[F]$-module. 
With these notations, 
the intersection operation $\langle -,-\rangle$ 
given in Proposition \ref{prop:intersection}
induces canonical isomorphisms
\begin{equation}  \label{eq:two_isomorphisms}
\AAlpha^r \simeq  I_F^*
\quad \hbox{and} \quad I_F  \simeq (\AAlpha^{r})^*
\end{equation}
of right and left $\Z[F]$-modules, respectively.

\begin{lemma} \label{lem:vartheta}
For any left  $\Z[F]$-module $M$,  
there is a canonical isomorphism
$$
\vartheta:Z^1(F,M) \stackrel{\simeq}{\longrightarrow} \AAlpha^r \otimes_{\Z[F]} M.
$$
\end{lemma}

\begin{proof}
For any $c\in Z^1(F, M)$, let 
$d: I_F \to M$ be the restriction to $I_F$ of the $\Z$-linearization $c:\ZF\to M$ of $c$. 
Conversely, each $d\in \Hom_{\Z[F]}(I_F,M)$ 
yields a $1$-cocycle $c\in Z^1(F,M)$ by $c(f)=d(f-1)$. Let $\vartheta$ be the composition of the following isomorphisms
$$
Z^1(F , M) \simeq \Hom_{\Z[F]}\big(I_F,{M}\big) \simeq
I_F^*  \otimes_{\Z[F]} M
\ \stackrel{\eqref{eq:two_isomorphisms}}{\simeq} \AAlpha^r \otimes_{\Z[F]} M.\\[-0.4cm]
$$
\end{proof}

Let $[-,-]: \AAlpha \times  {\overline{\Alpha}_+} \to {\overline{\Alpha}_+}$ denote the restriction of the Lie bracket of $\ol\A_+$.
We regard ${\overline{\Alpha}_+}$ as a $\Z[F]$-module.
Then we have  
$\overline{\Alpha}_{+}/I_F\overline{\Alpha}_{+}
\simeq(\ol\A_+)_F$, the $F$-coinvariants of $\ol\A_+$.
Define a $\Z$-linear map
$$
\beta: \AAlpha^r \otimes_{\Z[F]}  {\overline{\Alpha}_+} 
\to  \overline{\Alpha}_{+}/I_F\overline{\Alpha}_{+}
$$
by composing $[-,-]$ with the projection
${\overline{\Alpha}_+} \to \overline{\Alpha}_{+}/I_F\overline{\Alpha}_{+}$,
and set
\begin{equation} \label{eq:D^0_+}
D^0_+ = \big\{ c \in Z^1(F,\overline{\Alpha}_+) : \beta \vartheta(c)=0\big\}.    
\end{equation}

The following proposition means that every element of $\Der^\zeta_+(\overline{\A}_\bu)$ is determined by its degree $0$ part.

\begin{proposition} \label{prop:iso_D0}
We have an isomorphism
\begin{equation} \label{eq:iso_D0}
\Der_+^\zeta(\overline{\Alpha}_\bu) \stackrel{\simeq}{\longrightarrow} D^0_+,
\ (d_i)_{i\geq 0} \longmapsto d_0,
\end{equation}
\end{proposition}

\begin{proof}
We first verify that, for any $(d_i)_{i\geq 0} \in \Der_+(\overline{\Alpha}_\bu)$ with $d_1([\zeta])=0$ we have $\beta \vartheta(d_0)=0$. This follows from 
\begin{equation} \label{eq:fact_c}
\beta\vartheta(d_0)   \equiv  d_1([\zeta])  \mod I_F\overline{\Alpha}_{+},
\end{equation}
which we now prove.
As in \eqref{eq:zeta_ab},
let $a_i=[\alpha_i]\in \overline{\Alpha}_1 = \AAlpha$
and $x_i=[\beta_i] \in \overline{\Alpha}_0\simeq F$,
for $i\in \{1,\dots,g\}$.
By \eqref{eq:in_bases}, the isomorphism \eqref{eq:two_isomorphisms}  maps 
the dual basis  $\big((x_i-1)^*\big)_i$ of $I_F^*$ to $-1$ times the basis $(a_i)_i$ of $\AAlpha^r$.
Therefore, 
\begin{equation} \label{eq:betavartheta}
\beta\vartheta(d_0) = \beta\Big(- \sum_{i=1}^g  a_i \otimes d_0(x_i)\Big)
\equiv -  \sum_{i=1}^g   [a_i ,d_0(x_i)]\mod I_F\overline{\Alpha}_{+}.    
\end{equation}
On the other hand, we have
\begin{eqnarray}
\label{eq:1_0} d_1([\zeta]) 
& \stackrel{\eqref{eq:zeta_ab}}{=} & d_1 \big(\sum_{i=1}^g (1-x_i^{-1})\cdot a_i\big)\\
\notag &\stackrel{\eqref{eq:equiv1}}{=}&
\sum_{i=1}^g (1-x_i^{-1}) \cdot d_1(a_i) 
-  \sum_{i=1}^g [ d_0(x_i^{-1}), {x_i^{-1}} \cdot a_i]\\
\notag &=& \sum_{i=1}^g (1-x_i^{-1}) \cdot d_1(a_i) 
+  \sum_{i=1}^g x_i^{-1} [d_0(x_i) ,  a_i]
\end{eqnarray}
which immediately implies \eqref{eq:fact_c}.
 
Next, we prove the injectivity of \eqref{eq:iso_D0}.
Let $(d_i)_{i\geq 0} \in \Der_+^\zeta(\overline{\Alpha}_\bu) $ 
with $d_0=0$. We deduce from \eqref{eq:1_0} that
$$
\sum_{i=1}^g (1-x_i^{-1}) \cdot d_1(a_i) =0
$$
and, by Lemmas \ref{lem:free_Lie} and  \ref{r54}, we obtain  
$d_1(a_1)=\cdots=d_1(a_g)=0$. Therefore~$d_1$ and 
consequently $(d_i)_{i\geq 0}$ are trivial. 

Finally, we prove the surjectivity of \eqref{eq:iso_D0}. Let $c\in D_+^0$. 
It follows from \eqref{eq:betavartheta} that
$\sum_{i=1}^g   x_i^{-1}\cdot[a_i ,c(x_i)] \in I_F \overline{\A}_+$.
Hence,  for some $u_1,\dots,u_g  \in\overline{\Alpha}_{+}$ we have
\begin{equation} \label{eq:d_ci}
\sum_{i=1}^g    x_i^{-1}[a_i ,c(x_i)]= \sum_{i=1}^g (1-x_i^{-1}) \cdot u_i .   
\end{equation}
Since the abelian group $\AAlpha$ is free on the ${}^f a_i$
for all $f\in F$ and $i\in \{1,\dots,g\}$, there is a unique homomorphism
$u: \AAlpha \to \overline{\Alpha}_{+}$ such that
$$
u({}^f a_i)= {}^f u_i + \big[c(f),{}^f a_i\big].
$$
By construction, the pair $(c,u)$ belongs to $D_+(\overline{\Alpha}_\bu)$
and, setting $(d_0,d_1)=(c,u)$, it can be extended 
to $(d_i)_{i\geq 0} \in \Der_+(\overline{\Alpha}_\bu)$.
Then, combining \eqref{eq:1_0} and \eqref{eq:d_ci}, we get $d_1([\zeta])=0$. 
Therefore, there exists a $(d_i)_{i\geq 0} \in \Der_+^\zeta(\overline{\Alpha}_\bu)$
such that $d_0=c$.
\end{proof}

\begin{remark}
According to \eqref{eq:b0}, 
the Lie bracket in $D^0_+$ corresponding to the Lie bracket of derivations through
\eqref{eq:iso_D0} is given by the following formula. For any $d,e \in D^0_+$,
we define $[d,e]$ by
\begin{equation} \label{eq:bracket_D^0_+}
 [d,e](f) = d_+(e(f))-e_+(d(f))- [d(f),e(f)] \quad\text{for all } f\in F,
\end{equation}
 where $d_+$ and $e_+$ 
 are the derivations of $\Lie(\AAlpha)=\overline{\Alpha}_+$ 
 completing  $d$ and $e$, respectively,
to derivations of the extended graded Lie algebra $\overline{\Alpha}_\bu$.

 Proposition \ref{prop:inversion} below describes the map $ D^0_+ \to \Der_+^\zeta(\overline{\Alpha}_\bu)$
inverse to \eqref{eq:iso_D0}.
This  allows for a complete understanding of the  bracket \eqref{eq:bracket_D^0_+}.
 \end{remark}

\subsection{The graded Lie algebra $D^1_+$}

The group $F$ acts on the abelian group $\Hom(\AAlpha, \overline{\A}_+)$ 
by the rule 
$$
({}^fd)(a):= {}^f\big(d\big({}^{f^{-1}}a\big)\big)\quad
\text{for all } a\in \AAlpha
$$
for any  $d\in \Hom(\AAlpha, \overline{\A}_+)$ and $f\in F$.
An element $d\in \Hom(\AAlpha, \overline{\A}_+)$ 
is said to be \emph{quasi-equivariant} 
if, for all $f\in F$, we have ${}^fd-d=[v,-]$
for some $v\in \overline{\A}_+$. Set
$$
D^1_+ = \big\{ d\in \Hom(\AAlpha, \overline{\A}_+) : 
\hbox{\small $d$ is quasi-equivariant
with values in $\overline{\A}_{\geq 2}$ and } d([\zeta])=0 \big\}.
$$

The following proposition means that every element of $\Der^\zeta_+(\overline{\A}_\bu)$ is also determined by its degree $1$ part.

\begin{proposition} \label{prop:iso_D1}
We have an isomorphism
\begin{equation}
\label{eq:iso_D1}
\Der_+^\zeta(\overline{\Alpha}_\bu) \stackrel{\simeq}{\longrightarrow} D^1_+,
\ (d_i)_{i\geq 0} \longmapsto d_1    .
\end{equation}
\end{proposition}

\begin{proof}
For $(d_i)_{i\geq 0} \in \Der_+^\zeta(\overline{\Alpha}_\bu)$,
the quasi-equivariance of $d_1$ follows from \eqref{eq:equiv1}.

We prove the injectivity. 
Let $(d_i)_{i\geq 0} \in \Der_+^\zeta(\overline{\Alpha}_\bu) $ 
with $d_1=0$. Let $f\in F$.
By~\eqref{eq:equiv1}, $d_0(f)$ 
commutes with $\AAlpha=\overline{\A}_1$
in the free Lie algebra $\Lie(\AAlpha)=\overline{\A}_+$, hence $d_0(f)=0$.
Thus, $d_0$ and consequently $(d_i)_{i\geq 0}$ are trivial.

We now prove the surjectivity. Let $d_1\in D_+^1$.
There exists a (unique) map $d_0:F\to \overline{\A}_+$
such that $[d_0(f), -]= d_1-{}^fd_1$ for any $f\in F$,
and it is easily checked that $d_0$ is a $1$-cocycle.
By construction, we have $(d_0,d_1)\in D_+^\zeta(\overline{\A}_\bu)$, which
extends to a derivation
$(d_i)_{i\geq 0} \in \Der_+^\zeta(\overline{\A}_\bu) $.
\end{proof} 

\begin{remark}
According to \eqref{eq:b1}, 
the Lie bracket in $D^1_+$ corresponding to the Lie bracket of derivations through
\eqref{eq:iso_D1} is given by the following formula. For any $d,e \in D^1_+$,
define $[d,e]$ by
$$
 [d,e](a) = d_+(e(a)) -e_+(d(a))
\quad \text{for all } a\in \AAlpha,$$
 where $d_+$ and $e_+$ are the derivations of 
 $\Lie(\AAlpha)=\overline{\Alpha}_+$ extending $d$ and $e$ respectively.
 \end{remark}

The following proposition describes the isomorphism $D^0_+\simeq D^1_+$
that is obtained by combining Propositions \ref{prop:iso_D0}
and \ref{prop:iso_D1}.

\begin{proposition} \label{prop:inversion}
The $1$-cocycle $c:F \to \Lie(\AAlpha)$ in $D^0_+$ 
corresponding to a  homomorphism $d:\AAlpha \to \Lie(\AAlpha)$ in $D^1_+$
is given by
\begin{equation} \label{eq:c_formula}
\quad [c(f),-]=d -{}^fd\quad
\text{for all }f\in F.
\end{equation}
Conversely, the homomorphism $d:\AAlpha \to \Lie(\AAlpha)$ in $D^1_+$
corresponding to a $1$-cocycle $c:F \to \Lie(\AAlpha)$ in $D^0_+$ is given by 
\begin{equation} \label{eq:d_formula}
d(a) = - \Big[\overline{\Psi\big(\vartheta(c)^\ell,a\big)^\ell} 
\cdot \vartheta(c)^r, \Psi\big(\vartheta(c)^\ell,a\big)^r \Big]\quad
\text{for all }a\in \AAlpha, 
\end{equation}
where $\vartheta(c)^\ell \otimes \vartheta(c)^r 
\in \AAlpha \otimes \Lie(\AAlpha)$
denotes a lift of $\vartheta(c) 
\in \AAlpha^r \otimes_{\Z[F]} \Lie(\AAlpha)$
such that 
$[\vartheta(c)^\ell,\vartheta(c)^r]=0 \in \Lie(\AAlpha)$.
\end{proposition}

In the above proposition,
$\Psi: \AAlpha \otimes \AAlpha \to \Z[F] \otimes \AAlpha$
is the intersection operation introduced in~\S\ref{subsec:Psi}.
Recall also that, with our conventions,  the expansion of a tensor product
$\Psi(a,b)\in \Z[F] \otimes \AAlpha$ 
is denoted by $\Psi(a,b)^\ell \otimes \Psi(a,b)^r$.

\begin{proof}[Proof of Proposition \ref{prop:inversion}]
The first statement follows immediately 
from the proof of Proposition \ref{prop:iso_D1}.

To prove the second statement, let $c\in D^0_+$ 
and let us first check that $\vartheta(c) 
\in \AAlpha^r \otimes_{\Z[F]} \Lie(\AAlpha)$
does have a lift to $\AAlpha \otimes \Lie(\AAlpha)$
with trivial Lie bracket. 
Choose any lift
$\sum_i u_i \otimes v_i \in \AAlpha \otimes \Lie(\AAlpha)$
of $\vartheta(c)$. Since $\beta\vartheta(c) =0$, 
there exist some $f_j \in F$ and $w_j \in \Lie(\AAlpha)$ such that
$$
\sum_{i} [u_i, v_i] = \sum_j (f_j-1) \cdot w_j.
$$
For every $j$, we can find some $s_{jk}\in \AAlpha$ and $t_{jk} \in \Lie(\AAlpha)$
such that $w_j =  \sum_k [s_{jk},t_{jk}]$. Hence
$$
\sum_i u_i \otimes v_i - \sum_{j,k} (f_j\cdot s_{jk})\otimes (f_j \cdot t_{jk})
+ \sum_{j,k}  s_{jk}\otimes   t_{jk}
$$
has trivial Lie bracket and, clearly, it is also a lift of $\vartheta(c)$.

Let now $d:\AAlpha \to \Lie(\AAlpha)$
be the homomorphism defined by~\eqref{eq:d_formula}. 
The condition on the lift 
$\vartheta(c)^\ell \otimes \vartheta(c)^r$ of $\vartheta(c)$ 
implies that
$$
d([\zeta]) \stackrel{\eqref{eq:Psi_zeta}}{=} - [\vartheta(c)^r, \vartheta(c)^\ell] =0.
$$
Besides, for any $f\in F$ and $a\in \AAlpha$, we have
\begin{eqnarray*}
 d({}^f a)&=& - \Big[\overline{\Psi\big(\vartheta(c)^\ell,{}^fa\big)^\ell} 
\cdot \vartheta(c)^r, \Psi\big(\vartheta(c)^\ell,{}^fa\big)^r \Big] \\  
&\stackrel{\eqref{eq:conj_right}}{=} &
 - \Big[\overline{\Psi\big(\vartheta(c)^\ell,a\big)^\ell f^{-1}} 
\cdot \vartheta(c)^r, f\cdot \Psi\big(\vartheta(c)^\ell,a\big)^r \Big]  + \big[{\langle f,\vartheta(c)^\ell  \rangle}  
\cdot \vartheta(c)^r,  {}^f a \big]\\
&=& {}^f d(a) + \big[ c(f),{}^f a \big].
\end{eqnarray*}
Thus, we have $d\in D^1_+$, and the corresponding $1$-cocycle is $c$.
\end{proof}

\begin{remark}
It follows from the injectivity of \eqref{eq:iso_D0} 
that the right hand side of \eqref{eq:d_formula}
does not depend on the choice of the lift of
$\vartheta(c) 
\in \AAlpha^r \otimes_{\Z[F]} \Lie(\AAlpha)$
to $\AAlpha^r \otimes \Lie(\AAlpha)$. 
This can also be checked directly using \eqref{eq:conj_left} and Lemma \ref{r54}.
\end{remark}

\section{The Lie algebra of oriented trees with beads} \label{sec:diagrams}

In this section, we define the Lie algebra of oriented trees with beads, and give a diagrammatic 
description of the Lie algebra 
of special derivations~
$D^0_+\simeq \Der_+^\zeta(\overline{\Alpha}_\bu) \simeq D^1_+$
with rational coefficients.

\subsection{Trees with beads}

We start with a general study of certain spaces of ``trees'', which are ``tree-shaped Jacobi diagrams with beads'', see Remark \ref{Jd}.

By a \emph{tree} we mean a finite simply-connected graph 
with only univalent vertices (called \emph{leaves}) 
and trivalent vertices  (called \emph{nodes}).
A tree is \emph{edge-oriented} if each
 edge is {oriented},  \emph{node-oriented} if
a cyclic order of the three edges incident at each vertex is specified, and  \emph{oriented} if it is both edge-oriented and node-oriented.
In figures, the edge-orientations are shown with little arrows, 
and we agree that node-orientations are always given 
by the counter-clockwise direction.

Let $V$ and $H$ be sets.
A tree $T$ is  said to be \emph{on} $V$ 
if all its leaves are colored by elements of $V$, 
and it is said to have \emph{$H$-beads} 
if some of its edges are colored by elements of $H$.
In figures, the edges that are colored by $H$
are decorated with beads;
there may be multiple beads on a single edge.

We define the \emph{degree} of a tree $T$ to be $1$ plus the number of nodes of $T$, which equals the number of leaves of $T$ minus $1$.

\begin{example}
Here is an oriented tree on $V$ with $H$-beads of degree $4$:\\[0.1cm]
$$
\qquad \quad
\labellist
\small \hair 2pt
 \pinlabel {$a$} [ru] at -15 -5
  \pinlabel {$b$} [rd] at -15 150
  \pinlabel {$c$} [d] at 190 200
  \pinlabel {$d$} [ld] at 380 155
  \pinlabel {$e$} [lt] at 360 0
   \pinlabel {$x$} [t] at 145 65
    \pinlabel {$y$} [lt] at 335 120
\endlabellist
\centering
\includegraphics[scale=0.24]{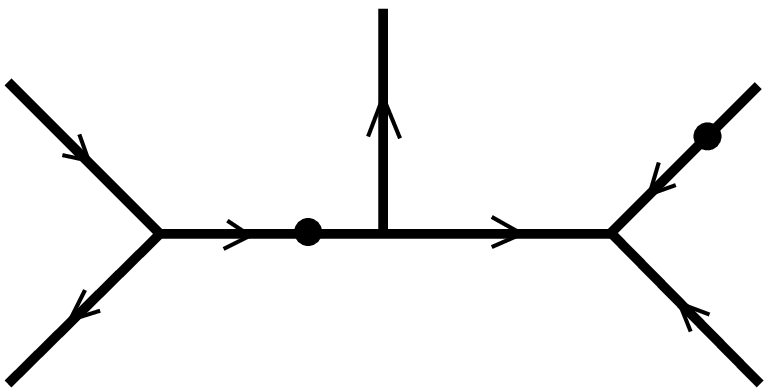}
\qquad \hbox{(with $a,b,c,d,e\in V$, $x,y\in H$)}
$$ 
\end{example}

In the following, let 
$H=(H,\Delta, \varepsilon, S)$ 
be a cocommutative Hopf $\Q$-algebra
and $V$ a left $H$-module.
Consider the $\Q$-vector space 
$$
\mathcal{D}(V,H) :=
\frac{\Q \cdot \left\{
\hbox{oriented trees on $V$ with $H$-beads}\right\}}{\hbox{AS, IHX, multilinearity, Hopf, bead-out}}
$$
where the relations are defined as follows.
\begin{itemize}
\item The \emph{AS} and \emph{IHX} relations,
which appear in the theory of finite-type invariants
(see, for instance, \cite[\S 6.1]{Ohtsuki}),
take place in a neighborhood of a node 
and an internal edge, respectively:
$$
\labellist
\small\hair 2pt
 \pinlabel {$=-$} at 110 45
 \pinlabel {$-$}  at 483 45
  \pinlabel {$,$} at 300 47
 \pinlabel {$+$} at 607 47
 \pinlabel {$=0$}  at 725 50
\endlabellist
\centering
\includegraphics[scale=0.3]{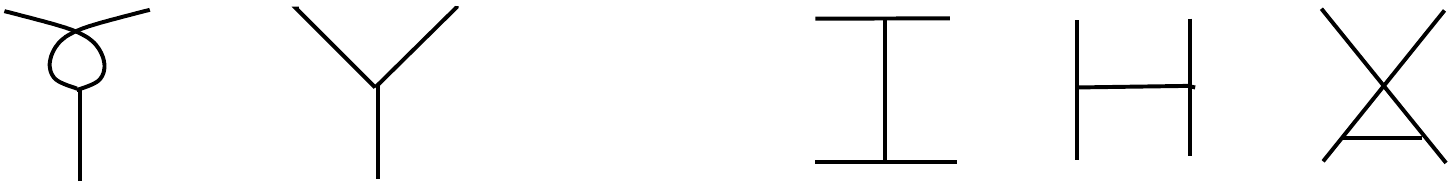}
$$
(No beads should appear here, and the edge-orientations are arbitrary;
the trivalent vertices are nodes of the trees, 
whereas the 4-valent vertices are ``fake'', being
caused by the planar presentation of trees.)
\item The \emph{multilinearity} relations require $\Q$-multilinearity
for the colors at the leaves and edges.
\item The \emph{Hopf} relations involve the Hopf algebra
structure of $H$, and they take place in a neighborhood
of an edge or a node:
$$
\labellist
\small\hair 2pt
 \pinlabel {$1$} [t] at 65 192
  \pinlabel {$=$}  at 195 196
   \pinlabel {$=$}  at 195 50
      \pinlabel {=} at 820 196
   \pinlabel {=}  at 820 50
    \pinlabel {,}  at 500 194
\pinlabel {,}  at 500 50
   \pinlabel {,}  at 1100 194
 \pinlabel {$x$} [t] at 35 47
 \pinlabel {$y$} [t] at 72 48
 \pinlabel {$xy$} [t] at 320 47
 \pinlabel {$x$} [t] at 681 194
 \pinlabel {$\overline{x}$} [t] at 943 194
 \pinlabel {$x$} [t] at 655 46 
 \pinlabel {$x'$} [br] at 959 76
 \pinlabel {$x''$} [rt] at 955 28
\endlabellist
\centering
\includegraphics[scale=0.25]{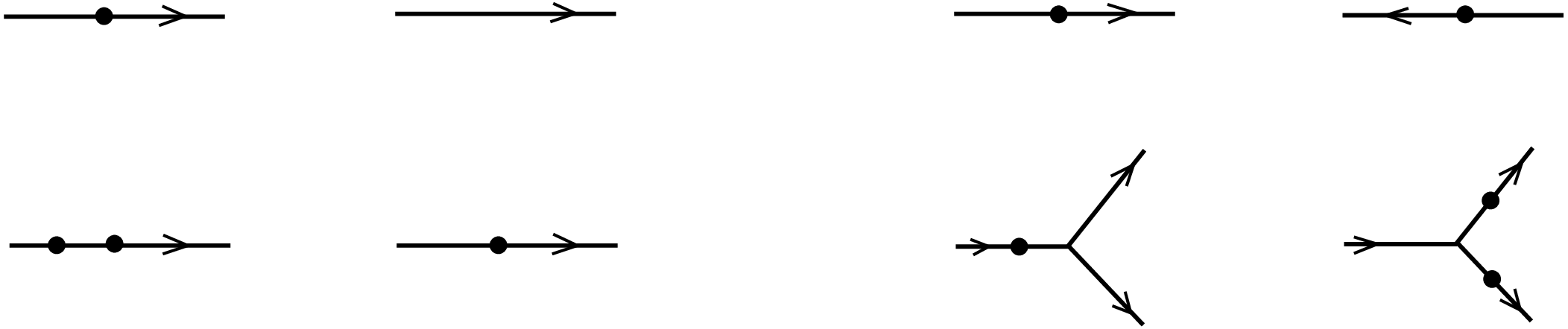}
$$
(Here the antipode $S(x)$ of an  $x \in  H$ is denoted by $\overline{x}$
and the coproduct $\Delta(x)=x' \otimes x''$ is expanded
using Sweedler's notation.)
\item The \emph{bead-out} relation takes
place in a neighborhood of a leaf, and it involves
the $H$-action on $V$:
$$
\labellist
\small\hair 2pt
 \pinlabel {$x$} [r] at 1 50
 \pinlabel {$v$} [t] at 8 3
 \pinlabel {$=$} at 144 93
 \pinlabel {${}^x v$} [t] at 240 4
\endlabellist
\centering
\includegraphics[scale=0.2]{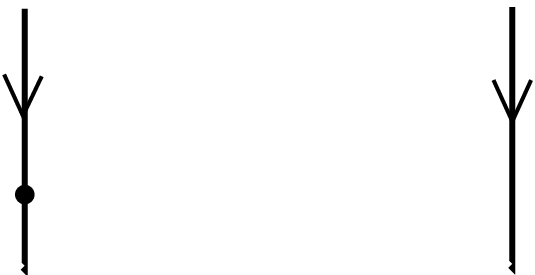}
$$
\end{itemize}
We obtain a graded $\Q$-vector space
$$
\mathcal{D}(V,H) = \bigoplus_{k=1}^\infty \mathcal{D}_k(V,H)
$$
where the degree $k$ part of $\mathcal{D}(V,H)$ is spanned by the trees of degree $k$.

\begin{remark}\label{Jd}
 Oriented trees with beads
 are an instance of ``Jacobi diagrams with beads'',
 which  appear 
 in  the theory of finite-type invariants,
 see e.g$.$ \cite{GL01}, \cite{HM21}. 
\end{remark}

We will define a $\Q$-vector space $\mathcal{K}(V,H)$, which is isomorphic to $\mathcal D(V,H)$ by Proposition \ref{prop:eta} below.
The (left) $H$-action on $V$ extends to that on $\Lie(V)$, the free Lie algebra  on $V$, by the inductive rule
$$x\cdot[u,v]=[x'\cdot u,x''\cdot v]\quad(x\in H,\;u,v\in\Lie(V)).$$
Let $V^r$ denote $V$ 
with the right $H$-action 
defined by $v^h:= {}^{\overline{h}} v$
($v\in V$, $h\in H$).
Since the Lie bracket $[-,-]:V\ot_\Q\Lie(V)\to\Lie(V)$ is a $H$-module map, it induces a $\Q$-linear map
$$
\beta: V^r \otimes_{H} \Lie(V) \simeq \big(V\ot_\Q\Lie(V)\big)_H
\longrightarrow 
\Lie(V)_H
$$
where $W_H= W\big/\big(\ker(\varepsilon)\cdot W\big)$ 
denotes the space of $H$-coinvariants of $W$ for any $H$-module $W$.
Then, we define the $\Q$-vector space
$$
\mathcal{K}(V,H) := \ker(\beta) \ \subset \  V^r \otimes_{H}\Lie(V),
$$
which is graded
with  degree $k$ part
$\mathcal K_k(V,H)=\mathcal K(V,H)\cap (V^r\ot_H \Lie_k(V)).$

\nc\D{\mathcal D}
\nc\DVH{\D(V,H)}
Let $D$ be a tree in $\DVH$ of degree $k$, and $\ell$ a leaf of $D$.
Let $ \col(\ell) \in V$ denote the color of $\ell$. 
We reorient the tree $D$ (if necessary) by using the second Hopf relation
to have all edges of $D$ oriented \emph{outwards} $\ell$,
and we then let $\word(\ell) \in \Lie_{k}(V)$ 
denote the Lie word that is encoded  by the tree $D$ rooted at $\ell$,
where each bead  colored by an $x\in H$ is interpreted 
as a (left) action of $x$ on the outer subtree
and each node is interpreted as a Lie bracket
according to the following inductive rule:
$$
\begin{array}{c} \labellist
\small\hair 2pt
\pinlabel {root} [r] at 0 110
\pinlabel {left} [r] at 350 190
\pinlabel {right} [r] at 360 40
\endlabellist
\centering
\includegraphics[scale=0.24]{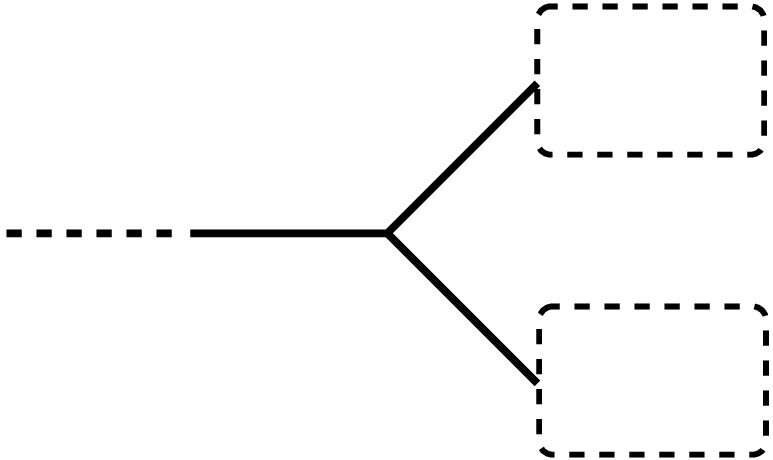}
\end{array}
\leadsto \big[\operatorname{left}, \operatorname{right}\!\big]
$$
Then set
\begin{equation}   \label{eq:eta_iso}
\eta(D) := \sum_{\ell:\operatorname{leaf} }   \col(\ell)  \otimes \word(\ell)
\ \in V^r \otimes_{H} \Lie_{k}(V).
\end{equation}

\begin{example} For instance, we have
\begin{eqnarray*}
\eta\Bigg(\ \
\begin{array}{c}\labellist
\small\hair 2pt
\pinlabel {${a}$} [r] at 2 150
 \pinlabel {${d}$} [r] at 2 5
\pinlabel {${b}$} [l] at 293 149
\pinlabel {${c}$} [l] at 294 2
\pinlabel {$x$} [b] at 53 111
\pinlabel {$y$} [b] at 153 85
\pinlabel {$z$} [tr] at 255 48
\endlabellist
\centering
\includegraphics[scale=0.24]{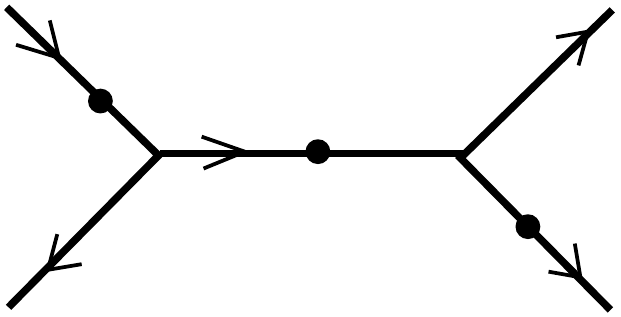}\end{array}\ \ \Bigg) &= &
a  \otimes {}^x \big[{}^y[{b},{}^z{c}],{d}\big]
+ b\otimes  \big[{}^z{c} ,{}^{\overline{y}}[ {d}, {}^{\overline{x}}{a}] \big]\\
&& 
+ c\otimes {}^{\overline{z}}\big[{}^{\overline{y}}[{d},{}^{\overline x}a],{b}\big] 
+ d \otimes \big[{}^{\overline{x}}{a},{}^y[b,{}^z{c}]\big].
\end{eqnarray*}
\end{example}

\begin{proposition} \label{prop:eta}
Assume that $\operatorname{Tor}_1^H\big(\ker(\varepsilon),\Lie(V)\big)=0$.
Then \eqref{eq:eta_iso} defines an  isomorphism  
$\eta:\mathcal{D}(V,H)  \to \mathcal{K}(V,H)$
of  graded $\Q$-vector spaces.
\end{proposition}

\begin{remark}
The $\Q$-vector space
$\operatorname{Tor}_1^H\big(\ker(\varepsilon),\Lie(V)\big)
\simeq H_2\big(H;\Lie(V)\big)$
vanishes if $V$ is a projective $H$-module
or if we have $H=\Q[F]$ with $F$ a free group.
\end{remark}

\begin{proof}[Proof of Proposition \ref{prop:eta}]
Consider the following graded $\Q$-vector space, which depends only on $V$:
$$
\mathcal{D}(V) :=
\frac{\Q \cdot \left\{
\hbox{node-oriented trees on $V$}\right\}}{\hbox{AS, IHX, multilinearity}}.
$$
Every generator $D$ of $\mathcal{D}(V)$ of degree $k$
is transformed to an element 
$\eta(D)\in {V \otimes_{\Q} \Lie_{k}(V)}$
using the same formula as  \eqref{eq:eta_iso}.
 It is well known 
 that one defines in this way
a $\Q$-linear isomorphism 
\begin{equation} \label{eq:eta_first}
\eta:\mathcal{D}(V)  \longrightarrow 
\mathcal{K}(V)
\quad
\hbox{where }
\mathcal{K}(V) 
:= \ker\big([-,-]: V \otimes_\Q \Lie(V) \to \Lie(V) \big),
\end{equation}
see e.g$.$ \cite{HP}.

The (left) action of $H$ on $V$ extends to that on
$\mathcal{D}(V)$ using the coproduct of $H$.
Consider the space of $H$-coinvariants
\begin{equation} \label{eq:with_coinvariants}
\mathcal{D}(V)_H = 
{\mathcal{D}(V)}\,/\,{\ker(\varepsilon)\cdot \mathcal{D}(V)}.
\end{equation}
Since  \eqref{eq:eta_first} is a $H$-module isomorphism,
it induces a linear isomorphism 
$\eta_H:\mathcal{D}(V)_H  \to \mathcal{K}(V)_H$.
By the snake lemma we easily see  that
the kernel of the canonical map
$\mathcal{K}(V) \to \mathcal{K}(V,H)$
is the kernel of 
$$
\ker(\varepsilon)\otimes_H [-,-]: 
\ker(\varepsilon)\cdot \big(V \otimes_\Q\Lie(V)\big)
\lto \ker(\varepsilon)\cdot  \Lie(V) 
$$
and, by the flatness assumption, this is $ \ker(\varepsilon)\cdot \mathcal{K}(V)$.
Therefore, we have a canonical isomorphism 
$\mathcal{K}(V)_H \simeq \mathcal{K}(V,H)$,
so that we can view $\eta_H$
as an isomorphism between $\mathcal{D}(V)_H$ and $\mathcal{K}(V,H)$. 

We have a natural $\Q$-linear map  $\mathcal{D}(V) \to \mathcal{D}(V,H)$.
The Hopf and bead-out relations in $\mathcal{D}(V,H)$ 
imply that this map is surjective and vanishes on $\ker(\varepsilon)\cdot\mathcal{D}(V)$.
Therefore, it induces a surjective $\Q$-linear map
$u:\mathcal{D}(V)_H \to \mathcal{D}(V,H)$.

It is easily checked that  \eqref{eq:eta_iso}
defines a map $\eta:\mathcal{D}(V,H) \to V^r \otimes_{H} \Lie(V)$,
and, clearly, we have $\eta \circ u = \eta_H$.
Since $u$ is surjective and $\eta_H$ is an isomorphism,
we conclude that $\eta$ is an isomorphism onto
the subspace $\mathcal{K}(V,H)$ of $V^r \otimes_{H} \Lie(V)$.
\end{proof}


\subsection{Diagrammatic description of $D^0_+ \simeq D^1_+$
with rational coefficients}

We now restrict to the case where $H:=\Q[F]$ 
and $V:= \AAlpha^\Q = \AAlpha\otimes \Q$.
Thus, we consider the space
$$
\mathcal{D} := \mathcal{D}(\AAlpha^\Q, \Q[F])
$$
of trees on $\AAlpha^\Q$ with $\Q[F]$-beads.
Note that $ \mathcal{K}(\AAlpha^\Q, \Q[F])$ is isomorphic
to $D^0_+ \otimes \Q$ via the identification $\vartheta^{-1}$ of 
Lemma \ref{lem:vartheta}. Hence, by  Proposition \ref{prop:eta},
we get  an isomorphism of graded $\Q$-vector spaces
\begin{equation} \label{eq:eta_bis}
\eta: \mathcal{D} \stackrel{\simeq}{\longrightarrow} D^0_+\ot\Q.
\end{equation}
The following describes the derivation of $\Lie(\AAlpha^\Q)$
corresponding to an element of~$\mathcal{D}$.
(Again, we use the pairings 
$\langle-,-\rangle: \Z[F] \times \AAlpha \to \Z[F]$ and 
$\Psi: \AAlpha \times \AAlpha \to \Z[F] \otimes \AAlpha$
presented in~\S \ref{subsec:hif_handlebody}
 and~\S\ref{subsec:Psi}, respectively).)

\begin{proposition} \label{prop:action}
Let $D$ be a tree on $\AAlpha^\Q$ with $\Q[F]$-beads,
and let $d\in D^1_+\otimes \Q$ be the derivation 
corresponding to $\eta(D) \in D^0_+\otimes \Q$. 
Then, for any $a\in \AAlpha^\Q$, we have
\begin{eqnarray}
\label{eq:d(a)} d(a) &=&
-\sum_v  \Big[\overline{\Psi(\col(v),a)^\ell}  
\cdot \word(v) , \Psi(\col(v),a)^r  \Big]\\
\notag && + \sum_b \overline{\langle \col(b)',a\rangle} 
\cdot \left[\word^\ell(b),\col(b)''\cdot \word^r(b)\right].
\end{eqnarray}
Here the first sum is over all leaves $v$ of $D$
and uses the same notations as \eqref{eq:eta_iso},
and the second sum is over all beads $b$ of $D$.
The element
$\col(b)\in \Q[F]$ is the color of $b$,
while the elements $\word^\ell(b),\word^r(b)\in\Lie(V)$ are 
represented by the two half-trees obtained
by cutting $D$ at $b$
(orienting all the edges of these half-trees outwards $b$,
deleting the bead $b$ from both,
and assuming that 
$\word^\ell(b)$ comes before $\word^r(b)$
if one follows the original orientation of the edge around $b$).
\end{proposition}

\begin{proof}
If $D$ has no bead, then \eqref{eq:d(a)} immediately
follows from \eqref{eq:eta_iso} and \eqref{eq:d_formula}.
Since the $\Q$-vector space $\mathcal{D}$ 
is generated by trees without bead, 
it suffices to show that (for $a$ fixed)
the right-hand side of \eqref{eq:d(a)} 
defines a $\Q$-linear form on $\mathcal{D}$.
To prove this, we need to check that each defining relation
of $\mathcal{D}$ is mapped to $0$.
This consists in  straightforward verification
whose key arguments are given in Table \ref{tab:arg}.
\begin{table}[h] 
\begin{tabular}{|c|c|} \hline
\textbf{Defining relation of $\mathcal{D}$} & \textbf{Argument}  \\
\hline
multilinearity 
&  $\Q$-bilinearity of $\Psi$ and $\langle-,-\rangle$ \\
\hline 
$\begin{array}{c}  \\[-0.2cm] \includegraphics[scale=0.3]{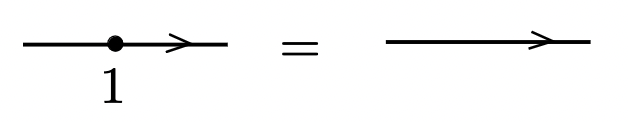}\end{array} $ &   
Property \eqref{eq:Fox_left} \\
\hline
$\begin{array}{c}  \\[-0.2cm] \includegraphics[scale=0.3]{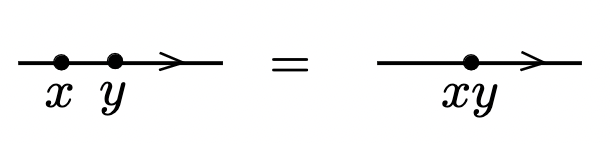}\end{array} $ &   
Property \eqref{eq:Fox_left} \\
\hline
$\begin{array}{c}  \\[-0.2cm] \includegraphics[scale=0.3]{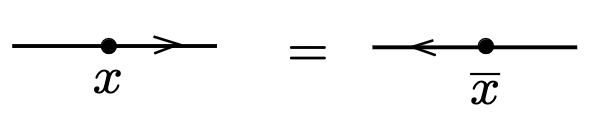}\end{array} $ & 
Property \eqref{eq:Fox_left}  \\
\hline
$\begin{array}{c}  \\[-0.2cm] \includegraphics[scale=0.3]{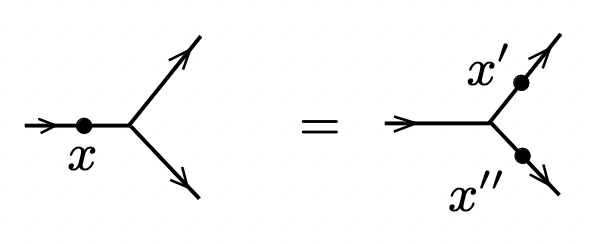}\end{array} $ &   
Axioms of $\Lie(\AAlpha)$  \\
\hline
IHX  & Axioms of $\Lie(\AAlpha)$  \\
\hline
AS & Axioms of $\Lie(\AAlpha)$   \\
\hline
$\begin{array}{c}  \\[-0.2cm] \includegraphics[scale=0.25]{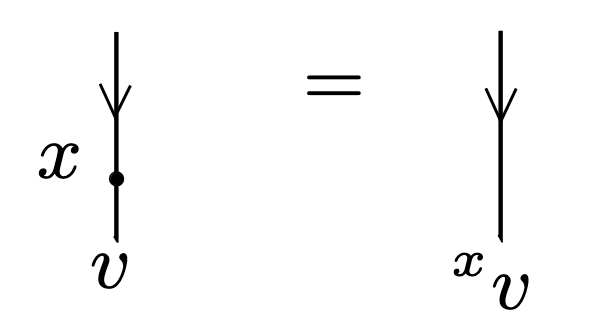}\end{array} $  & 
Property\eqref{eq:conj_left} \& formula \eqref{eq:<>_to_Theta}
\\ \hline
\end{tabular}\\[0.5cm]
\caption{} \label{tab:arg}
\end{table}
\end{proof}

The Lie bracket on $D^0_+\otimes \Q$ transports through $\eta$
to a Lie bracket on $\mathcal{D}$.
We now aim at giving an explicit description of this Lie bracket $[D,E]$ 
for any two trees $D,E$ in~$\mathcal{D}$.
For any leaves $v$ and $w$ of $D$ and $E$, respectively,
let $D \stackrel{v,w}{\vee} E$ be the $\Q$-linear combination
of trees defined by ``branching'' as follows: 
$$
\labellist
\small\hair 2pt
 \pinlabel {$v$} [t] at 74 90
 \pinlabel {$w$} [t] at 178 93
 \pinlabel {$\Psi(v,w)^\ell$} [r] at 589 169
 \pinlabel {$\Psi(v,w)^r$} [ ] at 647 1
 \pinlabel {$\leadsto$} [r] at 406 94
  \pinlabel {.} [l] at 756 94
\endlabellist
\centering
\includegraphics[scale=0.25]{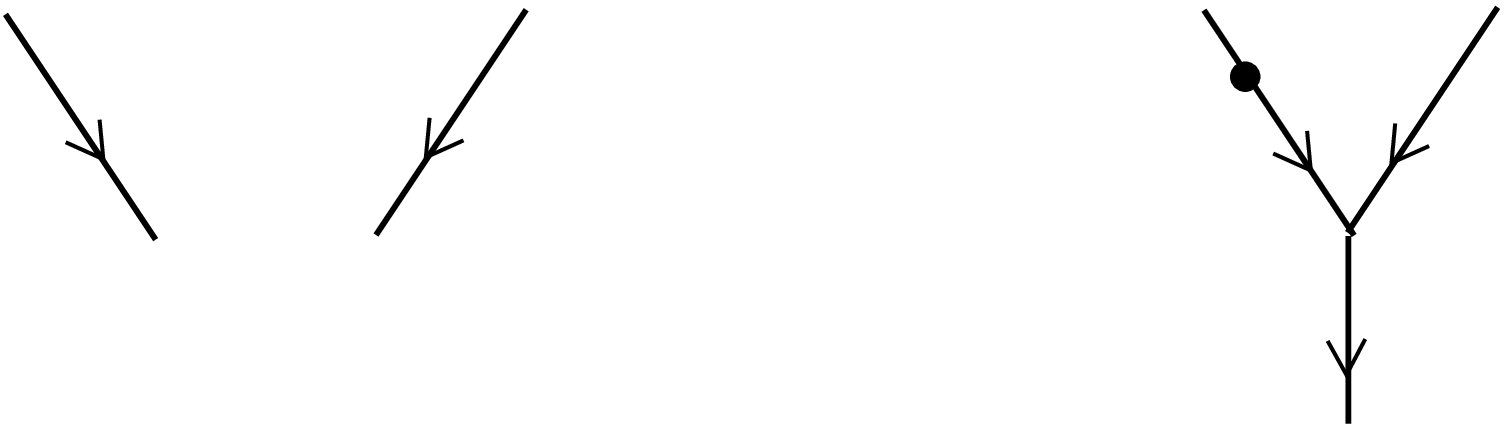} 
$$
Here we use the rationalization 
$\Psi: \AAlpha^\Q  \times \AAlpha^\Q   \to \Q[F] \otimes_\Q \AAlpha^\Q$
of the intersection operation \eqref{eq:Psi},
and we have omitted the symbols $\col(-)$ on the right-hand side
for simplicity.
Similarly, for any bead $b$ of $D$ and for any leaf $w$ of $E$,
let $D \stackrel{b,w}{\bot} E$ be the $\Q$-linear combination
of trees defined by ``grafting'' as follows:
$$
{\labellist
\scriptsize \hair 2pt
 \pinlabel {$w$} [t] at 109 76
 \pinlabel {$b$} [t] at 108 8
 \pinlabel {$\Theta(b,w)^r $} [t] at 538 8
 \pinlabel {$\Theta(b,w)^\ell$} [t] at 686 6
 \pinlabel {$b''$} [t] at 1200 6
 \pinlabel {$\overline{\langle b',w\rangle}$} [l] at 1130 98
 \pinlabel {$\leadsto$} at 368 50
\pinlabel {{\normalsize $\Bigg($}} at 900 50
\pinlabel {{\normalsize $\Bigg)$}} at 1280 50
 \pinlabel {$=$} at 950 50
  \pinlabel {.} at 1310 50
\endlabellist
\centering
\includegraphics[scale=0.25]{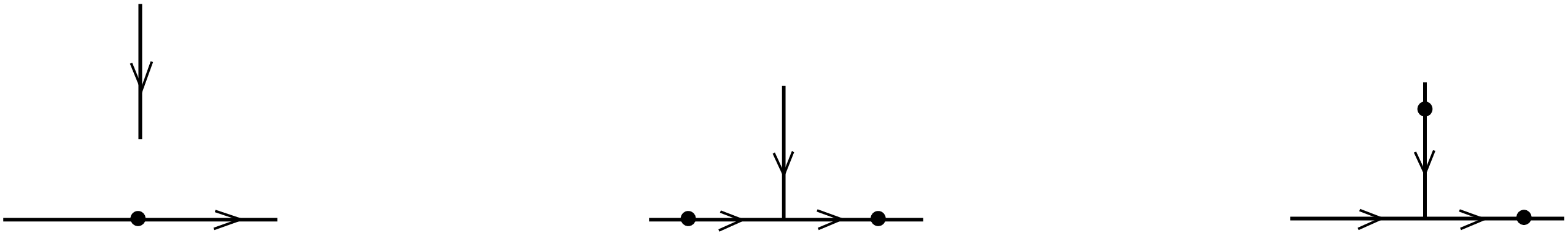}}\\[0.2cm]
$$
Here we use the rationalization 
$\Theta: \Q[F]  \times \AAlpha^\Q  \to \Q[F] \otimes_\Q \Q[F] $
of the intersection operation \eqref{eq:Theta},
and the indicated identity follows from \eqref{eq:<>_to_Theta}.

\begin{theorem} \label{th:tree_bracket}
Let $D,E$ be oriented trees on $\AAlpha^\Q$ with $\Q[F]$-beads. 
Then we have
\begin{equation} \label{eq:tree_bracket}
[D,E] = \sum_{v,w} D \stackrel{v,w}{\vee} E -
\sum_{b,w} D \stackrel{b,w}{\bot} E 
+ \sum_{v,c} E \stackrel{c,v}{\bot} D  \ \in \mathcal{D}   , 
\end{equation}
where $v$ (resp. $w$) 
runs over all leaves of $D$ (resp. $E$), 
and $b$ (resp. $c$) 
runs over all beads of $D$ (resp. $E$).
\end{theorem}

\begin{proof}
We claim that the right-hand side of \eqref{eq:tree_bracket}
defines a binary operation in the space  $\mathcal{D}$.
To prove this, we need to verify that each defining relation of $\mathcal{D}$ is mapped to $0$.
This is a straightforward verification, which is left to the reader.
Table \ref{tab:arguments} gives the key arguments 
that are involved for each of those relations.

\begin{table}[h] 
\begin{tabular}{|c|c|} \hline
\textbf{Defining relation of $\mathcal{D}$} & \textbf{Argument}  \\
\hline
multilinearity 
&  $\Q$-bilinearity of $\Theta$ and $\Psi$ \\
\hline 
$\begin{array}{c}  \\[-0.2cm] \includegraphics[scale=0.3]{bead_one}\end{array} $ &   
Property \eqref{eq:left_Theta} \\
\hline
$\begin{array}{c}  \\[-0.2cm] \includegraphics[scale=0.3]{bead_product}\end{array} $ &   
Property \eqref{eq:left_Theta}  \\
\hline
$\begin{array}{c}  \\[-0.2cm] \includegraphics[scale=0.3]{bead_inversion}\end{array} $ & 
Property \eqref{eq:inversion}\\
\hline
$\begin{array}{c}  \\[-0.2cm] \includegraphics[scale=0.3]{bead_coproduct}\end{array} $ &   IHX   \\
\hline
IHX  & IHX   \\
\hline
AS & AS   \\
\hline
$\begin{array}{c}  \\[-0.2cm] \includegraphics[scale=0.25]{beadout}\end{array} $  & 
Properties \eqref{eq:right_Theta}, \eqref{eq:conj_left} \& \eqref{eq:conj_right} 
\\ \hline
\end{tabular}\\[0.5cm]
\caption{} \label{tab:arguments}
\end{table}

We now aim at proving that the Lie bracket
$[D,E]$ is equal to the right-hand side of \eqref{eq:tree_bracket}.
Since the $\Q$-vector space $\mathcal{D}$
is generated by trees without bead,
we can assume  that $D$ and $E$ have no bead.
Set  $c:=\eta(D)$ and $g:=\eta(E)$,
and let $d,e \in D^1_+\otimes \Q$ be the derivations corresponding
to $c,g\in D^0_+\otimes \Q$, respectively.
Then, by Proposition \ref{prop:action}, we have
$$
d(a)= -\sum_v  \Big[\overline{\Psi(v,a)^\ell}  \cdot \word(v) , \Psi(v,a)^r  \Big] \quad
\text{for all }a \in \AAlpha^\Q,
$$
where the sum is over all leaves $v$ of $D$
and we have denoted $\col(v)$ simply by $v$.
Therefore, for any  $a \in \AAlpha^\Q$,
we get $e(d(a))=L(a)+M(a)+N(a)$ with
\begin{eqnarray*}
L(a) &:=& -\sum_v  \left[\overline{\Psi(v,a)^\ell}  \cdot e(\word(v)), \Psi(v,a)^r  \right] \\
M(a) &:= & -\sum_v  
\left[ \left[ g\big(\big(\overline{\Psi(v,a)^\ell}\big)'\big)  ,
\big(\overline{\Psi(v,a)^\ell}\big)'' \cdot \word(v) \right], \Psi(v,a)^r  \right] \\
&=&  \sum_v  
\left[ \overline{\Psi(v,a)^{\ell'}} \cdot  \left[ g\big({\Psi(v,a)^{\ell''}}\big)  ,
\word(v) \right], \Psi(v,a)^r  \right]  \\
N(a)&:=&  -\sum_v  \left[\overline{\Psi(v,a)^\ell}  \cdot \word(v), 
e\big(\Psi(v,a)^r\big)  \right] ,
\end{eqnarray*}
where, in the sum $M(a)$, we have used Sweedler's notation for the coproduct
of $\overline{\Psi(v,a)^\ell}$.
We get a similar formula $d(e(a))=P(a)+Q(a)+R(a)$  
by exchanging the roles of $D$ and $E$.
It follows from the above computation that the Lie bracket $[d,e]$ of the derivations $d,e$ maps any $ a \in \AAlpha^\Q$ to 
\begin{equation} \label{eq:de-ed}
d(e(a))-e(d(a)) =  P(a)+Q(a)+R(a)-L(a)-M(a)-N(a).
\end{equation}

Since $D$ and $E$ have no bead, the right-hand side of \eqref{eq:tree_bracket}
reduces to the sum $\sum_{v,w} D \stackrel{v,w}{\vee} E$.
Besides, according to Proposition \ref{prop:action},
the derivation in $D^1_+\otimes \Q$ corresponding to 
$\eta\big(\sum_{v,w} D \stackrel{v,w}{\vee} E\big)$ 
maps an $a \in \AAlpha^\Q$ to $\big(X(a)+Y(a)+Z(a)\big)+W(a)$, where
\begin{eqnarray*}
X(a)&=&-\sum_{v,w}  
\Big[\overline{\Psi\big(\Psi(v,w)^r,a\big)^\ell}  \cdot 
[\overline{\Psi(v,w)^\ell}\cdot \word(v),  \word(w)]
, \Psi\big(\Psi(v,w)^r,a\big)^r  \Big],\\
Y(a)&=& -\sum_{v,w,v'\neq v}  \Big[\overline{\Psi(v',a)^\ell}  \cdot
\big(\word(v')\big\vert_{v\mapsto \Psi(v,w)^\ell \cdot 
[\word(w),\Psi(v,w)^r] } \big), \Psi(v',a)^r  \Big], \\
Z(a)&=& -\sum_{v,w,w'\neq w} \Big[\overline{\Psi(w',a)^\ell}  \cdot
\big(\word(w')\big\vert_{w\mapsto[\Psi(v,w)^r, 
\overline{\Psi(v,w)^\ell}\cdot \word(v)]} \big)
, \Psi(w',a)^r  \Big],\\
W(a)&=& \sum_{v,w} \overline{\langle \Psi(v,w)^{\ell'}, a\rangle} \cdot
\left[\word(v) , \Psi(v,w)^{\ell''} \cdot[ \word(w), \Psi(v,w)^r] \right]. 
\end{eqnarray*}
In the sum $Y(a)$, the third index
runs over all leaves $v'$ of $D$ different from $v$, 
and the notation $\word(v')\vert_{v\mapsto u}$
means that $\word(v')\in \Lie(\AAlpha^\Q)$ is transformed into
another word by inserting $u\in \Lie(\AAlpha^\Q)$
in place of the letter $v$; a similar remark applies to the sum $Z(a)$.
It is easily seen that $Z(a)=P(a)$
and, using \eqref{eq:kind_of_sym}, that $Y(a)=-L(a)$.
Therefore, by comparison with \eqref{eq:de-ed}, it suffices to prove
the following identity:
\begin{equation}
\label{eq:final_identity}  Q(a)+R(a)-M(a)-N(a)= X(a)+W(a).
\end{equation}

To prove this, we come back to 
$M(a)+N(a)$ and  we simplify our notations further  by denoting, for any leaf $v$,  
the corresponding element $\word(v)$ of $\Lie(\AAlpha^\Q)$ 
by the corresponding upper-case letter $V$:
\begin{eqnarray*}
&& M(a)+N(a)    \\
&=& \sum_{v,w}  
\left[\,  \overline{\Psi(v,a\big)^{\ell'}} \cdot   
\left[ \big\langle  \Psi(v,a\big)^{\ell''}, w\big\rangle \cdot W , V \right], 
\Psi(v,a)^r  \right]\\
&& + \sum_{v,w}  \left[\overline{\Psi(v,a)^\ell}  \cdot V, 
\left[ \overline{\Psi(w,\Psi(v,a)^r)^\ell} \cdot W , 
\Psi(w,\Psi(v,a)^r)^r \right]  \right]  \\
&=& -\sum_{v,w}   \left[  \overline{\Psi(v,a\big)^{\ell'''}} \cdot   V ,
\left[ \left(\overline{\Psi(v,a\big)^{\ell'}}  
\big\langle  \Psi(v,a\big)^{\ell''}, w\big\rangle \right)\cdot W, 
\Psi(v,a)^r  \right]  \right]\\
&& + \sum_{v,w}   \left[  
\left(\overline{\Psi(v,a\big)^{\ell'}}  
\big\langle  \Psi(v,a\big)^{\ell''}, w\big\rangle \right)\cdot W , 
\left[ \
\overline{\Psi(v,a\big)^{\ell'''}} \cdot   V , \Psi(v,a)^r \right] \right]\\
&& + \sum_{v,w}  \left[\overline{\Psi(v,a)^\ell}  \cdot V, 
\left[ \overline{\Psi(w,\Psi(v,a)^r)^\ell} \cdot W , 
\Psi(w,\Psi(v,a)^r)^r \right]  \right] .
\end{eqnarray*}
Therefore, a symmetric computation for $Q(a)+R(a)$ gives
\begin{eqnarray*}
&& M(a)+N(a)-Q(a)-R(a) \\
&=& \Bigg(  -\sum_{v,w}   \left[  \overline{\Psi(v,a\big)^{\ell'''}} \cdot   V ,
\left[ \left(\overline{\Psi(v,a\big)^{\ell'}}  
\big\langle  \Psi(v,a\big)^{\ell''}, w\big\rangle \right)\cdot W, 
\Psi(v,a)^r  \right]  \right]\\
&& + \sum_{v,w}  \left[\overline{\Psi(v,a)^\ell}  \cdot V, 
\left[ \overline{\Psi(w,\Psi(v,a)^r)^\ell} \cdot W , 
\Psi(w,\Psi(v,a)^r)^r \right]  \right] \\
&& - \sum_{w,v}   \left[  
\left(\overline{\Psi(w,a\big)^{\ell'}}  
\big\langle  \Psi(w,a\big)^{\ell''}, v\big\rangle \right)\cdot V , 
\left[ \
\overline{\Psi(w,a\big)^{\ell'''}} \cdot   W , \Psi(w,a)^r \right] \right] \Bigg)\\
&& -\big(\hbox{the symmetric counterpart $D\leftrightarrow E$}\big).
\end{eqnarray*}
Next, a double application of \eqref{eq:pseudo-Jacobi_bis} leads to
\begin{eqnarray*}
&& M(a)+N(a)-Q(a)-R(a) \\
&=& \sum_{v,w} \left[ \overline{\Psi(\Psi(w,v)^r,a)^{\ell'}} \cdot V, 
\left[ \overline{\Psi(w,v)^\ell\, \Psi(\Psi(w,v)^r,a)^{\ell''}} \cdot W  , 
\Psi(\Psi(w,v)^r,a)^r\right]\right] \\
&& - \sum_{w,v} \left[ \overline{\Psi(\Psi(v,w)^r,a)^{\ell'}} \cdot W, 
\left[ \overline{\Psi(v,w)^\ell\, \Psi(\Psi(v,w)^r,a)^{\ell''}} \cdot V  , 
\Psi(\Psi(v,w)^r,a)^r\right]\right].
\end{eqnarray*}
Here, the first sum can be transformed by applying 
\eqref{eq:kind_of_sym} to $\Psi(w,v)$ 
and by using \eqref{eq:conj_left} next, 
which results in the following identity:
\begin{eqnarray*}
&& M(a)+N(a)-Q(a)-R(a) \\
&=& -\sum_{v,w} \Bigg[\overline{\Theta(\Psi(v,w)^{\ell ''},a)^{r'}} \cdot V, \\
&& \qquad \qquad 
\left[\left(\overline{\Theta(\Psi(v,w)^{\ell ''},a)^{r''}}\, \Psi(v,w)^{\ell'}
\right) \cdot W,
\Theta(\Psi(v,w)^{\ell''},a)^\ell \cdot \Psi(v,w)^r\right]\Bigg]\\
&& + \sum_{v,w} \left[ \overline{\Psi(v,w)^{\ell}\, 
\Psi(\Psi(v,w)^r,a)^{\ell'}} \cdot V, 
\left[ \overline{ \Psi(\Psi(v,w)^r,a)^{\ell''}} \cdot W  , 
\Psi(\Psi(v,w)^r,a)^r\right]\right] \\
&& - \sum_{w,v} \left[ \overline{\Psi(\Psi(v,w)^r,a)^{\ell'}} \cdot W, 
\left[ \overline{\Psi(v,w)^\ell\, \Psi(\Psi(v,w)^r,a)^{\ell''}} \cdot V  , 
\Psi(\Psi(v,w)^r,a)^r\right]\right].
\end{eqnarray*}
In this last identity, the first term can be seen to coincide with $-W(a)$,
while the second and third terms give $-X(a)$.
This proves \eqref{eq:final_identity}.
\end{proof}

\begin{remark}
We can directly prove that the binary operation  \eqref{eq:tree_bracket} 
in $\mathcal{D}$ satisfies the axioms of a Lie bracket.
Since the space $\mathcal{D}$ is generated by trees without beads,
it is enough to verify the antisymmetry $[D,E]=-[E,D]$
and the Jacobi identity $[[D,E],F]+[[F,D],E]+[[E,F],D]=0$
for trees $D,E,F$ with no bead. Then,
the antisymmetry is easily deduced from \eqref{eq:kind_of_sym},
and the Jacobi identity can be proved by a long computation
using \eqref{eq:pseudo-Jacobi} and \eqref{eq:pseudo-Jacobi_bis}.
Note that, in the proof of Proposition~\ref{prop:Psi},
the latter identities are derived from \eqref{eq:quasi-Jacobi},
which is a form of ``quasi-Jacobi'' identity satisfied 
by the intersection double bracket \cite{MT14}.
\end{remark}

The space $\mathcal{D} = \mathcal{D}(\AAlpha^\Q, \Q[F])$ has another description,
which will be used in the next section in a few places.
Consider the space 
$$
\mathcal{D}' :=
\mathcal{D}\big(\Lie(\AAlpha^\Q), T(\AAlpha^\Q) \otimes_\Q \Q[F]\big)
$$
of trees on $\Lie(\AAlpha^\Q)$ 
with $\big(T(\AAlpha^\Q) \otimes_\Q \Q[F]\big)$-beads.
Here $T(\AAlpha^\Q) \otimes_\Q \Q[F]$ is viewed 
as the universal enveloping algebra of the extended graded Lie algebra 
$\overline{\Alpha}_\bullet^\Q$.
The adjoint action of this cocommutative Hopf algebra on itself
restricts to an action on $\Lie(\AAlpha^\Q)= \overline{\Alpha}_+^\Q$.
We consider the following two types of operations 
on an arbitrary tree  $D \in \mathcal{D}'$:
\begin{itemize}
\item The \emph{expansion} of $D$ at a leaf $\ell$  is the element 
$D_\ell \in \mathcal{D}'$
that is obtained from $D$ by representing $\col(\ell)\in \Lie(\AAlpha^\Q)$ 
as a linear combination of half-trees with $\AAlpha^\Q$-colored leaves.
(That $D_\ell$ is well-defined follows from the AS, IHX and multilinearity relations.)
\item The \emph{expansion} of $D$ at a bead $b$  is the element 
$D_b \in \mathcal{D}'$
that is obtained from $D$ by the modification
$$
\labellist
\small\hair 2pt
 \pinlabel {$b$} [b] at 181 34
  \pinlabel {$a_i^{(1)}$} [b] at 760 92
   \pinlabel {$a_i^{(2)}$} [b] at 830 92
 \pinlabel {$a_i^{(n)}$} [b] at 910 92
  \pinlabel {$\cdots$} at 870 50
 \pinlabel {$x_i$} [b] at 1004 34
\pinlabel {$\leadsto \qquad {\normalsize \sum_i}$} at 594 15
\endlabellist
\centering
\includegraphics[scale=0.25]{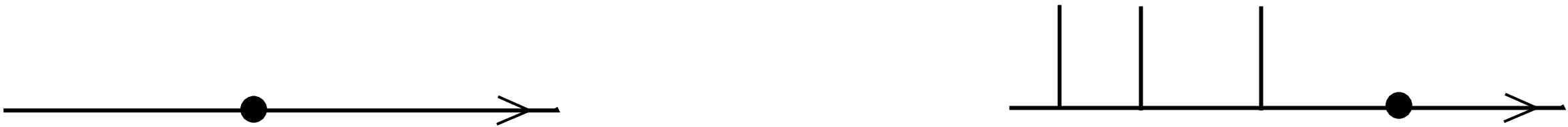}
$$
if $b$ is colored by $\sum_{i} a_i^{(1)}\cdots a_i^{(n_i)} x_i$
with $a_i^{(1)},\dots,a_i^{(n_i)} \in \AAlpha^\Q$ 
and $x_i\in \Q[F]$.
(That $D_b$ is well-defined follows from the multilinearity relations.)
\end{itemize}
Let $\mathcal{E}$ be the subspace of  $\mathcal{D}'$
generated by the differences 
$D-D_\ell$ and $D-D_b$,  for all trees
$D \in \mathcal{D}'$,
leaf $\ell$ of $D$ and bead $b$ of $D$. 
(In fact, it is easily checked from the ``bead-out'' relation
that the expansions of beads follow from expansions of leaves.)

\begin{lemma} \label{lem:T_tilde}
The $\Q$-linear map 
$f:\mathcal{D} \to \mathcal{D}' /\mathcal{E}$
that is induced by the inclusions $\AAlpha^\Q \subset \Lie(\AAlpha^\Q)$ and 
$\Q[F] \subset T(\AAlpha^\Q) \otimes_\Q \Q[F]$ is an isomorphism.
\end{lemma}

\begin{proof}
Clearly, $f$ is  surjective.     
For any tree 
$D \in \mathcal{D}'$,
let $e(D)\in\mathcal{D}$ be obtained from~$D$
by expanding all its leaves and all its beads at the same time.
It is easily checked that the assignment $D\mapsto e(D)$ defines
a $\Q$-linear map
$e:\mathcal{D}'/\mathcal{E}\to  \mathcal{D}$.
Clearly, $e \circ f$ is the identity.
\end{proof}

\section{Formulas, examples and applications}

\label{sec:formulas_examples}

In this  section, we provide some explicit formulas 
for Johnson homomorphisms,
which are based on the diagrammatic descriptions 
of Section \ref{sec:diagrams}. 
Specifically, for every $k\geq 1$, 
let $\tau_k^d:\modH_k\to\modD_k$ be the composition
\begin{equation} \label{eq:tau_k^d}
\xymatrix{
\modH_k \ar[r]^-{\tau_k^0} \ar@{-->}@/_2pc/[rrr]^-{\tau_k^d}  &
D^0_k  \ar@{>->}[r] &  D^0_k \otimes \Q &
\ar[l]_-\eta^-\simeq  \mathcal{D}_k(\AAlpha^\Q, \Q[F]) = \modD_k,
}\\[0.2cm]
\end{equation}
which may be regarded as a diagrammatic version of $\tau_k^0$.
As an application, we prove that certain quotients 
(of subgroups) of the twist group are not finitely generated.
We also identify the  restriction of the Johnson filtration  
to the  pure braid group, in relation with Milnor invariants.

\subsection{The first Johnson homomorphism on disk twists}

The first proposition of this section
computes the first Johnson homomorphism on a disk twist.

\begin{proposition} \label{prop:tau_1(disk_twist)}
For any properly embedded disk $U \subset V$, we have
\begin{equation} \label{eq:disk_twist}
\tau_1^d(T_{\partial U}) 
= - \frac{1}{2}[ u] \textsf{---} [ u]
\in \modD_1,
\end{equation}
where $u\in \Alpha$ is the homotopy class of the closed curve $\partial U$ 
(with an arbitrary orientation, and an arbitrary basing at $\star$)
and $[u]\in \AAlpha$ is the corresponding class.
\end{proposition}

\begin{proof}
Note that, as a consequence of the bead-out relation,
the right-hand side of \eqref{eq:disk_twist} does not depend on the choice of $u\in \Alpha$.
Let $U'$ be a  closed curve  in $\Sigma$ 
which is isotopic to $\partial U$  
and satisfies $U'\cap \partial \Sigma=\{\star \}$.
Hence, by orienting $U'$ arbitrarily 
and by regarding it as a loop based at $\star$,
we can take $u:=[U']\in\Alpha$.

We shall use the same notation as in \S \ref{subsec:hif_surface}.
In particular, let $\bullet$ be a second base-point in $\partial \Sigma$.
Let $x\in \pi$ and let $X$ be a loop based at $\bullet$
such that $(\overline{\partial \Sigma})_{\star\bullet }\, X\, 
(\partial \Sigma)_{\bullet \star}$ represents $x$.
We assume that $X$ meets $U'$ transversely in finitely many double points,
which are numbered $1,\dots,n$,
and appear in this order along~$X$. 
Then we have
$$
(T_{\partial U}(x))\, x^{-1} = \Big[
\prod_{i=1}^n (\overline{\partial \Sigma})_{\star\bullet} \,
X_{\bullet i} \, (U')^{\varepsilon_i}_i \, \overline{X}_{i \bullet}\,
(\partial \Sigma)_{\bullet \star} \Big] \in \pi,
$$
where $\varepsilon_i=\varepsilon_i(U',X)\in\{-1,+1\}$ 
is the sign of the intersection of $U'$
and $X$ at $i$, and $(U')_i^{\varepsilon_i}$ denotes the loop
$U'$ based at $i$, with the opposite orientation if $\varepsilon_i=-1$.
Hence the $1$-cocycle $\tau_1^0(T_{\partial U})$ maps $[x]\in F$ to 
\begin{equation} \label{eq:result}
\sum_{i=1}^n \varepsilon_i \left[  
(\overline{\partial \Sigma})_{\star\bullet} \,
X_{\bullet i} \, U'_i \, \overline{X}_{i \bullet}\,
(\partial \Sigma)_{\bullet \star} \right] \in \AAlpha.    
\end{equation}
Besides, using now the notations of  \S \ref{subsec:hif_handlebody},
we have 
\begin{eqnarray*}
\langle [x], [u] \rangle \ = \ \varpi \big(\eta(x,u)\big) &=&
-\sum_{i=1}^n \varepsilon_i
\left[(\overline{\partial \Sigma})_{\star\bullet } 
X_{\bullet i}  U'_{i \star}\right] \in \Z[F].
\end{eqnarray*}
Thus, the action of $-\langle [x], [u] \rangle$ on $[u]$ gives \eqref{eq:result},
which completes the proof of~\eqref{eq:disk_twist}.
\end{proof}

\begin{remark}
Recall from Example \ref{ex:0_1} that 
$\tau_1^0: \mathcal{T} \to Z^1(F,\AAlpha)$
is equivalent  to the Magnus representation 
$\Mag=\Mag_1^0: \mathcal{T}\to \operatorname{Mat}(g\times g;\Z[F])$.
Tracing  the sequence of isomorphisms \eqref{eq:sequence_iso}
and taking into account Proposition \ref{prop::hermitian},
it is easily checked that, for any $f\in \mathcal{T}$ 
with Magnus representation  $M=(m_{ij})_{i,j}$, 
\begin{equation} \label{eq:tau_M}
\tau_1^d(f) = 
-\frac{1}{2} \sum_{i,j=1}^g a_i  
\textsf{---}\!\!\!\!\!\stackrel{m_{ij}}{\bullet}\!\!\!\!
\textsf{---}\!\!\!\!>\!\!\!\!\textsf{---}  a_j.
\end{equation}
For instance, for the Dehn twist along the boundary curve, 
Proposition \ref{prop:tau_1(disk_twist)} gives 
$$
\tau_1^d(T_{\partial \Sigma})
= -\frac{1}{2} [\zeta] \textsf{---} [\zeta]
\stackrel{\eqref{eq:zeta_ab}}{=} -\frac{1}{2} \sum_{i,j=1}^g  
\big((1-x_i^{-1})\cdot a_i\big)\textsf{-----} \big((1-x_j^{-1})\cdot a_j\big)
$$
and we recover our previous computation \eqref{eq:Magnus_boundary} 
of $\Mag(T_{\partial \Sigma})$.
\end{remark}

\subsection{Infiniteness results}

Proposition \ref{prop:tau_1(disk_twist)} is the key
for a new proof of a result of 
McCullough \cite[Th$.$~1.2]{McCullough}.

\begin{theorem}[McCullough] \label{th:McCullough}
If $g\geq 2$, 
then the twist group $\mathcal{T}$ surjects onto 
a free abelian group of countably-infinite rank.
In particular, $\mathcal{T}$ is not finitely generated.
\end{theorem}

\begin{proof}
Let  $p:\Z[F] \to \Z[\Z]$ be the ring homomorphism
induced by the homomorphism $F\to \Z$ that maps every $x_i$ to $1$.
By reducing coefficients with $p$,
the Magnus representation induces a homomorphism
$\operatorname{Mag}^p: \mathcal{T} \to \operatorname{Mat}(g\times g;\Z[t^{\pm 1}])$.

The upper-left corner of $\operatorname{Mag}^p$ provides
a homomorphism $m:\mathcal{T} \to \Z[t^{\pm 1}]$. By Proposition \ref{prop::hermitian}, $m$ takes values in the subgroup
$S$ of $\Z[t^{\pm 1}]$ generated by $1$ and $t^n +t^{-n}$ for all $n\in \Z$.
It suffices to prove $m(\mathcal{T})=S$. 

For any $n\in \Z$, there is a simple closed
curve $\gamma_n$ in $\Sigma$, whose homotopy class 
(for an appropriate orientation and basing at $\star$) is of the form
$$
\alpha'_1 z_n \alpha_1^{-1} z_n^{-1} \in \Alpha \subset \pi,
$$
where $\alpha_1,\alpha'_1$ are the loops shown in \eqref{eq:loops}, 
and $z_n\in \pi$ satisfies $\varpi(z_n)= x_2^n\in F$. 
By Proposition \ref{prop:tau_1(disk_twist)}, we have 
$$
\tau_1^d(T_{\gamma_n}) = -\frac{1}{2}
(x_1^{-1}\cdot a_1- x_2^n\cdot a_1) \textsf{---} (x_1^{-1}\cdot a_1- x_2^n\cdot a_1),
$$
and we deduce from \eqref{eq:tau_M} that 
$$
m(T_{\gamma_n})= (t-t^{-n})(t^{-1}-t^n) = 2 - t^{n+1} -t^{-n-1}.
$$
Besides, by considering now the curve $\alpha_1$, we obtain
$$
\tau_1^d(T_{\alpha_1}) = -\frac{1}{2} a_1 \textsf{---} a_1, \quad
\hbox{hence } m(T_{\alpha_1})= 1.
$$
Thus, products of Dehn twists (and their inverses) along the curves $\alpha_1$
and $\gamma_n$ (for $n\in \Z$) realize any element of $S$.
\end{proof}

\begin{remark}
Some arguments similar to the proof of Theorem \ref{th:McCullough} 
show that  
$\operatorname{Mag}^p: \mathcal{T} \to \operatorname{Mat}\big(g\times g;\Z[t^{\pm 1}]\big)$
is surjective onto the subgroup of hermitian matrices, for $g\geq 3$.
In fact, the above proof 
is  inspired by the arguments in \cite{McCullough}, 
which involve an infinite cyclic cover of the handlebody~$V$.
\end{remark}

We now explain how the Johnson homomorphisms \eqref{eq:tau_k^d}
of higher degrees for the handlebody group 
can be used to obtain
further infiniteness results on subquotients of the twist group.
Let $\overline{\modT}_+$ be the graded Lie algebra associated 
with the lower central series of $\modT$,
and recall that $\overline{\modH}_+$ denotes the graded Lie algebra
associated with the Johnson filtration of $\modH$.

\begin{theorem} \label{th:infiniteness}
Assume that $g\geq 3$.
There exists a subgroup $\modL$ of $\modT$
which is free of countably-infinite rank and 
whose associated graded $\overline{\modL}_+=\Lie(\modL_{\operatorname{ab}})$
(with respect to its lower central series)
embeds both into $\overline{\modT}_+$ and $\overline{\modH}_+$.
Consequently, for every $k\geq 1$, 
 the abelian groups  $\Gamma_k\modT/\Gamma_{k+1}\modT$ and 
$\modH_k /\modH_{k+1}$ are of infinite rank.
\end{theorem}

\begin{proof}
Having fixed a system of curves $(\alpha,\beta)$ 
as in \eqref{eq:(a,b)}, let $a_i=[\alpha_i]\in \AAlpha$
and $x_i=\varpi(\beta_i)\in F$ for $i\in \{1,\dots,g\}$.
We start by constructing, for all $n\in \Z$, 
an element $\ell_n\in\modT$ satisfying
\begin{equation} \label{eq:tau_un}
\tau_1^d(\ell_n) = a_2 
\textsf{---}\!\!\!\!\!\stackrel{x_3^{-n}}{\bullet}\!\!\!\!
\textsf{---}\!\!\!\!>\!\!\!\!\textsf{---}a_1.
\end{equation}
Indeed, there is a simple closed
curve $\upsilon_n$ in $\Sigma$ whose homotopy class 
(for an appropriate orientation and basing at $\star$)
is
$$
 z_n \alpha_2 z_n^{-1} \alpha_1 \in \Alpha \subset \pi,
$$
where $z_n\in \pi$ satisfies $\varpi(z_n)= x_3^n\in F$. 
By Proposition \ref{prop:tau_1(disk_twist)}, we have
$$
\tau_1^d(T_{\upsilon_n}) = -\frac{1}{2}
(x_3^n \cdot a_2+a_1) \textsf{---} (x_3^n \cdot a_2+a_1).
$$
Thus, the element 
$\ell_n := T_{\upsilon_n}^{-1} T_{\alpha_1}  T_{\alpha_2}$
of $\modT$ has the desired property.
Let $\modL$ (resp$.$  ${L}$)
be the subgroup of $\modT$ (resp$.$ the free group)
generated by $\{\ell_n\,\vert \,n\in \Z\}$.
We denote by $\iota: {L} \to \modH$
the obvious homomorphism (with image $\modL$).

For every $k\geq 1$,
let $\modR_k$ be the subspace of $\modD_k$ generated by trees (of degree $k$)
with at least two leaves colored by elements of $\Q[F]\cdot a_1$,
and let $\modR:=\bigoplus_{k\ge1}\modR_k$. 
Viewing the Laurent polynomial ring 
$\Z[x_3^{\pm}]$ as a subring of $\Z[F]$, let
$$
S:= \Z[x_3^{\pm}]\cdot a_2
$$
be the $\Z[x_3^{\pm}]$-submodule of $\AAlpha$ generated by $a_2$.
By considering trees with a single leaf colored by $a_1$, no beads
and all other leaves colored by elements of $S\otimes \Q$,
and by viewing the unique $a_1$-leaf of such trees as a ``root'',
the $\Q$-vector space $\modD$ is seen
to contain one copy of the free Lie algebra $\Lie(S)$ on $S$. 
Besides, $\Lie(S)$ can also be viewed as a subspace of $\modD/\modR$ 
via the canonical projection $p:\modD \to \modD/\modR$.

Let $w$ be a non-associative word of length $k\geq 1$ in the single letter $\bullet$,
and let $n_1,\dots, n_k \in \Z$.
Interpreting $w$ as an iterated commutator, it defines an element
$$
W:= w(t_{n_1},\dots, t_{n_k}) \in  \Gamma_k \modL \subset \modH_k
$$
and, 
interpreting $w$ as in iterated Lie bracket, 
it also defines an element 
$$
W':= 
w\big(\tau_1^d(t_{n_1}),\dots, \tau_1^d(t_{n_k})\big) \in \modD_k.
$$
Since \eqref{eq:gr_gr} is a morphism of graded Lie algebras,
we have $\tau_k^d(W)=W'$. Furthermore, an easy computation based on
\eqref{eq:tau_un}, using Theorem \ref{th:tree_bracket} and the formulas 
$$
\Psi(a_i,a_j)=\delta_{ij} \otimes a_i, \quad 
\langle x_i, a_j\rangle =- \delta_{ij}
\qquad (i,j\in \{1,\dots,g\}),
$$
shows that
$$
W'\equiv \underbrace{w(x_3^{n_1}\cdot a_2, \dots, 
x_3^{n_k} \cdot a_2)}_{\in \Lie_k(S) 
\subset \modD_k} \mod \modR_k.
$$

Since the integer $k\geq 1$ and the non-associative word $w$
are arbitrary in the above paragraph,
 we have proved that the following diagram is commutative:
$$
\xymatrix{
\overline{L}_+ \ar[d]_-\simeq
\ar[r]^{\overline{\iota}_+} & \overline{\modH}_+ \ar[r]^{\overline{\tau}_+ } 
&  \modD \ar@{->>}[d]^-p \\
\Lie(S) \ar@{^{(}->}[rr]&& \modD/\modR
}
$$
Here, the isomorphism between $\overline{L}_+=\Lie(L_{\operatorname{ab}})$
and $\Lie(S)$ identifies each generator $\ell_n$ of the free group $L$
with the element $x_3^n\cdot a_2$ of $S$.
Thus,  the free Lie algebra
$\overline{L}_+$ embeds into $\overline{\modH}_+$
and, similarly, we can show that $\overline{L}_+$ 
embeds into $\overline{\modT}_+$.

It remains to observe that $\iota: L \to \modH$ is injective 
(so that $\modL$ is free). 
This follows  from the injectivity of $\overline{\iota}_+$
and the fact that $L$ is residually nilpotent.
\end{proof}

\begin{remark}
It is likely that Theorem \ref{th:infiniteness} can be extended 
to the genus $g=2$ case by considering the curves $\gamma_n$ 
(which were used in the proof of Theorem \ref{th:McCullough})
instead of the curves $\upsilon_n$.
\end{remark}

\subsection{An analogue of the Kawazumi--Kuno formula}

Proposition \ref{prop:tau_1(disk_twist)}
is fully generalized by the following result,
where the isomorphism of Lemma \ref{lem:T_tilde}
is implicit.

\begin{theorem} \label{th:KK_analogue}
For any special expansion $\theta$ 
of the free pair $(\pi,\AAlpha)$ 
and for any properly embedded disk $U \subset V$, we have
\begin{equation} \label{eq:analogue_KK}
\varrho^\theta(T_{\partial U})
= -\eta\Big( \frac{1}{2} \log \theta(u) \textsf{---} \log \theta(u)\Big),
\end{equation}
where $u\in \Alpha$ is the homotopy class of 
the closed curve $\partial U$  (with an arbitrary orientation
and basing at $\star$).
\end{theorem}

The proof involves the Kawazumi--Kuno formula 
for the logarithms of Dehn twists  \cite{KK14}.
To derive \eqref{eq:analogue_KK} from the  Kawazumi--Kuno formula,
we first need to relate precisely the notion of ``special expansion''
for the free pair $(\pi,\Alpha)$ 
to the notion of ``symplectic expansion'' for the free group $\pi$.
Recall from \cite{Mas12} that an \emph{expansion} of the free group $\pi$
is a monoid homomorphism $\tilde\theta: \pi \to \widehat{T}(H^\Q)$, with values
in the degree-completed tensor algebra on $H^\Q:=H_1(\pi;\Q)$, such that
$\tilde \theta(x)$ is group-like for every $x\in \pi$
and satisfies $\log \tilde \theta(x)=[x]+(\deg \geq 2)$.
Furthermore, 
the expansion $\tilde\theta$ is said to be \emph{symplectic} if 
\begin{equation}
\label{eq:symplectic_tilde}
\tilde\theta(\zeta)= \exp([\zeta]_2),
\end{equation}
where 
$[\zeta]_2\in \Gamma_2\pi /\Gamma_3 \pi \simeq \Lambda^2 H^\Q$
is regarded as a tensor of degree $2$.

\begin{lemma} \label{lem:special_to_symplectic}
There exists a  special expansion 
$\theta: \pi \to   \hat{T}(\AAlpha^\Q) \otimes_\Q \Q[F]$ of $(\pi,\Alpha)$
and  an injective $\Q$-algebra map
$\digamma: {T}(\AAlpha^\Q) \hat\otimes_\Q \Q[F] \to \widehat{T}(H^\Q)$
such that 
$\tilde \theta := \digamma \circ \theta$ is a symplectic expansion of $\pi$.
\end{lemma}

\begin{proof}
We consider the special expansion $\theta$ of $(\pi,\Alpha)$
that appears in the proof of Lemma \ref{lem:special}
starting from a special expansion $\theta_0$ of the free group~$D$.
We shall follow the notations of this proof,
but we will not specify the algebra map 
$q:  \widehat{T}(\mathbb{D}^\Q) \to  \widehat{T}(\AAlpha^\Q)$
when using it. 
Set  $a_i:=[\alpha_i] \in H^\Q$ and
$b_i:=[\beta_i] \in H^\Q$ for  $i\in \{1,\dots,g\}$.

Let
$\digamma: {T}(\AAlpha^\Q) \otimes_\Q \Q[F] \to \widehat{T}(H^\Q)$
be the $\Q$-algebra homomorphism defined by
\begin{equation}
\label{eq:digamma}
\digamma(a_i) := \Big(\frac{-\mathrm{ad}_{b_i}}{\exp(-\mathrm{ad}_{b_i})-\id}\Big)(a_i),
\quad \digamma(x_i) := \exp(b_i)    
\end{equation}
for all $i\in \{1,\dots,g\}$. 
Since $\digamma$ preserves the degree-filtrations,
it extends continuously to a complete $\Q$-algebra homomorphism
$\digamma: {T}(\AAlpha^\Q) \hat \otimes_\Q \Q[F] \to \widehat{T}(H^\Q)$.
We claim that, for an appropriate choice of $\theta_0$, the map
$\tilde \theta := \digamma \circ \theta$ is a symplectic expansion of $\pi$:
\begin{itemize}
\item[(i)] We have
\begin{eqnarray*}
\tilde \theta(\alpha_i)
&\stackrel{\eqref{eq:theta1}}{=}&
\digamma( \exp (u_i) \exp(a_i)  \exp (-u_i) \otimes 1) \\
&=& \digamma( \exp\big(\exp(\mathrm{ad}_{u_i}) (a_i)\big)  \otimes 1)\\ 
& = &  \exp\big(\exp(\mathrm{ad}_{\digamma(u_i)}) (\digamma(a_i))\big)  \otimes 1;
\end{eqnarray*}
since $u_i$ is primitive 
and since $\digamma$ preserves the primitive parts, 
the element $\tilde \theta(\alpha_i)$  of $\widehat{T}(H^\Q)$ is group-like;
besides, the above formula shows that  $\log \tilde \theta(\alpha_i)$ starts
like $\digamma(a_i)$ with $a_i$ in degree $1$. 
\item[(ii)] We have 
\begin{eqnarray*}
\tilde \theta(\beta_i)
&\stackrel{\eqref{eq:theta2}}{=}&
\digamma\big(\exp(u_i) \exp(-({}^{x_i}u'_i))\otimes x_i\big)\\
&=& \exp(\digamma(u_i)) \exp({b_i}) \exp(-\digamma(u_i'));
\end{eqnarray*}
by the same argument as in (i), the tensors $\digamma(u_i)$ and $\digamma(u'_i)$
are primitive,  so that the element $\tilde \theta(\beta_i)$  is group-like;
besides,
according to \eqref{eq:degree_one}, one can choose the special expansion $\theta_0$
so that the degree-one part of $u_i$ is 
$$
\frac{1}{2} a_i+\frac{1}{2}\sum_{j>i}(-a'_j+a_j)
$$
and is equal to the degree-one part of $u_i'$; 
therefore, the above formula shows that  $\log \tilde \theta(\beta_i)$ starts
with $b_i$ in degree $1$.
\item[(iii)] It follows from (i) and (ii) that $\tilde \theta$ is an expansion of $\pi$.
\item[(iv)] It remains to verify the symplectic condition \eqref{eq:symplectic_tilde}:
\begin{eqnarray*}
\tilde\theta(\zeta) & \stackrel{\eqref{eq:symplectic}}{=}& 
\digamma\Big(\sum_{i=1}^g \big(a_i -(a_i)^{x_i}\big)\Big) \\
&=& \sum_{i=1}^g \big( \digamma(a_i) - \exp(-b_i)\, \digamma(a_i) \exp(b_i)\big) \\
&=& \sum_{i=1}^g  \big(\id -\exp(-\mathrm{ad}_{b_i})\big) (\digamma(a_i))
\ = \ - \sum_{i=1}^g [a_i,b_i].
\end{eqnarray*}
\end{itemize}

We now prove the injectivity of $\digamma$. 
Let $P: \widehat T(H^\Q) \to \widehat T(H^\Q)$ be the endomorphism
of complete $\Q$-algebras defined by 
$$
P(a_i) = 
\Big(\frac{\exp(-\mathrm{ad}_{b_i})-\id}{-\mathrm{ad}_{b_i}}\Big)(a_i), \qquad
P(b_i)=b_i
$$
for all $i\in \{1,\dots,g\}$. Since $P$ induces the identity at the graded level,
it is an isomorphism. Thus, the injectivity of $\digamma$
is equivalent to the injectivity of the map $\digamma' := P \circ \digamma$, 
which is given by $\digamma'(a_i)=a_i$ and $\digamma'(x_i)=\exp(b_i)$.

Let $I=I_F$ denote the augmentation ideal of $\Q[F]$. Consider the filtration $(V_n)_{n\geq 0}$ 
on $T(\AAlpha^\Q) \otimes_\Q \Q[F]$ induced by the degree-filtration on  ${T}(\AAlpha^\Q)$
and the $I$-adic filtration on $\Q[F]$:
$$
V_n = \sum_{i+j=n} T^{\geq i}(\AAlpha^\Q) \otimes_\Q I^j, \ n\geq 0.
$$
Let $T(\AAlpha^\Q) \widehat{\hat{\otimes}}_\Q \Q[F]$ denote the completion of 
 $T(\AAlpha^\Q) \otimes_\Q \Q[F]$ with respect to $V$.
 The degree-filtration of  $T(\AAlpha^\Q) {\otimes}_\Q \Q[F]$ is contained in $V$,
i.e$.$ we have 
$$
T^{\geq n}(\AAlpha^\Q) {\otimes}_\Q \Q[F]\subset V_n, \ n\geq 0.
$$
Therefore, the identity induces a $\Q$-linear map
$\rho: T(\AAlpha^\Q) {\hat{\otimes}}_\Q \Q[F] \to T(\AAlpha^\Q) \widehat{\hat{\otimes}}_\Q \Q[F]$.
Let $\Upsilon:\widehat{T}(H^\Q) \to T(\AAlpha^\Q) \widehat{\hat{\otimes}}_\Q \Q[F]$
 be the unique homomorphism of filtered algebras 
 such that $\Upsilon(a_i) = a_i$ and 
 $\Upsilon(b_i)=  \log(x_i) =\sum_{k\geq 1} (-1)^{k+1}(x_i-1)^k/k$. 
 Clearly, we have $\Upsilon\circ \digamma'= \rho$. So, it is enough to prove the injectivity of $\rho$.
 
 Let $w\in T(\AAlpha^\Q) {\hat{\otimes}}_\Q \Q[F]$ with $\rho(w)=0$. We write $w$ as 
 a (possibly infinite) sum $w =\sum_{m\geq 0} w_m$
with  $w_m \in T^m(\mathbb{A}^\Q) \otimes_\Q \Q[F]$. 
 Fix an integer $k\geq 0$.
 By assumption on $w$, we have 
 $ \sum_{m= 0}^k w_m \in V_{k+1}$.
 For every $n\geq 0$, the space $V_n$ can also be written as  the direct sum
 $$
V_n =  \Big( T^{>n}(\AAlpha^\Q) \otimes_\Q \Q[F] \Big) \oplus 
 \Big( \bigoplus_{i+j=n} T^{i}(\AAlpha^\Q) \otimes_\Q I^j\Big).
 $$
 Hence, for every $m\in \{0,\dots,k\}$, we obtain that 
 $w_m$ belongs to the subspace  $T^m(\mathbb{A}^\Q) \otimes_\Q I^{k+1-m}$ of $T^m(\mathbb{A}^\Q) \otimes_\Q \Q[F]$.
 Now, fixing $m\geq 0$, we get that
 $$
w_m \in \bigcap_{k>m} \Big(T^m(\mathbb{A}^\Q) \otimes_\Q I^{k+1-m}\Big)
  = T^m(\mathbb{A}^\Q) \otimes_\Q \bigcap_{k>m}  I^{k+1-m} \ = \ \{0\}
 $$
and we conclude that $w=0$.
\end{proof}

To prove \eqref{eq:analogue_KK}, we shall also need 
the diagrammatic description
of the conjugation action of the automorphism group 
on the Lie algebra of special derivations.
Recall from \S \ref{subsec:formal_eN-series} that 
$\operatorname{IAut}\big({T}(\AAlpha^\Q) \hat\otimes_\Q \Q[F]\big)$
is the automorphism group of complete Hopf algebra 
inducing the identity on the associated graded.
Recall also that 
${\operatorname{Der}}_+\big({T}(\AAlpha^\Q) \hat\otimes_\Q \Q[F]\big)$
is the Lie algebra of derivations mapping any $x\in F$
to  $\widehat \Lie(\AAlpha^\Q)\, x$ 
and mapping $ \AAlpha^\Q $ to $\widehat \Lie_{\geq 2}(\AAlpha^\Q)$.
Let also  $\operatorname{IAut}^\zeta({T}(\AAlpha^\Q) \hat\otimes_\Q \Q[F])$
 be the subgroup of automorphisms fixing $[\zeta]\in \AAlpha$
and, similarly,  let
${\operatorname{Der}}_+^\zeta({T}\big(\AAlpha^\Q) \hat\otimes_\Q \Q[F]\big)$ be the Lie subalgebra vanishing on $[\zeta]\in \AAlpha$.
We can sum up Lemma \ref{lem:canonical_isos}, isomorphism \eqref{eq:DD},
Proposition \ref{prop:iso_D0}, Proposition \ref{prop:iso_D1}
and isomorphism \eqref{eq:eta_bis} as follows:
\begin{equation} \label{eq:many_isomorphisms}
\xymatrix{
\operatorname{IAut}^\zeta\big({T}\big(\AAlpha^\Q) \hat\otimes_\Q \Q[F]\big)
\ar[r]_-{\simeq}^-\log & 
{\operatorname{Der}}_+^\zeta({T}\big(\AAlpha^\Q) \hat\otimes_\Q \Q[F]\big)
\ar[r]_-\simeq^-{\operatorname{res}} & 
\widehat{\Der}_+^\zeta\big(\overline{\Alpha}^\Q_\bullet \big)
\ar[d]_-\simeq^-{\operatorname{res}} \ar[ld]^-\simeq_-{\operatorname{res}}\\
\widehat{\mathcal{D}} \ar[r]_-\eta^-\simeq &  
\widehat D^0_+\otimes \Q & \widehat D^1_+\otimes \Q
}
\end{equation}
Besides, by transforming beads and leaves in the obvious way, 
$\operatorname{IAut}\big({T}(\AAlpha^\Q) \hat\otimes_\Q \Q[F]\big)$
acts on the degree-completion of the quotient space $\mathcal{D}'/\mathcal{E} \simeq \modD$
introduced in Lemma~\ref{lem:T_tilde}. So
we get a  canonical action of  
$\operatorname{IAut}\big({T}(\AAlpha^\Q) \hat\otimes_\Q \Q[F]\big)$
on $\widehat{\mathcal{D}}$.

\begin{lemma} \label{lem:actions}
The  canonical action of  
$\operatorname{IAut}^\zeta\big({T}(\AAlpha^\Q) \hat\otimes_\Q \Q[F]\big)$
on $\widehat{\mathcal{D}}$ corresponds to the conjugation action of
$\operatorname{IAut}^\zeta\big({T}(\AAlpha^\Q) \hat\otimes_\Q \Q[F]\big)$
on ${\operatorname{Der}}_+^\zeta({T}\big(\AAlpha^\Q) \hat\otimes_\Q \Q[F]\big)$
through the isomorphism $\widehat{\mathcal{D}} 
\simeq {\operatorname{Der}}_+^\zeta({T}\big(\AAlpha^\Q) \hat\otimes_\Q \Q[F]\big)$ 
described by \eqref{eq:many_isomorphisms}.
\end{lemma}
\begin{proof}
The (left)  conjugation action of the automorphism group
of the algebra ${T}(\AAlpha^\Q) \hat\otimes_\Q \Q[F]$
on the Lie algebra of its derivations restricts to an action of the subgroup
$\operatorname{IAut}^\zeta\big({T}(\AAlpha^\Q) \hat\otimes_\Q \Q[F]\big)$
on the Lie subalgebra 
${\operatorname{Der}}_+^\zeta({T}\big(\AAlpha^\Q) \hat\otimes_\Q \Q[F]\big)$.
Indeed, let $\psi \in \operatorname{IAut}^\zeta\big({T}(\AAlpha^\Q) \hat\otimes_\Q \Q[F]\big)$
and let $d \in {\operatorname{Der}}_+^\zeta({T}\big(\AAlpha^\Q) \hat\otimes_\Q \Q[F]\big)$.
Setting $\delta:=\log(\psi)$, we have
$
\psi d \psi^{-1} = \exp(\ad_\delta)(d).
$
Since $\delta$ belongs to
${\operatorname{Der}}_+^\zeta({T}\big(\AAlpha^\Q) \hat\otimes_\Q \Q[F]\big)$, 
so does $\ad_\delta^n(d)$ for any $n\geq 0$.
Therefore, $\psi d \psi^{-1}$ belongs to 
${\operatorname{Der}}_+^\zeta({T}\big(\AAlpha^\Q) \hat\otimes_\Q \Q[F]\big)$.

Let  $t,\tau \in \widehat{\mathcal{D}} $   correspond
to $d,\delta \in  {\operatorname{Der}}_+^\zeta({T}\big(\AAlpha^\Q) \hat\otimes_\Q \Q[F]\big)$, respectively,
in \eqref{eq:many_isomorphisms}.
Let  $t':=\psi \cdot t$ be the result of the action of $\psi$ on $t$,
and let $d'\in {\operatorname{Der}}_+^\zeta({T}\big(\AAlpha^\Q) \hat\otimes_\Q \Q[F]\big)$
correspond to $t' \in \widehat{\mathcal{D}}$ in \eqref{eq:many_isomorphisms}.
We claim that
\begin{equation} \label{eq:claim_t}
t' = \exp(\ad_\tau)(t),
\end{equation}
which implies that $d' = \exp(\ad_\delta)(d)=\psi d \psi^{-1}$ and proves the lemma.
By  multilinearity and the bead-out relation, we can assume that $t$ consists of a single tree without bead.
Then, assume that $t$ has $r$ leaves and number them in an arbitrary way:
let $\ell_1,\dots,\ell_r \in \AAlpha^\Q$ be the  colors of these leaves.
Given $\ell'_1,\dots,\ell'_r \in \widehat{\Lie}(\AAlpha^\Q)$,
let $t(\ell'_1,\dots,\ell'_r)\in  \widehat{\mathcal{D}}$ denote the series of trees
obtained from $t$ by changing each color $\ell_i$ to~$\ell'_i$
and expanding the leaf (see Lemma \ref{lem:T_tilde}).
Then, the claim \eqref{eq:claim_t} can be reformulated as
$$
t\big(\exp(\delta)(\ell_1),\dots,\exp(\delta)(\ell_r)\big) = \exp(\ad_\tau)(t),
$$
or, more explicitly, as 
\begin{equation} \label{eq:t_tau}
\sum_{k_1,\dots, k_r\geq 0} \frac{1}{k_1!\cdots k_r!}
t\big(\delta^{k_1}(\ell_1),\dots,\delta^{k_r}(\ell_r)\big)
= \sum_{n\geq 0} \frac{1}{n!} [\underbrace{\tau,[\tau,\dots [\tau}_{\hbox{\scriptsize $n$ times}},t] \dots ]] .
\end{equation}
Next, we observe the following fact by comparing 
the formulas \eqref{eq:d(a)} and \eqref{eq:tree_bracket}.
Let~$E'$ be a tree 
with leaves  colored by  $\widehat{\Lie}(\AAlpha^\Q)$ and  without bead,
and let $E\in \widehat{\mathcal{D}}$
be obtained by simultaneous expansions  of $E'$ at all leaves.
Then $[\tau,E] \in \widehat{\mathcal{D}}$ 
is the sum of all ways of choosing a leaf of $E'$, 
applying $\delta$ to the color of that leaf, and then expanding at all leaves.
It follows from the previous observation that
$$
 [\underbrace{\tau,[\tau,\dots [\tau}_{\hbox{\scriptsize $n$ times}},t] \dots ]] 
 = \sum_{k_1,\dots, k_r\geq 0} \binom{n}{k_1,\dots,k_r} \,
 t\big(\delta^{k_1}(\ell_1),\dots,\delta^{k_r}(\ell_r)\big)
$$
and \eqref{eq:t_tau} immediately follows.
\end{proof}

We can now prove formula \eqref{eq:analogue_KK}.

\begin{proof}[Proof of Theorem \ref{th:KK_analogue}]
Let $\theta_1$ and $\theta_2$ 
be two special expansions of the free pair $(\pi,\Alpha)$.
Then, there is a (unique)  automorphism 
$\psi$ of the complete Hopf algebra $ T(\AAlpha^\Q) \hat \otimes_\Q \Q[F]$
inducing the identity at the graded level and 
such that ${\psi \circ \theta_1 =\theta_2}$. 
On the one hand, using the actions of 
$\operatorname{IAut}^\zeta\big({T}(\AAlpha^\Q) \hat\otimes_\Q \Q[F]\big)$
that are discussed in Lemma~\ref{lem:actions}, we have
\begin{eqnarray*}
\eta\Big( \frac{1}{2} \log \theta_2(u) \textsf{---} \log \theta_2(u)\Big) 
&=& \eta\Big( \frac{1}{2} \psi\big(\log \theta_1(u)\big) 
\textsf{---}\, \psi\big( \log \theta_1(u) \big)\Big)\\
&=& 
\psi \circ \eta\Big( \frac{1}{2} \log \theta_1(u) \textsf{---} \log \theta_1(u)\Big)   \circ \psi^{-1} ,
\end{eqnarray*}
where the values of $\eta$ are viewed as elements of
$\widehat{D}^0_+ 
\simeq {\operatorname{Der}}_+^\zeta({T}\big(\AAlpha^\Q) \hat\otimes_\Q \Q[F]\big)$. 
On the other hand, the definition of the representation
$\varrho^{\theta_i}: \mathcal{T} \to 
{\operatorname{Der}}_+^\zeta({T}\big(\AAlpha^\Q) \hat\otimes_\Q \Q[F]\big)$
implies that
$$
\varrho^{\theta_2}(T_{\partial U}) = 
\psi \circ \varrho^{\theta_1}(T_{\partial U}) \circ \psi^{-1}.
$$
Therefore, \eqref{eq:analogue_KK} holds for $\theta_1$
if and only if it does for  $\theta_2$.

Thus, we can restrict ourselves  to a special expansion $\theta$ 
of $(\pi,\Alpha)$ as in Lemma~\ref{lem:special_to_symplectic}.
Let $\tilde \theta = \digamma \theta$ 
be the corresponding symplectic expansion of $\pi$.
By its definition, and when viewed as a $1$-cocycle,
$\varrho^\theta(T_{\partial U})$ maps $x_i\in F$ $(1 \leq i \leq g)$ to 
$$
U_i:= \log\left( \theta\, \widehat{\Q[{T_{\partial U}}]}\, 
\theta^{-1}\right)(x_i) \ x_i^{-1}
\in \widehat{\Lie}(\AAlpha^\Q) \subset T(\AAlpha^\Q) \hat \otimes \Q[F],
$$
where $\widehat{\Q[{T_{\partial U}}]}: \widehat{\Q[\mathsf{A}_*]} \to \widehat{\Q[\mathsf{A}_*]}$ is induced by $T_{\partial U}\in \Aut(\pi,\Alpha)$
on the $J_*^\Q(\Alpha_*)$-completion of $\Q[\pi]$,
and 
$\theta: \widehat{\Q[\mathsf{A}_*]} \to 
\hat U(\overline{\mathsf{A}}_*^\Q)=  T(\AAlpha^\Q) \hat \otimes_\Q \Q[F]$
is the continuous extension of $\theta$.
We deduce from the definition \eqref{eq:digamma} of $\digamma $ that
\begin{equation} \label{eq:U_i}
\digamma(U_i) = 
 \log\left(  \tilde{\theta}\,  \widehat{\Q[{T_{\partial U}}]}\, 
 \tilde{\theta}^{-1} \right) (\exp(b_i)) \ \exp(-b_i),
\end{equation}
where $\widehat{\Q[{T_{\partial U}}]}: \widehat{\Q[\pi]} \to \widehat{\Q[\pi]}$ 
is induced by $T_{\partial U}\in \Aut(\pi)$ 
on the $I$-adic completion of $\Q[\pi]$,
and $\tilde \theta:\widehat{\Q[\pi]} \to \hat T(H^\Q)$ 
is the continuous extension of $\tilde \theta$.

Next, the map
$D:= \log\big(  \tilde{\theta}\,  \widehat{\Q[{T_{\partial U}}]}\, 
 \tilde{\theta}^{-1} \big)$,
which appears in  \eqref{eq:U_i},
is a derivation of $\hat T(H^\Q)$, and it can be computed as follows.
Indeed, Kawazumi \& Kuno gave in \cite{KK14} a closed formula 
for the logarithm $\log(\widehat{\Q[{T_{\gamma}}]})$ of any Dehn twist~$T_\gamma$,
as a derivation of the complete $\Q$-algebra $\widehat{\Q[\pi]}$.
Their formula can be stated diagrammatically using any symplectic expansion of $\pi$, 
such as $\tilde \theta$.
Let $\tilde \eta: \modD(H^\Q) \to \Der_+^\omega\big(\Lie(H^\Q)\big)$ denote
the isomorphism \eqref{eq:Kontsevich} between 
the space of trees on $H^\Q$ (modulo AS and IHX)
and the Lie algebra of symplectic derivations.
Then, the Kawazumi--Kuno formula writes  
$$
 \log\big(  \tilde{\theta}\,  \widehat{\Q[{T_{\gamma}}]}\, 
 \tilde{\theta}^{-1} \big) \big\vert_{\widehat{\Lie}(H^\Q)}
= \tilde \eta\Big( \frac{1}{2} \log \tilde \theta(\gamma ) 
\textsf{---} \log \tilde \theta(\gamma)\Big),     
$$
see  \cite[\S 4]{KMT21} for details. Thus,
specializing to the meridian $\gamma:=\partial U$, we obtain
\begin{equation}
\label{eq:KK_D}
D\big\vert_{\widehat{\Lie}(H^\Q)}
= \tilde \eta\Big( \frac{1}{2}  \digamma(\log \theta(\gamma ))  \textsf{---} \digamma(\log  \theta(\gamma))\Big).
\end{equation}

By the usual formula expressing the value of a derivation $D$
on a formal power series (see e.g$.$ \cite[Theorem 3.22]{Reutenauer}), we~get
\begin{equation} \label{eq:digamma(U_i)}
\digamma(U_i) \stackrel{\eqref{eq:U_i}}{=} D(\exp(b_i))\, \exp(-b_i)
= \big(f^{-1}(\mathrm{ad}_{b_i})\big) \big(D(b_i)\big),
\end{equation}
where 
$$
f(u) := \frac{u}{\exp(u)-1} \in \Q[[u]].
$$
The series of trees 
$ \frac{1}{2} \digamma(\log \theta(\gamma ))  \textsf{---} \digamma(\log  \theta(\gamma))$ (with leaves colored by $H^\Q$) is obtained 
from the series of trees 
$
S:= \frac{1}{2} \log \theta(\gamma )  \textsf{---} \log \theta(\gamma)
$
(with leaves colored by $\AAlpha^\Q$ and  beads colored by $F$)
by applying the following operations to all leaves  and beads,
respectively:\\[0.1cm]
$$
\labellist
\small\hair 2pt
 \pinlabel {$x_i$} [b] at 133 40
 \pinlabel {$a_i$} [r] at 50 173
 \pinlabel {$\leadsto$} at 314 175
 \pinlabel {$\leadsto$}  at 314 28
 \pinlabel {$a_i$} [r] at 432 172
 \pinlabel {\scriptsize $f\!(u)$} [ ] at 514 173
 \pinlabel {$b_i$} [b] at 514 220
 \pinlabel {\scriptsize $e^u$} [ ] at 512 27
 \pinlabel {$b_i$} [b] at 513 74
 \pinlabel {where} [ ] at 756 110
 \pinlabel {\scriptsize $u^n$} at 969 106
  \pinlabel {$:=$} at 1080 106
 \pinlabel {$b_i$} [b] at 967 159
 \pinlabel {\scriptsize  $b_i$} [b] at 1170 154
 \pinlabel {\scriptsize  $b_i$} [b] at 1194 154
  \pinlabel {$\overbrace{\hphantom{qqqqq}}$} [b] at 1195 175
  \pinlabel {\scriptsize $n$ times} [b] at 1195 205
  \pinlabel {\tiny $...$} [c] at 1205 130
 \pinlabel {\scriptsize  $b_i$} [b] at 1225 154
\endlabellist
\centering
\includegraphics[scale=0.29]{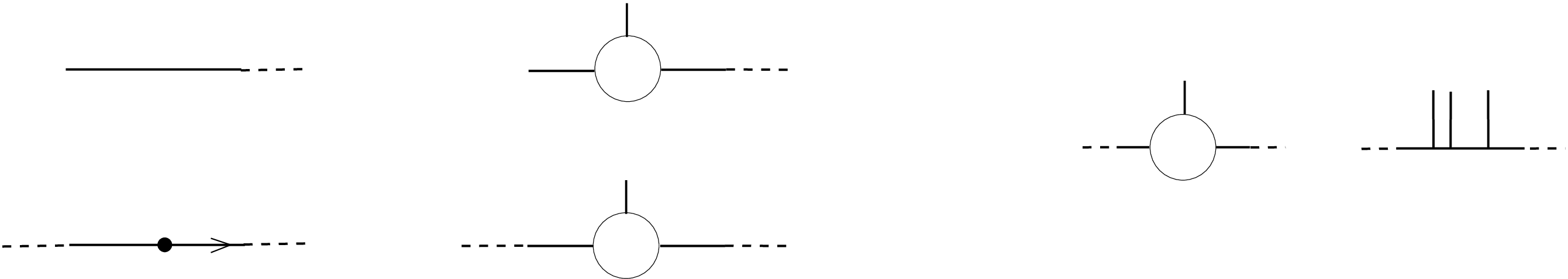}
$$
(Here, thanks to the bead-out relation, we assume that all the leaves 
of $S$ are colored by the $\Q[F]$-basis $(a_1,\dots,a_g)$ of $\AAlpha^\Q$.)
Thus, we can compute $D(b_i)$ from \eqref{eq:KK_D} and we get the sum
$$
\sum_{\ell} \big(f(\mathrm{ad}_{b_i})\big) \big(\digamma(\mathrm{word}(\ell))\big)
$$
over all the leaves $\ell$ of  $S$  that are
colored by $a_i$. Therefore,
applying \eqref{eq:digamma(U_i)}, we obtain that 
$$
\digamma(U_i)= 
\digamma\Big(\sum_{\ell} \mathrm{word}(\ell)\Big)
= -\digamma\big(\eta(S)(x_i)\big).
$$
Since $\digamma$ is injective, we conclude that 
$
\varrho^\theta(T_{\partial U})(x_i) = 
U_i = -\eta(S)(x_i).
$
Therefore, the 1-cocycles underlying the special derivations
$\varrho^\theta(T_{\partial U})$
and $-\eta(S)$ are equal.
\end{proof}

\begin{remark}
The presence of a minus sign in \eqref{eq:analogue_KK}
in contrast with the absence of sign in \eqref{eq:KK_D}
is explained as follows.
The identification \eqref{eq:two_isomorphisms}
between $\mathbb{A}^\Q$ and the dual of the
augmentation ideal of $\Q[F]$ (which is involved in the definition of $\eta$)
 is given by $a\mapsto \langle -,a\rangle$
using the homotopy intersection form \eqref{eq:<-,->} of the handlebody.
On the contrary,
the identification of $H^\Q=H_1(\Sigma;\Q)$ with its dual
(which is involved in the definition of $\tilde \eta$)
is given by $h\mapsto \omega(h,-)=-\omega(-,h)$
using the homology intersection form $\omega$ of the surface.
\end{remark}

\subsection{Example: the pure braid group}

We consider here  Oda's embeddings \cite{Oda}
of the $g$-strand pure braid group $PB_g$  
into the mapping class group $\modM=\modM(\Sigma,\partial \Sigma)$
and, to fit our purposes,
we assume here that the image of the embedding 
is contained in the twist group~$\modT=\modT(V)$.
Embeddings of the (framed) pure braid groups into the twist groups, in the context of Johnson filtrations, were also considered in \cite{HV20}.

To be more specific, we decompose $\partial H $ as 
$C \cup_\partial C’$,
where $C,C'$ are surfaces of genus $0$ such that the disk
$D=\partial H \setminus \hbox{int}(\Sigma)$ is contained in $C’$ and,
for a system of curves $(\alpha,\beta)$ such as \eqref{eq:(a,b)},
$\partial C$ consists of the curves $\alpha_1, \dots, \alpha_g$ 
and an ``outer'' boundary curve $\upsilon$: 
$$
\labellist
\small\hair 2pt
 \pinlabel {$\star$} at 832 3
 \pinlabel {$\alpha_1$} [tr] at 294 72
 \pinlabel {$\alpha_2$} [tr] at 506 73
 \pinlabel {$\alpha_g$} [tr] at 777 75
 \pinlabel {$C$}  [l] at 790 45
 \pinlabel {$C'\setminus \operatorname{int}(D)$}  at 860 166
 \pinlabel {$\upsilon$} [tr] at 229 55
\endlabellist
\centering
\includegraphics[scale=0.36]{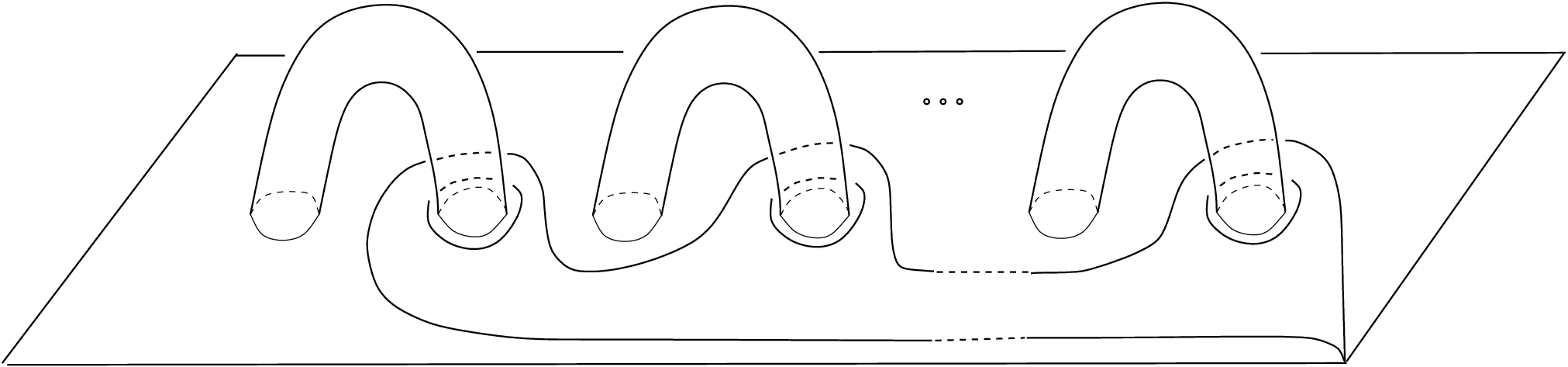}
$$

The inclusion of $C$ into $\Sigma$ induces a homomorphism
$\modM(C,\partial C)\to \modM$, 
which is injective.
Furthermore, $\modM(C,\partial C)$ can be identified with the 
$g$-strand framed pure braid group,
so that $PB_g$ embeds canonically in $\modM(C,\partial C)$
as the group of $0$-framed pure braids.
Hence, we can view $PB_g$ as a subgroup of $\modM$,
and it is easily seen that $PB_g$ is actually contained in
the twist group $\modT=\modT(V)$.

The next result relates the lower central series of the pure braid group
to the lower central series of the twist group,
and to the Johnson filtration of the handlebody group.
This is an analogue of \cite[Theorem 1.1]{GH02},
which deals with the Torelli group and the usual Johnson filtration
of the mapping class group.

\begin{theorem} \label{th:Milnor_to_Johnson}
For all $k\geq 1$, we have
$
\Gamma_k PB_g = PB_g \cap \Gamma_k \modT = PB_g \cap \modH_k
$
and the following diagram is commutative:
\begin{equation} \label{eq:square}
\xymatrix{
\Gamma_k PB_g \ar@{^{(}->}[r]  \ar[d]_-{(-1)^k\,\mu_k} 
& \modH_k \ar[d]^-{\tau_k^d} \\
\mathcal{D}_k(A^\Q) \ar@{>->}[r] 
& \mathcal{D}_k(\AAlpha^\Q,\Q[F]).
}
\end{equation}
\end{theorem}

\noindent
In the above diagram, $\mu_k$ is the \emph{$k$-th Milnor homomorphism}
which encompasses all Milnor invariants of length $k+1$.
Denoting by $A^\Q$ the $\Q$-vector space with basis 
$\{a_1,\dots,a_g\}$, the invariant  $\mu_k$ takes values in the space
$$
\mathcal{D}_k(A^\Q) =
\frac{\Q \cdot \left\{
\hbox{node-oriented trees on $A^\Q$}\right\}}{\hbox{AS, IHX, multilinearity}}.
$$
We now review the definition of this homomorphism.
It involves the canonical action of $PB_g$
on the fundamental group $A:=\pi_1(C,\star)$,
which is the free group generated by $\alpha_1,\dots,\alpha_g$. 

Let $t\in PB_g$.
Let $\ell_1(t),\dots, \ell_g(t) \in A$ denote
the \emph{longitudes} of $t$, which are uniquely defined by the conditions that
 $t(\alpha_i) = \ell_i(t) \cdot \alpha_i \cdot \ell_i(t)^{-1}\in A$,
and $[\ell_i(t)]\in H_1(A;\Z)$ is a $\Z$-linear combination
of the $a_j=[\alpha_j]$ for $j\neq i$.
Then, for any $k\geq 2$,
the following statements are well-known to be equivalent to each other
(see e.g$.$ \cite[Proof of Lemma 16.4]{HM00}):
\begin{itemize}
\item[(i)] $t$ belongs to $\Gamma_k PB_g$,
\item[(ii)] for all $x\in A$, $t(x)x^{-1}$ belongs to $\Gamma_{k+1} A$,
\item[(iii)] for all $i\in \{1,\dots,g\}$, $\ell_i(t)$ belongs to $\Gamma_{k} A$.
\end{itemize}
In the sequel, we identify  the associated graded of the lower central series of $A$
(with rational coefficients) with the Lie $\Q$-algebra generated by $A^\Q$.
Assume now that $t \in \Gamma_k PB_g$ and define
$$
\mu_k(t) = \sum_{i=1}^g a_i \otimes [\ell_i(t)]_k 
\in  A^\Q \otimes_\Q \Lie_k(A^\Q).
$$
Since $t$ preserves 
the  boundary component $\upsilon$ of $C$, it follows that $\mu_k(t)$
belongs to the kernel of the Lie bracket of $\Lie(A^\Q)$.
Hence, using the isomorphism \eqref{eq:eta_first}, 
we can view $\mu_k(t)$ as an element of $\mathcal{D}_k(A^\Q)$.
The resulting map  $\mu_k:\Gamma_k PB_g \to \mathcal{D}_k(A^\Q)$ is a homomorphism
and, by the above equivalence ``(i)$\Leftrightarrow$(iii)'',
we have $\ker \mu_k = \Gamma_{k+1} PB_g$. 
It is also known (see e.g$.$ \cite[Prop$.$ 3]{HP}) that the  map
$$
(\mu_k)_{k\geq 1}: 
\bigoplus_{k\geq 1} \frac{\Gamma_k PB_g}{\Gamma_{k+1} PB_g}\otimes \Q
\longrightarrow \mathcal{D}(A^\Q) 
$$
is a homomorphism of graded Lie $\Q$-algebras, 
if $\mathcal{D}(A^\Q)$ is endowed with the Lie bracket defined, 
for any trees~$D$ and $E$, by the formula
\begin{equation} \label{eq:grafting_bis}
[D,E] = \sum_{v,w}  \ \delta_{\col(v),\col(w)} {\labellist
\small\hair 2pt
 \pinlabel {$D$} [ ] at 23 86
 \pinlabel {$E$} [ ] at 122 86
 \pinlabel {$\col(v)$} [t] at 73 3
\endlabellist
\centering
\includegraphics[scale=0.22]{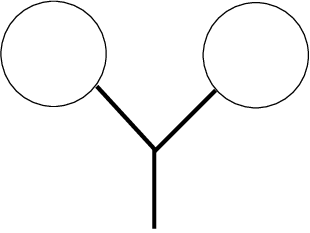}}.
\end{equation}
Here the sum is over all leaves $v$ and $w$  of $D$ and $E$, respectively,
and the corresponding ``branched'' tree is obtained by gluing $D$ and $E$ along the half-edges incident to $v$ and $w$
whenever $\col(v)=\col(w)$, the new leaf being then colored by~$\col(v)$.

\begin{proof}[Proof of Theorem \ref{th:Milnor_to_Johnson}]
Let us first prove the injectivity of 
$\mathcal{D}(A^\Q)  \to \mathcal{D}(\AAlpha^\Q,\Q[F])$.
For this, we identify $\mathcal{D}(\AAlpha^\Q,\Q[F])$
with the space $\mathcal{D}(\AAlpha^\Q)_{\Q[F]}$ 
defined at~\eqref{eq:with_coinvariants}.
Set
$
S:= \{a_1,\dots, a_g\}
$.
Consider the cartesian product $F\times S$,
which, to suggest the canonical left action of $F$ on $F\times S$,
we prefer to refer to as
$$
F\cdot S:= \big\{f \cdot a_1,\dots, f\cdot a_g\, \big\vert\, f\in F \big\}.
$$
Then,
getting rid of the ``multilinearity'' relations, 
we can identify $\mathcal{D}(A^\Q)$ with
$$
\mathcal{D}(S) :=
\frac{\Q \cdot \left\{
\hbox{node-oriented trees on $S$}\right\}}{\hbox{AS, IHX}},
$$
and identify $\mathcal{D}(\AAlpha^\Q)_{\Q[F]}$ with
$$
\mathcal{D}(F\cdot S)_F :=
\frac{\Q \cdot \left\{
\hbox{node-oriented trees on $F\cdot S$}\right\}}{\hbox{AS, IHX, translation}},
$$
where, for any $f\in F$, 
the ``translation'' relation identifies any tree on $F\cdot S$
with the same tree where $f$ as acted on each leaf.
(This last identification is allowed since $\AAlpha^\Q$
is free on $a_1,\dots,a_g$ as a $\Q[F]$-module.)
We now fix a degree~$k\geq 1$.
The diagonal action of $F$ on $(F\cdot S)^{k+1}$ commutes
with the canonical action of the symmetric group $\mathfrak{S}_{k+1}$.
Let $Q$ be the corresponding double quotient set. 
An element of $Q$ can be assigned to any tree of degree $k$ on $F\cdot S$. Thus, $\mathcal{D}(F\cdot S)_F$ splits as a direct sum over $Q$ and, clearly,
$\mathcal{D}(S)$ is there the sum of the direct summands corresponding 
to the double cosets that have a (unique) representative in $S^{k+1}$.
This proves the injectivity of the map 
$\mathcal{D}_k(A^\Q)  \to \mathcal{D}_k(\AAlpha^\Q,\Q[F])$ for every $k\geq 1$.

Since $PB_g\subset\modT$, we have 
$\Gamma_k PB_g \subset PB_g \cap \Gamma_k \modT \subset PB_g \cap \modH_k$. Hence, we only have to prove the inclusion
$PB_g \cap \modH_k \subset \Gamma_k PB_g$.
But, since we have $\Gamma_{k+1} PB_g =\ker \mu_k$ 
and $\modH_{k+1}=\ker \tau_k$, the inclusion 
$PB_g \cap \modH_{k+1} \subset \Gamma_{k+1} PB_g$ follows from 
\eqref{eq:square} by induction on $k$.

Therefore, it only remains to prove the commutativity of \eqref{eq:square}.
Recall that $PB_g$ is generated by the ``elementary'' pure braids $t_{ij}$
(for $i\neq j$) that ``clasp'' the $i$-th and the $j$-th strands in front of
the other strands. Specifically, we have 
$$
t_{ij} = T_{\gamma_{ij}} T_{\alpha_i}^{-1} T_{\alpha_j}^{-1} ,
$$
where $\gamma_{ij}$ is a simple closed curve 
``encircling'' $\alpha_i$ and $\alpha_j$.
We have $\mu_1(t_{ij})= a_i \textsf{---} a_j$ 
and, according to Proposition \ref{prop:tau_1(disk_twist)}, we have 
$$
\tau_1^d(t_{ij})= -\frac{1}{2}   (a_i+a_j) \textsf{---} (a_i+a_j)
+ \frac{1}{2}  a_i \textsf{---} a_i +  \frac{1}{2}  a_j \textsf{---} a_j 
= - a_i \textsf{---} a_j.
$$
Therefore, the diagram \eqref{eq:square} is commutative for $k=1$.
Besides, by comparing formulas \eqref{eq:grafting_bis} and
\eqref{eq:tree_bracket},
we observe that
the map $\mathcal{D}_k(A^\Q)  \to \mathcal{D}_k(\AAlpha^\Q,\Q[F])$
preserves the Lie brackets.
The conclusion follows since,
for arbitrary $k\geq 1$, the abelian group $\Gamma_k PB_g/ \Gamma_{k+1} PB_g$
is generated by iterated commutators of length $k$ in the $t_{ij}$'s.
\end{proof}

\section{Connections and prospects}

In this concluding section, we mention some relations between our approach and some other approaches to the handlebody group,
and also briefly describe some new perspectives
that the present paper opens up.

\subsection{Restriction of the usual Johnson filtration to the handlebody group}

Another natural filtration of the handlebody group $\modH$ is the restriction $(\modM_k\cap\H)_{k\ge0}$ of the usual Johnson filtration $(\modM_k)_{k\geq 1}$ 
of the mapping class group~$\modM$,
which was reviewed in \S \ref{subsec:usual_Johnson}.
For instance, 
the first term $\modH \cap \modM_1$ is the intersection
of  the handlebody group with the Torelli group,
see  \cite{Pitsch, Omori} for a generating system.

This approach, which is evoked in \cite[\S 7]{Hensel}
and  considered e.g$.$ in \cite{Faes}, is quite different from ours.
Nevertheless, 
our Johnson filtration 
of $\modH$ is contained in $(\T\cap\modM_k)_{k\ge1}$,
and  our Johnson homomorphisms  determine the usual ones.
Specifically, we have
$$
\modH_{k} \ \subset \ \modM_{k-1} \cap \modH, \quad \hbox{for all $k\geq 1$}
$$
and, for $f \in \modH_k$ with $k\geq 2$, 
the usual Johnson homomorphism $\tau_{k-1}^{\operatorname{usual}}(f)\in \modD_{k-1}(H^\Q)$
is obtained from $\tau_{k}(f)\in \modD_{k}(\AAlpha^\Q,\Q[F])$
by ignoring beads and transforming the colors of leaves 
through the homomorphism $\AAlpha \to H$ 
(induced by the inclusion of $\Alpha$ in $\pi$).
More generally, it would be interesting to relate our Johnson filtration of the handlebody group to the “double Johnson filtration” that was introduced in \cite{HV20}.

\subsection{Relative weight filtration for the handlebody group}

In \cite{Hain_handlebody}, Hain considers the ``relative weight filtration'' on (the relative completion of) the mapping class group with respect to 
a finite set of pairwise-disjoint simple closed curves on the surface. 
When those curves are given by a pants decomposition of the surface,  this filtration depends only on the handlebody underlying the pants decomposition.
Furthermore, the $0$-th and $1$-st terms of this filtration
are  the corresponding handlebody group and twist group, respectively. 

It would be interesting to compare Hain's relative weight filtration
of the handlebody group  to our Johnson filtration.
Indeed, combining the strictness and exactness of the former 
with the explicit algebraic descriptions of the latter 
may be  useful in the study of handlebody groups.

\subsection{Abelianization of the twist group}

In contrast with the Torelli group \cite{Johnson_abelianization},
the structure of the abelianization of $\modT$ is not well understood.
The first step in the understanding of this structure would be to determine  the image of the first Johnson homomorphism
$\tau_1^0: \modT \to D^0_1$
or, equivalently, the image of the Magnus representation 
$\Mag: \modT \to \hbox{Mat}(g\times g;\Z[F])$.

The second step would be to decide 
whether the abelianization of $\modT$ is torsion-free.
In fact, computing the rational abelianization of
$\modT$ is already a challenge,
and it is necessary 
for the computation of its Malcev Lie algebra. 

\subsection{Images of the Johnson homomorphisms and trace maps}

A more general problem for a further study of the filtration 
$(\modH_k)_{k\geq 1}$ would be to determine the images of the Johnson homomorphisms in any degree $k\geq 1$.
In the case of the usual Johnson filtration $(\modM_k)_{k\geq 1}$,
reviewed in \S \ref{subsec:usual_Johnson},
this problem has not been solved yet, 
but there exist    ``divergence'' $1$-cocycles
on the Lie algebra $\Der_+^\omega(\Lie(H))$ which are known to vanish 
on $\overline{\tau}_+(\overline{\modM}_+)$. Such $1$-cocycles include 
the  Morita trace \cite{Morita_abelian}
and the Enomoto--Satoh trace \cite{ES}.
It is important to construct analogues of those $1$-cocycles
for the Johnson filtration of  the handlebody group.

\subsection{Tree-level of the extended Kontsevich integral}
\label{treelevel}

The handlebody groups are
the groups of automorphisms 
in the category of ``bottom tangles in handlebodies'' \cite{Habiro}. 
The Kontsevich integral (originally defined for tangles in balls)
was extended in \cite{HM21} to a functor~$Z$ from this category 
to the category of ``Jacobi diagrams in handlebodies''.
Thus, we obtain from $Z$ new diagrammatic representations
of the handlebody group $\modH$ in any genus $g\geq 1$.
These might be useful for the
problem of determining the associated graded 
$\overline{\modT}_+$ of the lower central series of $\modT$
and, more difficultly, the Malcev Lie algebra of $\modT$.

In a future work, it will be shown that the tree-level of $Z\vert_\modT$
(which consists in ignoring non-acyclic Jacobi diagrams)
is equivalent to the ``infinitesimal'' action $\varrho^\theta$
of $\modT$ on the free pair $(\pi,\Alpha)$.
Here $\theta$ is a certain  special expansion of $(\pi,\Alpha)$,
which is itself defined from the extended Kontsevich integral.
It follows that, for every~$k\geq 1$,
the degree~$k$ part of the tree-level of 
$Z\vert_{\modH_k}$ is equivalent to~$\tau_k$.

\appendix
\section{Intersection operations in a handlebody}
\label{sec:intersection}

In this appendix, we describe several intersection operations  in a handlebody.

\subsection{The homotopy intersection form of a surface}
\label{subsec:hif_surface}

We start by reviewing Turaev's homotopy intersection form 
of a  surface \cite{Turaev}. This form determines 
the homology intersection forms of $\Sigma$ with arbitrary twisted coefficients,
and it is also implicit in Papakyariakopoulos' work \cite{Papakya}.

Let $\Sigma$ be a compact connected oriented surface with one boundary component.
Its fundamental group $\pi=\pi_1(\Sigma,\star)$ 
is based at a point $\star \in \partial \Sigma$. Let
$$
\eta: \Z[\pi] \times \Z[\pi] \longrightarrow \Z[\pi]
$$
be the $\Z$-bilinear pairing that maps any pair $(x,y)\in \pi\times \pi$ to
\begin{equation}\label{eq:eta}
\eta({x}, {y}) =
 \sum_{p\in X \cap Y} \varepsilon_p(X,Y)\
 \left[(\overline{\partial \Sigma})_{\star\bullet }
 X_{\bullet p} Y_{p \star}\right].
\end{equation}
Here $\bullet \ (\neq\star)$ is a second base-point in $\partial \Sigma$,
 $X$ is a loop  based at $\bullet$ 
 such that  $(\overline{\partial \Sigma})_{\star\bullet }\, X\, (\partial \Sigma)_{\bullet \star}$ represents $x$,
$Y$ is a loop   based at $\star$  representing $y$, and we assume that
$X$ and $Y$ meet transversely in a finite set of double points;
at every such point $p\in X\cap Y$, the sign
$ \varepsilon_p(X,Y)$ is equal to $+1$ if, and only if, 
a unit tangent vector of $X$ followed by a unit tangent vector of $Y$ 
gives a direct frame of $\Sigma$;
finally, $X_{\bullet p}$ (resp$.$ $Y_{p \star}$)
denotes the arc in $X$ (resp. $Y$)
connecting $\bullet$ to $p$ (resp$.$ $p$ to $\star$).

The operation $\eta$ is a \emph{Fox pairing} in the sense 
that it is a left \emph{Fox derivative} in its first argument:
\begin{equation} \label{eq:Fox_left}
\eta(xx',y) = x\, \eta(x',y)+ \eta(x,y)\, \varepsilon (x')  
\quad \hbox{(for all $x,x',y'\in \Z[\pi]$)}
\end{equation}
and right \emph{Fox derivative} in its second argument:
\begin{equation} \label{eq:Fox_right}
 \eta(x,yy') = \eta(x,y)\, y'+ \varepsilon (y)\,  \eta(x,y')
 \quad \hbox{(for all $x,y,y'\in \Z[\pi]$)},
\end{equation}
where $\varepsilon: \Z[\pi] \to \Z$ denotes the augmentation map. 
 (See \cite[\S 7]{MT14} for details.)
Observe that  $\eta$ is ``almost skew-symmetric'':
\begin{equation}    \label{eq:ask}
\forall u,v \in \pi, \quad \eta(u,v) = - u\, \overline{\eta({v},{u})}\, v - (u-1)(v-1)
\end{equation}

The Fox pairing $\eta$  is determined by its values on a basis of $\pi$.
For instance,
the matrix of $\eta$ in a basis $(\alpha,\beta)$ 
 of  type \eqref{eq:(a,b)} is 
\begin{equation} \label{eq:E}
E=\left(\begin{array}{c|c} E_{\alpha \alpha} & E_{\alpha\beta} \\ 
\hline E_{\beta \alpha} & E_{\beta \beta} \end{array}\right)
\in \Mat(2g\times 2g; \Z[\pi]),    
\end{equation}
where
\begin{gather*}
\begin{matrix}
    (E_{\alpha\alpha})_{i,j}=
\begin{cases}
    \alpha_i-1&(i=j),\\
    P(\alpha_i,\alpha_j)&(i>j),\\
    0&(i<j),
\end{cases}
&\quad
    (E_{\alpha\beta})_{i,j}=
\begin{cases}
    \alpha_i+\beta_i-1&(i=j),\\
    P(\alpha_i,\beta_j)&(i>j),\\
    0&(i<j),
\end{cases}
\\&\\
    (E_{\beta\alpha})_{i,j}=
\begin{cases}
    -1&(i=j),\\
    P(\beta_i,\alpha_j)&(i>j),\\
    0&(i<j),
\end{cases}
&\quad
    (E_{\beta\beta})_{i,j}=
\begin{cases}
    \beta_i-1&(i=j),\\
    P(\beta_i,\beta_j)&(i>j),\\
    0&(i<j),
\end{cases}
\end{matrix}
\end{gather*}
with $P(x,y)=-(x-1)(y-1)$.
(See \cite[\S 5]{Morita_abelian} or 
\cite[Lemma 2.4]{Perron} for a similar computation.)

\subsection{The intersection double bracket of a surface}

We now review a variant of the homotopy intersection form $\eta$,
which was considered  in~\cite{MT14}.
Define a $\Z$-linear map
$$
\double{-}{-}: \Z[\pi] \otimes \Z[\pi] 
\lto  \Z[\pi] \otimes \Z[\pi], 
$$
by
\begin{equation} \label{eq:Fox_to_double}
\forall a,b \in \Z[\pi], \quad
\double{a}{b} = b'\, \overline{\big(\eta( a'', b'')\big)'}\,  a' 
\otimes \big(\eta( a'', b'')\big)'',
\end{equation}
where we use Sweedler's notation for the coproduct of $\Z[\pi]$.
The pairing $\eta$ can be recovered from $\double{-}{-}$
by the identity $\eta = (\varepsilon \otimes \id) \circ\double{-}{-}$.
Note that, with the notations of \eqref{eq:eta}, we have 
\begin{equation}\label{eq:double}
\forall x,y\in \pi, \quad
\double{x}{y} =
 \sum_{p\in X \cap Y} \varepsilon_p(X,Y)\
\left[Y_{\star p}   X_{p \bullet }   (\partial \Sigma)_{\bullet \star }\right]
\otimes \left[(\overline{\partial \Sigma})_{\star\bullet }X_{\bullet p} Y_{p \star}\right].
\end{equation}

Properties \eqref{eq:Fox_left}--\eqref{eq:Fox_right} imply that the operation $\double{-}{-}$
is a biderivation in the following sense:
\begin{eqnarray}
\label{eq:outer}\qquad \forall a,b,c\in \Z[\pi],  \quad 
\double{a}{bc}& = &b 
\double{a}{c}^\ell \otimes \double{a}{c}^r+ \double{a}{b}^\ell \otimes \double{a}{b}^rc,\\
\label{eq:inner}\qquad \forall a,b,c\in \Z[\pi],  \quad 
\double{ab}{c}& = & 
\double{b}{c}^\ell \otimes a\double{b}{c}^r+ \double{a}{c}^\ell b \otimes \double{a}{c}^r.
\end{eqnarray} 
Besides, the ``skew-symmetry'' \eqref{eq:ask} of $\eta$ implies the following:
\begin{equation}
\label{eq:ask_bis}
\forall a,b\in \Z[\pi], \quad
\double{a}{b}= - \double{b}{a}^r \otimes \double{b}{a}^\ell
- ba \otimes 1 -  1 \otimes ab + b\otimes a  + a \otimes b.
\end{equation}
Furthermore,
according to \cite[(7.2.12)]{MT14},
the operation $\double{-}{-}$ 
satisfies the following ``quasi-Jacobi'' identity 
in $\Z[\pi]^{\otimes 3}$:
\begin{eqnarray} \label{eq:quasi-Jacobi}
\quad  \forall a,b,c \in \Z[\pi], 
&& \double{a}{\double{b}{c}^\ell} \otimes \double{b}{c}^r 
-  \double{a}{c}^l \otimes \double{b}{\double{a}{c}^r}\\
\notag && -  \double{\double{a}{b}^\ell}{c}^\ell \otimes  \double{a}{b}^r
\otimes \double{\double{a}{b}^\ell}{c}^r  \\
\notag &=&     \double{a}{c}^\ell  \otimes  1 \otimes b\double{a}{c}^r  
- \double{a}{c}^\ell  \otimes b\otimes \double{a}{c}^r.
\end{eqnarray}

\begin{remark}
A slight modification of the operation $\double{-}{-}$ translates
the properties from \eqref{eq:outer} to \eqref{eq:quasi-Jacobi}
into the axioms of ``quasi-Poisson double bracket''
in the sense of Van den Bergh \cite{VdB}.
See \cite{MT14}.
\end{remark}

\subsection{The homotopy intersection form of a handlebody}
\label{subsec:hif_handlebody}

We now view $\Sigma$ as the boundary of a handlebody $V$, 
with the interior of disk $D$ removed. 
Thus we have $\partial V = \Sigma \cup D$.

Let $\varpi: \pi \to F$ be the canonical map onto $F:=\pi_1(V,\star)$.
Set  $\Alpha:=\ker \varpi$ and $\AAlpha:=\Alpha_{\operatorname{ab}}$.
Let $\AA^r$ denote $\AA$ with the right $\Z[F]$-action induced by the right conjugation of $\pi$ on $\A$.
Let $I_F$ denote the augmentation ideal of $\ZF$,
 which we regard as a (left) $\Z[F]$-module.

\begin{proposition} \label{prop:intersection}
The homotopy intersection form $\eta$ of $\Sigma$ induces 
a map $\langle -,-\rangle :  \Z[F] \times \AAlpha \to \Z[F]$,
which restricts to a non-singular $\Z[F]$-bilinear map
\begin{equation} \label{eq:<-,->}
   \langle -,-\rangle :  I_F \times \AAlpha^r \longrightarrow \Z[F].
\end{equation}
The latter corresponds (via canonical isomorphisms) to 
the homology intersection pairing
\begin{equation} \label{eq:intersection_form}
H_1(V,D;\Z[F])  \times 
H_2(V, \Sigma;\Z[F]) \longrightarrow  \Z[F]
\end{equation}
of $V$
with coefficients twisted by the canonical homomorphism 
$F=\pi_1(V,\star) \to \Z[F]$.
\end{proposition}

\begin{proof}
We denote by the same letter $\varpi:\Z[\pi] \to \Z[F]$
the ring homomorphism induced by $\varpi: \pi \to F$.
It follows from \eqref{eq:Fox_right} that
the restriction of $\varpi\circ  \eta$ to $\Z[\pi] \times \Alpha$ 
is additive in its second argument,
so that it induces a $\Z$-bilinear map 
$\tilde\eta: \Z[\pi] \times \AAlpha \to \Z[F]$.
Another application of \eqref{eq:Fox_right}  shows that
\begin{equation} \label{eq:r_conj}
\forall x\in \pi, \ \forall {a} \in \Alpha, \quad
\varpi\eta \big(-,  x^{-1} ax\big) =  \varpi\eta(-,a)\cdot \varpi(x),
\end{equation}
hence $\tilde\eta: \Z[\pi] \times \AAlpha^r \to \Z[F]$ is $\Z[F]$-linear in its second argument.

Let $J$ be the kernel of $\varpi:\Z[\pi] \to \Z[F]$.
To see that $\tilde\eta$
factors through a  map $\langle -,-\rangle :  \Z[F] \times \AAlpha \to \Z[F]$,
it suffices to show that 
\begin{equation} \label{eq:Ia_A} 
\tilde\eta(J, \AAlpha ) =0.
\end{equation}
Moreover, since we have $J=\Z[\pi]\, I_{\Alpha}$ where $I_\A=\ker(\epsilon:\Z[\A]\to \Z)$, it is enough to prove that
$\tilde\eta(I_{\Alpha}, \AAlpha ) =0$.
The group~$\sf{A}$ being normally generated by $\alpha_1, \dots,\alpha_g$ in~$\pi$,
the subset $I_{\sf{A}}$ is $\Z$-spanned by $x(\alpha_i-1)x^{-1}$ 
for all $x\in \pi$ and $i=1,\dots,g$.
But, similarly to \eqref{eq:r_conj}, we have 
\begin{equation} \label{eq:l_conj}
\forall x\in \pi, \ \forall {a} \in \Alpha, \quad
\varpi \eta (xax^{-1}, - ) = \varpi(x) \cdot \varpi\eta(a,-).
\end{equation}
Thus, \eqref{eq:Ia_A} follows since   
$ \varpi  \eta(\alpha_i, \alpha_j)=0$
for any  $i,j \in \{1,\dots,g\}$.

The restricted  map $\langle -,-\rangle : I_F \times \AAlpha^r \to \Z[F]$
is also $\Z[F]$-linear in its first argument because of \eqref{eq:Fox_left}. 
To justify its non-singularity, it suffices to compute it
in the basis $(x_i-1)_i$ of $I_F$ and in the basis $([\alpha_j])_j$ of $\AAlpha$:
\begin{equation}\label{eq:in_bases}
\langle x_i-1, [\alpha_j] \rangle = \varpi \eta(\beta_i,\alpha_j) =- \delta_{ij}.
\end{equation}
Thus the matrix of $\langle -,-\rangle$ 
in the above bases is  $-I_g$, which is invertible.

We prove the second statement,
which gives a 3-dimensional interpretation to the pairing $\langle -,-\rangle$.
The connecting map in the long exact sequence of the pair $(V,D)$ 
gives an injection
$
\partial_*: H_1(V,D;\Z[F]) \to H_0(D;\Z[F]) \simeq \Z[F]
$
with image $I_F$.
Hence we get a canonical isomorphism 
\begin{equation} \label{eq:iso_1}
H_1(V,D;\Z[F])\simeq I_F.
\end{equation}
Besides, the connecting map
of the triple $\big(V, \Sigma,\star\big)$ gives an injective homomorphism
$
\partial_*: H_2\big(V, \Sigma;\Z[F]\big)
\to H_1\big( \Sigma,\star;\Z[F]\big)
$
whose image coincides with
$$
\ker\Big(\hbox{incl}_*: 
\underbrace{H_1\big( \Sigma,\star;\Z[F]\big)}_{\simeq\, I_{\pi} \otimes_{\Z[\pi]} \Z[F]  } 
\longrightarrow \underbrace{H_1\big( V,\star;\Z[F]\big)}_{\simeq\, I_F}  \Big).
$$
It is easily checked that the map $I_{\sf{A}} \to I_{\pi} \otimes_{\Z[\pi]}  \Z[F]  $ defined by $x\mapsto  x\otimes 1$
induces an isomorphism between 
$\AAlpha \simeq I_{\sf{A}}/ I^2_{\sf{A}}$ and this kernel.
Hence we get 
a canonical isomorphism 
\begin{equation} \label{eq:iso_2}
H_2(V, \Sigma;\Z[F])\simeq \AAlpha^r.
\end{equation}
That the pairings \eqref{eq:<-,->} and \eqref{eq:intersection_form}
correspond to each other through the isomorphisms 
\eqref{eq:iso_1} and \eqref{eq:iso_2} follows 
immediately from the computation \eqref{eq:in_bases}. 
\end{proof}

Since  $J=\ker(\varpi: \Z[\pi]\to \Z[F])$ is an ideal,
we deduce from \eqref{eq:Fox_to_double} and \eqref{eq:Ia_A} 
that the composition  $(\varpi \otimes \varpi)\double{-}{-}$
induces a linear map
\begin{equation} \label{eq:Theta}
\Theta: \Z[F]\otimes \AAlpha \longrightarrow \Z[F]\otimes \Z[F]
\end{equation}
which is equivalent to  the pairing
$\langle-,-\rangle:\Z[F] \times \AAlpha \to \Z[F]$.
Indeed, we have
\begin{eqnarray}
\label{eq:<>_to_Theta}
\forall x\in \Z[F], \forall a \in \AAlpha, \quad \Theta(x,a) &=& 
\overline{\langle x'', a\rangle'}\, x'  \otimes \langle x'', a\rangle''\\
\notag \hbox{ and, conversely, } \qquad \langle x,a \rangle &=& 
\varepsilon\left(\Theta(x,a)^\ell\right) \otimes \Theta(x,a)^r.
\end{eqnarray}

\begin{proposition}
 The map $\Theta$ has the following properties:
 \begin{equation}
\label{eq:left_Theta}
\forall x,y\!\in \Z[F], \, \forall a \in \AAlpha, \ 
\Theta(xy,a) =  \Theta(y,a)^\ell \otimes x \Theta(y,a)^r 
 + \Theta(x,a)^\ell y  \otimes \Theta(x,a)^r,
\end{equation} 
 \begin{equation}
 \label{eq:right_Theta}
 \forall f \in F, \ \forall x\in \Z[F],  \ \forall a \in \AAlpha, \quad
 \Theta\big(x,{}^fa\big)= f \Theta(x,a)^\ell\!  \otimes \Theta(x,a)^r f^{-1},
\end{equation}
\begin{equation}
\label{eq:inversion}
\forall x\in \Z[F], \, \forall a \in \AAlpha, \quad 
\Theta(\overline{x},a)=-\overline{\Theta(x,a)^r}\otimes \overline{\Theta(x,a)^\ell}.
\end{equation}
\end{proposition}

\begin{proof}
The identity \eqref{eq:left_Theta} is a direct application of \eqref{eq:inner}, while \eqref{eq:right_Theta} 
follows from \eqref{eq:<>_to_Theta} and
the $\Z[F]$-linearity of $\langle-,-\rangle$ in its second argument:
\begin{eqnarray*}
 \Theta\big(x,{}^fa\big)&=&
 \overline{\langle x'', a^{f^{-1}}\rangle'}x'  
 \otimes \langle x'', a^{f^{-1}}\rangle''\\
 &=&  \overline{\big(\langle x'', a\rangle f^{-1}\big)'}x'  
 \otimes \big(\langle x'', a\rangle f^{-1}\big)''
 \ = \ f \Theta(x,a)^\ell  \otimes \Theta(x,a)^r f^{-1}.
\end{eqnarray*}
Property \eqref{eq:inversion} 
follows from the identity
$\double{\overline{u}}{\overline{v}}= 
\overline{\double{u}{v}^r} \otimes \overline{\double{u}{v}^\ell}$ 
($u,v\in \Z[\pi]$),
which can be checked using \eqref{eq:double}.
\end{proof}

\subsection{The intersection operation $\Psi$} \label{subsec:Psi}

We now derive another operation $\Psi$ from the homotopy intersection form
$\eta:\Z[\pi]\times \Z[\pi] \to \Z[\pi]$.

Recall that $J=\ker(\varpi: \Z[\pi]\to \Z[F])$. 
The following lemma refines the isomorphism \eqref{eq:Quillen} 
in degree 1 to integral coefficients.

\begin{lemma}
There is an isomorphism  of abelian groups
$$
\gamma: \Z[F]\otimes \AAlpha \lto J/J^2
$$ 
defined by 
$\gamma(\varpi(x)\otimes [a])= [x\, (a-1)]$ for any $x\in \Z[\pi]$ and $ a\in \Alpha$.
\end{lemma}

\begin{proof}
It is easily verified that the map $\gamma$ is well-defined.
Let $B$ be a subgroup of~$\pi$ that maps isomorphically onto $F$ by $\varpi$.
Then, we can write any $j\in J$ uniquely as a finite sum
$
j= \sum_{b \in B} b \, a_b(j)
$
with $a_b(j)\in I_\Alpha$. Thus, we have a homomorphism
$$
\rho: J \lto \Z[F]\otimes (I_\Alpha/I_\Alpha^2), 
\ j \longmapsto  \sum_{b\in B}\varpi(b) \otimes [a_b(j)]
$$
and we easily check that $\rho(J^2)$ is trivial.
Clearly, the resulting homomorphism 
$
\rho: J/J^2 \to \Z[F]\otimes (I_\Alpha/I_\Alpha^2)  \simeq  \Z[F]\otimes \AAlpha 
$
is the inverse to $\gamma$.
\end{proof}

To define our last intersection operation $\Psi$,
recall from \eqref{eq:Ia_A} that $\eta(\Alpha,\Alpha)\subset J$.

\begin{proposition} \label{prop:Psi}
There is a unique $\Z$-bilinear map $\Psi$
that fits into the following commutative diagram:
\begin{equation} \label{eq:Psi}
 \xymatrix{
\Alpha \times \Alpha \ar[r]^-{\eta} \ar@{->>}[d]&  J \ar@{->>}[r]& J/J^2 \\
\AAlpha \times \AAlpha \ar[rr]^-\Psi &&
\Z[F]\otimes \AAlpha \ar[u]_-\gamma^-\simeq &  \\
}
\end{equation}
Furthermore, $\Psi$ has the following properties:
\begin{equation} \label{eq:Psi_zeta}
    \forall a \in \AAlpha, \quad \Psi(a,[\zeta])=1 \otimes a ,
    \quad \hbox{where $\zeta \in \Alpha$ is the homotopy class of $\partial \Sigma$}.
\end{equation}
\begin{equation} \label{eq:conj_left}
    \forall x\in F, \ \forall a,b\in \AAlpha, \quad 
    \Psi({}^xa,b) =  x \Psi(a,b)^\ell \otimes \Psi(a,b)^r- 
    \Theta(x,b)^r \otimes {}^{\Theta(x,b)^\ell} a.
\end{equation} 
\begin{equation} \label{eq:conj_right}
     \forall y\in F, \ \forall a,b\in \AAlpha, \quad 
    \Psi(a,{}^yb) =  \Psi(a,b)^\ell  y^{-1}\otimes {}^y ( \Psi(a,b)^r) 
    - \overline{\langle y,a\rangle}\otimes {}^y b .
\end{equation} 
\begin{equation} \label{eq:kind_of_sym}
\forall a,b\in \AAlpha, \qquad
\Psi(b,a)= \overline{\Psi(a,b)^{\ell'}}
\otimes\, {}^{\Psi(a,b)^{\ell''}\!} \big(\Psi(a,b)^r\big).
\end{equation}
\begin{eqnarray} 
\notag  \qquad \quad \forall a,b,c\in \AAlpha, &&
\Psi(b,a)^\ell \otimes \Psi(\Psi(b,a)^r,c)^\ell
\otimes  \Psi(\Psi(b,a)^r,c)^r \\
\label{eq:pseudo-Jacobi}    &=& \Psi(b,\Psi(a,c)^r)^\ell\, \overline{\Psi(a,c)^{\ell '}} 
\otimes \Psi(a,c)^{\ell''} \otimes \Psi(b,\Psi(a,c)^r)^r \\*
\notag &&-\Theta(\Psi(b,c)^\ell,a)^r \otimes \Theta(\Psi(b,c)^\ell,a)^\ell 
\otimes \Psi(b,c)^r  \\*
\notag && - \overline{\langle \Psi(a,c)^{\ell '}, b \rangle} 
\otimes  \Psi(a,c)^{\ell ''} \otimes \Psi(a,c)^r.
\end{eqnarray}
\begin{eqnarray} 
\notag  \qquad \quad \forall a,b,c\in \AAlpha, &&
\Psi(a,c)^\ell \otimes \Psi(b,\Psi(a,c)^r)^\ell \otimes \Psi(b,\Psi(a,c)^r)^r \\
\label{eq:pseudo-Jacobi_bis}  &=& \Psi(\Psi(b,a)^r,c)^{\ell '} \otimes \Psi(b,a)^\ell\,
\Psi(\Psi(b,a)^r,c)^{\ell ''} \otimes \Psi(\Psi(b,a)^r,c)^{r}\\
 \notag && + \overline{\langle \Psi(b,c)^{\ell '},a\rangle}\,
\Psi(b,c)^{\ell''} \otimes \Psi(b,c)^{\ell'''} \otimes \Psi(b,c)^{r}  \\
\notag && + \Psi(a,c)^{\ell'} \otimes 
\overline{\langle \Psi(a,c)^{\ell''} ,b \rangle}\, \Psi(a,c)^{\ell'''} 
\otimes \Psi(a,c)^{r}.
\end{eqnarray}
\end{proposition}

\begin{proof}
Let $q:J \to J/J^2$ be the projection.
For any $a,b,c\in \Alpha$, we have
\begin{gather*}
\eta(ab,c)= \eta(b,c)+(a-1)\eta(b,c)+ \eta(a,c)\equiv  \eta(b,c) + \eta(a,c) \mod J^2,\\
\eta(a,bc)= \eta(a,b)+ \eta(a,b)(c-1)+ \eta(a,c) \equiv\eta(a,b) + \eta(a,c) \mod J^2,
\end{gather*}
which shows that the map $q\eta: \Alpha \times \Alpha \to J/J^2$ 
factors through the canonical projection 
$\Alpha \times \Alpha  \to \AAlpha \times \AAlpha$ to give a $\Z$-bilinear pairing.
This shows the existence (and uniqueness) of the pairing $\Psi$.

We have $\eta(a,\zeta)=a-1$ for any $a\in \Alpha$,
which proves \eqref{eq:Psi_zeta}.
Let $x\in \pi$ and $a,b\in \Alpha$;
then \eqref{eq:conj_left} is proved as follows:
\begin{eqnarray*}
\eta({}^xa,b)\ = \ \eta(xax^{-1},b)
&= &xa\eta(x^{-1},b)+x\eta(a,b)+\eta(x,b)\\
&=& (1-xax^{-1}) \eta(x,b)+x\eta(a,b) \\
&=& -\eta(x,b)' \,\overline{ \eta(x,b)''}x (a-1) x^{-1}\eta(x,b)'''+x\eta(a,b) .
\end{eqnarray*}
Let $y\in \pi$ and $a,b\in \Alpha$; 
then \eqref{eq:conj_right} is proved as follows:
\begin{eqnarray*}
\eta(a,{}^y b)\ = \ \eta(a,yby^{-1})&=&\eta(a,yb) y^{-1} + \eta(a,y^{-1})\\
&=&  \eta(a,y) (b-1)y^{-1} + \eta(a,b) y^{-1}  \\
&\stackrel{\eqref{eq:ask}}{\equiv} &  - \overline{\eta(y,a)}y (b-1)y^{-1} + \eta(a,b) y^{-1} \mod J^2.  
\end{eqnarray*}
Identity \eqref{eq:kind_of_sym} follows from
$$
q\eta(b,a) \stackrel{\eqref{eq:ask}}{=}
-\overline{q \eta (a,b)} \mod J^2\quad (a,b\in\Alpha)
$$
and the following observation:
through the isomorphism $\gamma$,
the   involution of $J/J^2$ induced by the antipode of $\Z[\pi]$ 
corresponds to the involution $(x \otimes a \mapsto - \overline{x'} \otimes\, {}^{x''}\!a)$ 
of $\Z[F]\otimes \AAlpha$.

It now remains to prove \eqref{eq:pseudo-Jacobi} and \eqref{eq:pseudo-Jacobi_bis}.
By applying the map $\varepsilon\otimes \id_{\Z[\pi]}\otimes \id_{\Z[\pi]}$ 
to \eqref{eq:quasi-Jacobi},
we  obtain the following identity in $\Z[\pi]^{\otimes 2}$:
\begin{eqnarray} \label{eq:reduced_quasi-Jacobi}
\qquad \  \forall a,b,c \in \Z[\pi], 
&& \eta\big({a},{\double{b}{c}^\ell}\big) \otimes \double{b}{c}^r
-    \double{b}{\eta({a},{c})}\\
\notag && -   \double{a}{b}^r
\otimes \eta\big(\double{a}{b}^\ell,{c}\big)  \ = \     
1 \otimes b\, \eta({a},{c}) -   b\otimes \eta({a},{c}).
\end{eqnarray}
Since we have  
$J= \Z[\pi]\, I_{\A}$ and  $\eta(\Alpha,\Alpha)\subset J$, we have $\eta(J,J)\subset J$.
Thus $q\eta$ induces a $\Z$-bilinear map
$\eta_J: (J/J^2)\times (J/J^2) \to J/J^2$.
It also follows that
$\double{J}{J} \subset \Z[\pi] \otimes J + J \otimes \Z[\pi]$.
Therefore, the composition of $\double{-}{-}$ with
$\varpi \otimes q$ induces a $\Z$-linear  map
$$
\double{-}{-}_J: J/J^2 \otimes J/J^2 \lto \Z[F] \otimes J/J^2.
$$

We now take $a,b,c$ in $\Alpha$ and
let $\underline a :=[a-1],\underline b := [b-1],\underline c := [c-1]$ in $J/J^2$.
Besides, the corresponding
elements $[a],[b],[c]$ in $\AAlpha$ are simply denoted by $a,b,c$, respectively.
Using the ``skew-symmetry'' \eqref{eq:ask} and its consequence \eqref{eq:ask_bis},
 we apply 
$\varpi \otimes q$ to \eqref{eq:reduced_quasi-Jacobi}
to get the following identity in $\Z[F] \otimes J/J^2$:
$$
- \overline{\left\langle \double{\underline b}{\underline c}^{\ell'}_J, 
 a \right\rangle} \double{\underline b}{\underline c}^{\ell''}_J  
 \otimes \double{\underline b}{\underline c}^r_J
 -   \double{\underline b}{\eta_J(\underline{a},\underline{c})}_J
 + \double{\underline b}{\underline a}^\ell _J 
 \otimes \eta_J(\double{\underline b}{\underline a}^r_J, \underline c)
\ = \ 0.
$$
Note that the first term in this identity is
\begin{eqnarray*}
\overline{\left\langle \double{\underline b}{\underline c}^{\ell'}_J, 
 a \right\rangle} \double{\underline b}{\underline c}^{\ell''}_J  
 \otimes \double{\underline b}{\underline c}^r_J 
 &=& \overline{\left\langle \overline{\double{\underline b}{\underline c}^{\ell''}_J }  \double{\underline b}{\underline c}^{\ell'}_J, 
 a \right\rangle}   \otimes \double{\underline b}{\underline c}^r_J \\
 && - \varepsilon(\double{\underline b}{\underline c}^{\ell'}_J )\overline{ 
 \left\langle \overline{\double{\underline b}{\underline c}^{\ell''}_J }  , 
 a \right\rangle}   \otimes \double{\underline b}{\underline c}^r_J  \\
 &=& -\overline{ 
 \left\langle \overline{\double{\underline b}{\underline c}^\ell_J }  , 
 a \right\rangle}   \otimes \double{\underline b}{\underline c}^r_J.
\end{eqnarray*}
Therefore, we get
\begin{equation} \label{eq:ouf}
 \overline{\left\langle \overline{\double{\underline b}{\underline  c}^\ell_J }  
 ,  a \right\rangle}   \otimes \double{\underline  b}{\underline c}^r_J
 -   \double{\underline b}{\eta_J(\underline{a},\underline{c})}_J
 + \double{\underline b}{\underline a}^\ell _J 
 \otimes \eta_J(\double{\underline b}{\underline a}^r_J, \underline c)
\ = \ 0 .   
\end{equation}
It also follows from the definitions that
\begin{equation} \label{eq:obs1}
\forall u,v\in \Alpha, \quad
\double{[u-1]}{[v-1]}_J= 
\overline{\Psi(u,v)^{\ell '}} 
\otimes \gamma\big( \Psi(u,v)^{\ell ''}  \otimes \Psi(u,v)^r\big),
\end{equation}
which, using \eqref{eq:outer} and \eqref{eq:ask_bis}, 
implies the following more general formula:
\begin{eqnarray} \label{eq:obs2}
\forall u,v\in \Alpha, \forall x\in \pi, \quad
&&\double{[u-1]}{[x(v-1)]}_J \\
\notag &=& \varpi(x) \overline{\Psi(u,v)^{\ell '}} 
\otimes \gamma\big( \Psi(u,v)^{\ell ''}  \otimes \Psi(u,v)^r\big), \\
\notag && - \Theta(\varpi(x),u)^r \otimes 
\gamma \left( \Theta(\varpi(x),u )^\ell\otimes  (v-1) \right)
\end{eqnarray}
Then, using \eqref{eq:obs1} and \eqref{eq:obs2}, we can rewrite the  identity \eqref{eq:ouf} as 
follows:
\begin{eqnarray*}
\overline{\left\langle {\Psi(b,c)^{\ell'} }  , 
 a \right\rangle} \otimes  \Psi({b},{c})^{\ell''}   \otimes \Psi({b},{c})^r &&\\
-  \Psi(a,c)^\ell\,  \overline{\Psi(b,\Psi(a,c)^r)^{\ell'}} \otimes \Psi(b,\Psi(a,c)^r)^{\ell''} \otimes \Psi(b,\Psi(a,c)^{r})^r && \\
+ \Theta(\Psi(a,c)^\ell,b)^r \otimes  \Theta(\Psi(a,c)^\ell,b)^l \otimes \Psi(a,c)^r \\
+ \overline{\Psi(b,a)^{\ell'}} \otimes \Psi(b,a)^{\ell''} \Psi(\Psi(b,a)^r,c)^\ell 
\otimes \Psi (\Psi(b,a)^r,c)^r &=&0.
\end{eqnarray*}
From this identity, we can derive \eqref{eq:pseudo-Jacobi} 
and  \eqref{eq:pseudo-Jacobi_bis} by applying the
automorphisms
$(u \otimes v \otimes w \mapsto \overline{u'} \otimes u''v \otimes w)$
and $(u \otimes v \otimes w \mapsto uv' \otimes v'' \otimes w)$  of $\Z[F]^{\ot2} \otimes \AAlpha$, respectively.
\end{proof}

\begin{remark}
Let $(\alpha,\beta)$ be a basis of $\pi$ of  type \eqref{eq:(a,b)}, and set $a_i=[\alpha_i] \in \AAlpha$.
It follows from \eqref{eq:E} that 
\begin{equation} \label{eq:values}
    \Psi(a_i,a_j)= \delta_{ij}  \otimes a_i 
\quad \hbox{for any $i,j\in\{1,\dots,g\}$}.
\end{equation}
Since $\AAlpha$ is freely generated by $(a_1,\dots,a_g)$ as a $\Z[F]$-module,
one can also define the map
$\Psi: \AAlpha \times \AAlpha \to \Z[F] \otimes \AAlpha$ 
 as the unique $\Z$-bilinear pairing satisfying 
\eqref{eq:conj_left}, \eqref{eq:conj_right} and \eqref{eq:values}.
\end{remark}

\end{document}